\documentclass[11pt,a4paper]{article}

\usepackage{amsfonts}
\usepackage{amsthm}
\usepackage{amsmath}
\usepackage{amssymb}
\usepackage{amscd}
\usepackage{mathrsfs}
\usepackage[mathscr]{eucal}
\usepackage{graphicx}
\usepackage{epstopdf} 
\usepackage{pict2e}
\usepackage{epic}
\usepackage{xcolor}
\usepackage{float}
\usepackage{mathtools}
\usepackage{centernot}

\usepackage{cite}
\usepackage{hyperref}
\hypersetup{colorlinks=true, urlcolor= black, linkcolor=black, citecolor=black}

\usepackage[utf8]{inputenc}
\usepackage[T1]{fontenc}
\usepackage{lmodern}
\usepackage{dsfont}
\usepackage{indentfirst}
\usepackage[margin=3cm]{geometry}

\usepackage{centernot}
\usepackage{bbm}

\numberwithin{equation}{section}
\newtheorem{thm}{Theorem}[section]
\newtheorem*{thm*}{Theorem}
\newtheorem*{prop*}{Proposition}
\newtheorem*{lem*}{Lemma}
\newtheorem{ass}[thm]{Assumption}
\newtheorem{lem}[thm]{Lemma}
\newtheorem{prop}[thm]{Proposition}
\newtheorem{cor}[thm]{Corollary}
\newtheorem{defn}[thm]{Definition}
\newtheorem{rem}[thm]{Remark}
\newtheorem*{op1}{Open Problem 1}
\newtheorem*{op2}{Open Problem 2}
\newtheorem*{op3}{Open Problem 3}

\newcommand{\ZZ}{\mathbb{Z}}
\newcommand{\RR}{\mathbb{R}}
\newcommand{\NN}{\mathbb{N}}
\newcommand{\PP}{\mathbb{P}}
\newcommand{\EE}{\mathbb{E}}
\newcommand{\TT}{\mathbb{T}}
\newcommand{\CC}{\mathbb{C}}

\DeclareFontFamily{OT1}{pzc}{}
\DeclareFontShape{OT1}{pzc}{m}{it}{<-> s * [1.10] pzcmi7t}{}
\DeclareMathAlphabet{\mathpzc}{OT1}{pzc}{m}{it}

\newcommand{\cG}{\mathcal G}
\newcommand{\ct}{\mathpzc{t}}
\newcommand{\cM}{\mathcal M}
\newcommand{\cB}{\mathcal B}
\newcommand{\cT}{\mathcal T}
\newcommand{\cH}{\mathcal H}
\newcommand{\cF}{\mathcal F}
\newcommand{\cP}{\mathcal P}
\newcommand{\cC}{\mathcal C}
\newcommand{\cA}{\mathcal A}
\newcommand{\cE}{\mathcal E}
\newcommand{\cI}{\mathcal I}
\newcommand{\cJ}{\mathcal J}
\newcommand{\cL}{\mathcal L}
\newcommand{\cS}{\mathcal S}

\newcommand{\bfP}{\mathbf{P}}

\def\fg{\mathfrak{g}}
\def\sfP{\mathsf{P}}

\newcommand\n{\mathbf{n}}
\newcommand\m{\mathbf{m}}
\newcommand\sn{\partial\mathbf{n}}

\newcommand\ising{\textup{Ising}}

\renewcommand{\phi}{\varphi}
\newcommand{\PLUS}{ \mathfrak{p}}
\newcommand{\MINUS}{\mathfrak{m}}

\numberwithin{equation}{section}

\begin{document}

\title{Random tangled currents for $\varphi^4$: translation invariant Gibbs measures and continuity of the phase transition}

\author{Trishen S. Gunaratnam\footnotemark[1]\footnote{TIFR Mumbai, \url{trishen@math.tifr.res.in}} \footnotemark[2]\footnote{ICTS-TIFR Bengaluru, \url{trishen.gunaratnam@icts.res.in}}\:, Christoforos Panagiotis\footnotemark[3]\footnote{University of Bath, \url{cp2324@bath.ac.uk}}\:, \\ Romain Panis\footnotemark[4]\footnote{Institut Camille Jordan (Lyon 1), \url{panis@math.univ-lyon1.fr}, \url{severo@math.univ-lyon1.fr}}\:, Franco Severo\footnotemark[4]}

\maketitle
\begin{abstract}
    We prove that the set of automorphism invariant Gibbs measures for the $\phi^4$ model on graphs of polynomial growth has at most two extremal measures at all values of $\beta$. We also  give a sufficient condition to ensure that the set of all Gibbs measures is a singleton. As an application, we show that the spontaneous magnetisation of the nearest-neighbour $\phi^4$ model on $\ZZ^d$ vanishes at criticality for $d\geq 3$. The analogous results were established for the Ising model in the seminal works of Aizenman, Duminil-Copin, and Sidoravicius (Comm.\ Math.\ Phys., 2015), and Raoufi (Ann.\ Prob., 2020) using the so-called random current representation introduced by Aizenman (Comm.\ Math.\ Phys., 1982).  One of the main contributions of this paper is the development of a corresponding geometric representation for the $\phi^4$ model called the \textit{random tangled current} representation.
\end{abstract}

\maketitle

\section{Introduction}

\subsection{Motivation}

The $\phi^4$ model is a statistical mechanics model of ferromagnetism with real-valued spins attached to each vertex of a graph whose values are confined according to a quartic single-site potential. Mathematically, this is modelled by a probability measure on real-valued spin configurations on a graph $G=(V,E)$ formally given by 
\begin{equation*}
\langle F(\phi) \rangle \,=\frac{1}{Z}
\,  \int F(\phi) \, \exp \Big( \beta\sum_{\lbrace x,y\rbrace \in E}  \phi_x \phi_y - \sum_{x \in V} \left(g\phi_x^4 + a \phi_x^2\right) \Big) \textup{d}\phi
\end{equation*} 
where $g > 0$ and $a \in \RR$ are coupling constants, $\beta > 0$ is the inverse temperature, and $Z$ is a normalisation constant.
This model naturally interpolates between two famous models of statistical physics, namely the Gaussian free field and the Ising model. The Gaussian free field is defined by setting $g=0$ and taking appropriate $a$. Formally, in the limit $g \rightarrow 0$ when $a > 0$, the $\phi^4$ model approximates a Gaussian free field of appropriate covariance. On the other hand, the Ising model is defined by integrating instead with the respect to the Bernoulli measure on $\{\pm 1\}$. In this case, the limit $g \rightarrow \infty$ for $a < 0$ formally corresponds to the Ising model. This can be seen easily if one takes a joint scaling limit when $a=-2g$:
 \begin{equation}\label{eq: ising limit}
\tfrac 12 ( \delta_{-1} +  \delta_1 )
=
\lim_{g \rightarrow \infty} \frac{1}{Z_g} e^{-g(\phi^2-1)^2} \textup{d}\phi. 
 \end{equation}

The model has its origins in Euclidean quantum field theory, where its analogue on $\RR^d$ is amongst the simplest candidates for a non-trivial (i.e.\ non-Gaussian) measure. These measures in dimensions $d\geq 2$ are a priori ill-defined due to the notorious problem of ultraviolet divergences \cite{glimm2012quantum}. Two major breakthroughs in the constructive field theory programme of the 1960s-1970s were the solutions to these problems in dimensions $2$ and $3$ due to Nelson \cite{N66}, and Glimm and Jaffe \cite{GJ73}. In dimensions $d\geq 4$ a non-Gaussian measure cannot be obtained due to the triviality results that we mention below.
 
It was later realised that the $\phi^4$ model on the lattice $\ZZ^d$ has importance in statistical physics. 
 Its viewpoint as a generalised Ising model was first rigorously explored in the landmark paper by Guerra, Rosen, and Simon \cite{GRS75}. 
 The approximation to Ising was further exploited by Glimm, Jaffe, and Spencer \cite{GJS75}, where they adapt the celebrated Peierls' argument \cite{P36} to prove that the phase transition occurs at a finite non-zero temperature for $\phi^4$ in dimensions $d\geq 2$. In particular, there is a critical inverse temperature $\beta_c$, depending on the fixed values of $g,a$, such that: for $\beta < \beta_c$, the two point correlation tends to $0$; and for $\beta > \beta_c$, the two point correlation does not tend to $0$ and there is long-range order.

The relationship between $\phi^4$ and the Ising model is predicted to be very deep: they are a fundamental example of two models that are believed to belong to the same universality class. Indeed, it is predicted from renormalisation group heuristics (see \cite{G70, K93} or the recent book \cite{BBS19}) that at their respective critical points many properties of these two models, such as critical exponents, exactly coincide. A manifestation of this was discovered by Griffiths and Simon in \cite{GS}, where they show that the $\phi^4$ model arises as a certain near-critical scaling limit of mean-field Ising models. Indeed, this allows one to transfer many special properties of the Ising model, such as correlation inequalities, to $\phi^4$. This makes it one of the simplest non-integrable models to study, even in spite of the unbounded spins. 

More direct results towards universality were obtained in dimensions $d\geq 5$ by Aizenman \cite{A82} and Fr\"ohlich \cite{F82}, where triviality of their scaling limits, both in the infinite volume and continuum, were established. The case of $d=4$ was settled in the beautiful paper by Aizenman and Duminil-Copin \cite{ADC}, where stochastic geometric representations for the model, as in this paper, play a central role. Establishing rigorously universality results in dimensions $2$ and $3$ remains a difficult challenge. We refer to the recent review \cite{A21} concerning $\phi^4$ for further discussion.

In this article, we are interested in understanding the structure of Gibbs measures for $\phi^4$, and in particular identifying the behaviour at the critical point. To do this, we build on ideas that have led to recent breakthroughs towards these questions in the case of the Ising model. In particular, the continuity of the magnetisation at the critical point in $d \geq 3$ was established by Aizenman, Duminil-Copin, and Sidoravicious \cite{ADCS} using the so-called random current representation, first introduced by Aizenman in \cite{A82}. These techniques were further exploited by Raoufi in \cite{RAO20} to show that the set of translation invariant Gibbs measures on any amenable graph at any temperature is given by convex combinations of two natural extremal measures. This was previously known for temperatures not equal to the critical point on $\ZZ^d$, see the results of Bodineau \cite{B06}. First, we rigorously define the quantities of interest and state our main result on the structure of Gibbs measures. Then, we state its consequences for continuity of the phase transition (as well as a precise definition of what this means).

\subsection{Definitions and structure theorem} 
Let $G=(V,E)$ be a connected, infinite, and locally finite graph. Let $J=(J_{x,y})_{x,y\in V}$ be a collection of ferromagnetic interactions, i.e.\ such that $J_{x,y}\geq 0$ for all $x,y\in V$. For $g>0$ and $a\in \mathbb R$, let $\rho_{g,a}$ be the single-site measure on $\RR$ defined by
\begin{equation}
    \textup{d}\rho_{g,a}(t)
    =
    \frac{1}{z_{g,a}}e^{-gt^4-at^2}\textup{d}t,
\end{equation}
where $z_{g,a}=\int_\RR e^{-gt^4-at^2}\textup{d}t$.
Let $\Lambda\subset V$ be finite. Let $h=(h_x)_{x\in V}, \eta = (\eta_x)_{x \in V}\in \mathbb R^V$ be the external magnetic field and boundary condition, respectively. We define the Hamiltonian
\begin{equation}
    H^\eta_{\Lambda,h,J}(\varphi)
    :=
    -\sum_{x,y\in \Lambda}J_{x,y}\varphi_x\varphi_y-\sum_{\substack{x\in \Lambda\\y\in V\setminus \Lambda}}J_{x,y}\varphi_x\eta_y-\sum_{x\in \Lambda}h_x\varphi_x,\qquad \forall \varphi\in \RR^\Lambda
\end{equation}
whenever
$
\sum_{\substack{y\in V\setminus \Lambda}}J_{x,y} |\eta_y|
<
\infty ~~\forall x \in \Lambda$.

The $\varphi^4$ model on $\Lambda$ with coupling constants $(J,g,a)$, boundary condition $\eta$, inverse temperature $\beta>0$, and external magnetic field $h$, is the measure  $\langle \cdot\rangle^{\eta}_{\Lambda,\beta,h,J,g,a}$ defined via the expectation values
\begin{equation}\label{eq: def}
    \langle f\rangle^\eta_{\Lambda,\beta,h,J,g,a}
    :=
    \frac{1}{Z^\eta_{\Lambda,\beta,h,J,g,a}}\int_{\varphi\in \RR^\Lambda}f(\varphi)e^{-\beta H^\eta_{\Lambda,h J}(\varphi)} \textup{d}\rho_{\Lambda,g,a}(\phi)
\end{equation}
where
$\textup{d}\rho_{\Lambda,g,a}(\phi)=\prod_{x \in \Lambda} \textup{d}\rho_{g,a}(\phi_x)$,  $f:\mathbb R^\Lambda\rightarrow \RR$ is a bounded measurable function and $Z^\eta_{\Lambda,\beta,h,J,g,a}$ is the partition function, i.e.\ the normalisation constant such that $\langle 1\rangle^\eta_{\Lambda,\beta,h,J,g,a}=1$. When clear from context we drop $(J,g,a)$ from the notation.

In this article, we impose a restriction on the growth of the graph. We fix $o\in G$ a distinguished origin and let $B_k$ denote the ball of radius $k$ centred at $o$ with respect to the graph distance $d_G$. Let $G$ be a \textit{vertex-transitive graph} of \emph{polynomial} growth, i.e.\ a graph for which there exists $d\geq 1$ (called the dimension of $G$) and $c,C>0$ such that, for all $k\geq 1$,
we have $ck^d\leq |B_k|\leq Ck^d$.
As such, $G$ is amenable: 
\begin{equation}
\inf_{A \subset V, |A| < \infty} \frac{|\partial A|}{|A|}
=
0
\end{equation}
where $\partial A = \{ x \in V\setminus A : \exists \, y \in A \text{ with } \lbrace x,y\rbrace \in E\}$ and $|\cdot|$ denotes cardinality. We fix $\Gamma$ a vertex-transitive subgroup of the automorphism group of $G$, denoted Aut$(G)$. 

A set of interactions $J=(J_{x,y})_{x,y \in V}$ is \textit{admissible} with respect to $G$ and $\Gamma$ if:
\begin{enumerate}
    \item[$\textbf{C1}$] \textbf{(Ferromagnetic)} $J_{x,y} \geq 0$ for all $x,y \in V$; 
    \item[$\textbf{C2}$] \textbf{($\Gamma$-invariance)} $J_{\gamma(x) ,\gamma(y)} = J_{x,y}$ for all $\gamma \in \Gamma$;
    \item[$\textbf{C3}$] \textbf{(Irreducibility)} for every $x,y\in V$, there exists a sequence $x_0=x,\ldots,x_k=y$ such that $J_{x_0,x_1}\ldots J_{x_{k-1},x_k}>0$;
    \item[$\textbf{C4}$] \textbf{(Integrability)} $\exists\, C>0$ and $\epsilon > 0$ such that $J_{x,y} \leq Cd_G(x,y)^{-(d+\epsilon)}$ for all $x\neq y \in V$.
\end{enumerate}
We remark that {\bf (C4)} indeed implies that $\sum_{x \in V} J_{o,x} < \infty$, see the calculation in Lemma \ref{lem: psi sum}.

We say that a Borel measure $\nu$ on $\RR^V$ is \textit{tempered} if:
\begin{itemize}
\item for every $\Lambda \subset V$ finite, $\nu|_{\Lambda}$ is absolutely continuous with respect to Lebesgue measure on $\RR^\Lambda$;
\item $\nu$ is supported on the event
\begin{equation}
R
:=
\Big\{ (\varphi_x)_{x\in V} : \exists C\in(0,\infty), \sum_{x \in B_k} \varphi_x^2 \leq C|B_k| \text{ for all } k \in \NN \Big\}.
\end{equation}
\end{itemize}

\begin{defn}[Gibbs measures]
We say that a Borel probability measure $\nu$ on $\mathbb R^{V}$ is a Gibbs measure at $(\beta,h)$ if it is tempered and for every finite $\Lambda\subset V$ and $f\in \RR^{\Lambda}$ bounded and measurable, the DLR equation
\begin{equation}
    \nu(f)
    =
    \int_{\eta\in \RR^V}\langle f\rangle^\eta_{\Lambda,\beta,h}\textup{d}\nu(\eta)
\end{equation}
holds. In particular, we assume that $\nu$ is almost surely supported on $\eta \in \RR^V$ such that $\langle \cdot\rangle^\eta_{\Lambda,\beta,h}$ is well-defined.
\end{defn}

We denote by $\mathcal{G}(\beta,h)$ the set of Gibbs measures at $(\beta,h)$. We say that a measure $\nu\in \cG(\beta,h)$ is \textit{$\Gamma$-invariant} if for all $\gamma\in \Gamma$, the pushforward measure $\gamma_\ast\nu$ is equal to $\nu$. We denote by $\cG_{\Gamma}(\beta,h)\subset \cG(\beta,h)$ the set of $\Gamma$-invariant Gibbs measures. The set $\cG_{\Gamma}(\beta,h)$ has a structure of a Choquet simplex, i.e.\ admits a \textit{unique} extremal ergodic decomposition (see \cite[Theorem 5.8]{R70}). 

Our main result is a structure theorem for $\cG_{\Gamma}(\beta,0)=\cG_{\Gamma}(\beta)$. It states that there are at most two extremal measures $\langle\cdot\rangle_\beta^+$ and $\langle \cdot\rangle_\beta^-$ which we define as follows. First, for any $h\in \RR$, let $\langle\cdot\rangle_{\beta,h}^0=\lim_{\Lambda\uparrow V}\langle \cdot\rangle_{\Lambda,\beta,h}^0$, which exists by tightness (see Proposition \ref{appendix prop: tightness}) and a standard monotonicity argument using Griffiths' inequality (see Proposition \ref{prop: monotonicity in vol}). Define $\langle\cdot\rangle_\beta^+$ and $\langle \cdot\rangle_\beta^-$ via the following monotonic weak limits (again by Griffiths' inequality, see Remark \ref{rem: monotoncity griffiths}):
\begin{equation}
    \langle\cdot\rangle_\beta^+=\lim_{h\downarrow 0}\langle \cdot\rangle_{\beta,h}^0,\qquad \langle\cdot\rangle_\beta^-=\lim_{h\uparrow 0}\langle \cdot\rangle_{\beta,h}^0.
\end{equation}
The measures $\langle\cdot\rangle_\beta^+$ and $\langle \cdot\rangle_\beta^-$ are ergodic (in fact mixing) and extremal--- see Corollary \ref{cor: extremal and ergodic}.

\begin{thm}\label{thm:main}
Let $G$ be a connected, vertex-transitive graph of polynomial growth and $\Gamma\subset \text{Aut}(G)$ a vertex-transitive subgroup. Let $J$ be an admissible set of coupling constants with respect to $G$ and $\Gamma$, and fix $g > 0$ and $a \in \RR$. Then, for any $\beta>0$, 
\begin{equation} \label{eq: convex hull of pm}
    \cG_\Gamma(\beta)
    =
    \left\lbrace \lambda \langle \cdot \rangle^+_\beta +(1-\lambda)\langle \cdot \rangle^-_\beta: \lambda \in [0,1]\right\rbrace.
\end{equation}
Moreover, if
\begin{equation}\label{eq: condition for continuity}
\liminf_{d_G(o,x)\rightarrow\infty} \langle \phi_o \phi_x \rangle_{\beta}^0 = 0,
\end{equation}
then there exists only one Gibbs measure, i.e.\
\begin{equation}
    |\cG(\beta)|=|\cG_\Gamma(\beta)|=1.
\end{equation}
\end{thm}

To the best of our knowledge, Theorem \ref{thm:main} is the first result that characterises the set of translation invariant Gibbs measures for $\phi^4$ at any $\beta$. It has natural consequences for the continuity of the phase transition, i.e.\ condition \eqref{eq: condition for continuity} can be verified at criticality in certain cases; these are explored in the next subsection. We believe that certain ideas in the proof may be useful to study the universality between $\phi^4$ and Ising, and perhaps of use in studying $P(\phi)$ models (i.e.\ when the quartic polynomial in the single-site measure is replaced by an arbitrary even polynomial $P$ of positive leading coefficient). We also remark that this seems to be the first instance where the $\phi^4$ model is studied on general graphs beyond $\ZZ^d$ and the techniques developed may be useful to study universality with respect to the underlying choice of graph, i.e.\ microscopic structure. The restriction to (amenable) graphs of polynomial growth is for technical reasons; our proof of the construction of general Gibbs measures works under certain geometric conditions, but we cannot give an example of such an amenable graph which exhibits superpolynomial growth. We refer to Appendix \ref{app: LP} for further discussion. However, the main techniques--- in particular all the stochastic geometric tools developed--- that we employ should extend to the case of a general amenable graph.

The starting point to prove Theorem \ref{thm:main} is a characterisation of the structure of $\Gamma$-invariant Gibbs measures for $\varphi^4$--- see Proposition \ref{Lebowitz}. This is similar in spirit to the Lebowitz characterisation for the Ising model \cite{L1977}, but there are significant complications that arise due to the presence of unbounded spins. The importance of this characterisation is that it reduces the problem to showing that the plus and free measures, i.e.\ $\langle \cdot \rangle_{\beta}^+$ and $\langle \cdot\rangle_{\beta,h=0}^0$, coincide on two-point correlations. Thus, we need to control differences of two-point correlations. For the case of the Ising model \cite{ADCS,RAO20} this is enabled by the celebrated random current representation, and in particular by the switching lemma \cite[Lemma 3.2]{A82} which allows one to give a probabilistic interpretation of these quantities.

It is well-known that the $\phi^4$ model on a \textit{finite} graph $\Lambda$ is approximated by an Ising model on $\Lambda \times K_N$ with special coupling constants, where $K_N$ is the complete graph on $N$ vertices. This is called the Griffiths--Simon approximation \cite{GS}. One natural question, in light of the above, is why the naive approach of proving the structure theorem for Gibbs measures directly via the Griffiths--Simon approximation of $\phi^4$ fails. This approach would go by proving a corresponding statement for the Ising model on $\Lambda\times K_N$ and then taking the limit $\Lambda\uparrow G$ and $N\to\infty$ appropriately. Since the proof in the case of the Ising model heavily relies on ergodicity, it only gives information about the model in infinite volume without any quantitative control in finite volume. Therefore, this approach would require a proof of the convergence of the Griffiths--Simon model on $G\times K_N$ to the $\phi^4$ model on $G$. However, to the best of our knowledge, such a convergence result in the infinite volume is not known. Indeed, the notion of single-site measure, which is crucial to the argument of \cite{GS}, disappears in the infinite volume limit.

The core of the paper is the development of a new geometric representation of spin correlations for $\phi^4$, based on the random current representation for the Ising model, that we call a \textit{random tangled current representation}. A current $\n$ is a positive integer valued function on the $J$-edges of the graph, meaning the pairs $\lbrace x,y\rbrace\subset V$ such that $J_{x,y}>0$. A tangled current is a pair $(\n,\ct)$ obtained as follows: first, the current $\n$ is lifted to its natural multigraph $\mathcal{N}$, where each $J$-edge $\lbrace x,y\rbrace$ is replaced by $\n_{x,y}$ edges. Then, each vertex $x$ is replaced by an ordered block of $\Delta\n(x)$ vertices, where $\Delta\n(x)$ is the degree of $x$ in $\mathcal{N}$ and each of the $\Delta\n(x)$ edges is incident to a distinct vertex in the block. Finally, a choice of an even partition at each vertex, $\ct$, naturally tangles the currents at each block. The key technical tool is the switching principle of Theorem \ref{thm: switching lemma}, which, in particular, allows us to express ratios and differences of correlation functions in terms of connectivity properties of random tangled currents. With this switching principle, the proof of Theorem \ref{thm:main} broadly follows ideas developed in \cite{ADCS} and \cite{RAO20}. However, there are severe obstacles that require a fine knowledge of the behaviour of the single-site tangles, for example that any partition occurs with positive probability. There are also additional essential difficulties caused by the presence of unbounded spin. 

The construction of the random tangled currents and the proof of the switching principle starts from the Griffiths--Simon approximation. Indeed, since the model in the prelimit of $N \rightarrow \infty$ is an Ising model, we may use the standard switching lemma. Capturing the stable quantities as $N \rightarrow \infty$ requires isolating the geometric structures that persist in the scaling limit, which then leads us to random tangled currents. We stress, however, that random tangled currents and the switching principle are stated on the level of $\phi^4$ and make no reference to the prelimit; it is only their construction that does. It is an interesting question to prove such a switching principle directly for $\phi^4$, without any reference to the Griffiths--Simon approximation.

Our representation complements the well-known Brydges--Fröhlich--Spencer random walk representation (or BFS representation) of the $\varphi^4$ model \cite{brydges1982random}. This path expansion of the model can be used to derive many classical correlation inequalities for the model (see \cite{fernandez2013random} and references therein) and also plays a pivotal role in Fröhlich's proof of triviality \cite{F82}. It would be interesting to see how this object relates to the representation we introduce in this paper. Let us mention that even in the simpler case of the Ising model, the connection between the BFS expansion and the random current representation is not completely obvious. The backbone expansion introduced by Aizenman \cite{A82} is such a random walk representation, built from explorations of the underlying random current model, but is not quite the same object as the BFS representation. Let us also mention that although it is a very powerful object, it is not clear whether the BFS representation is enough to obtain the results of this paper. The main reason being that we rely heavily on percolation type arguments that cannot be generalised to these random walk models.

\subsection{Continuity of the phase transition for $\phi^4$ on $\ZZ^d$}

We discuss now an important consequence of Theorem~\ref{thm:main} to the (continuity of) phase transition for $\phi^4$. We restrict ourselves to the hypercubic lattice $G=\ZZ^d$. We fix $g > 0$ and $a \in \RR$. We also fix a collection of interactions $J$ satisfying $(\textbf{C1}$-$\textbf{C4})$.

We begin by defining the phase transition rigorously. There are several natural ways to do this, which are all a posteriori equivalent. We consider here perhaps the most classical definition in terms of spontaneous magnetisation. Other definitions involve existence of long range order (as mentioned above), uniqueness of Gibbs measure, finiteness of susceptibility, etc. We refer the reader to Appendix~\ref{app: different pt} for these definitions and results relating them.
The spontaneous magnetisation at inverse temperature $\beta$ is defined as
$m^*(\beta)
    :=
     \langle \phi_0 \rangle^+_\beta$.
We can now define
\begin{equation*}
    \beta_c
    :=
    \inf\lbrace \beta>0,\text{ } m^*(\beta)>0\rbrace.
\end{equation*}
For $d\geq 2$ it has been proved \cite{GJS75} (see also \cite{glimm2012quantum}) that $\beta_c(g,a)\in (0,\infty)$ for all $g>0$ and $a\in \mathbb R$. 

In view of the definition above, it is a natural question to ask whether $m^*(\beta_c) = 0$ (i.e.~the spontaneous magnetisation vanishes at criticality). If so, $m^*$ is continuous at $\beta_c$ (see Proposition \ref{prop:rightcontinuity}) and the model is said to undergo a \textit{continuous} phase transition. This is a fundamental assumption on which arguments of universality (e.g.\ via renormalisation group heuristics) are built on. 
It turns out that continuity of the phase transition for lattice $\varphi^4$ models boils down to proving that $|\cG(\beta_c)|=1$, see Proposition \ref{prop: spont mag and gibbs is 0}. Theorem \ref{thm:main} gives a sufficient condition \eqref{eq: condition for continuity} for that to be satisfied.

Very few tools are available to obtain condition \eqref{eq: condition for continuity} at criticality. For the Ising model in dimension 2, this can be verified by either using exact integrability established in the famous work of Onsager \cite{onsager1944} or planar duality of the interaction as realised by Werner \cite{W09}. However, we stress these techniques only work for the nearest-neighbour model, and have not been extended to finite range interactions in general, or the case of $\phi^4$. In higher dimensions, one can use the lace expansion for either the Ising model \cite{S07} or $\phi^4$ \cite{sakai2015application}. In the intermediate dimensions, such as the physically relevant dimension $3$ for nearest-neighbour interactions, the infrared/Gaussian domination bound of Fr\"ohlich, Simon, and Spencer \cite{FSS} (see also \cite{FILS}) for reflection positive models (see \cite[Definition 5.1]{B}) is the main tool we have at our disposal. As already observed in \cite{ADCS}, the infrared bound is a good candidate to ensure continuity.

We now assume that the interaction $J$ is admissible and reflection positive--- see Section~\ref{sec: continuity}. The canonical examples of reflection positive interactions are
\begin{itemize}
    \item (Nearest-neighbor interactions) $J_{x,y}=\delta_{| x-y|_1=1}$;
    \item (Exponential decay / Yukawa potentials) $J_{x,y}=C\exp(-\mu | x-y|_1)$ for $\mu,C>0$;
    \item (Power law decay) $J_{x,y}=C| x-y|_1^{-d-\alpha}$ for $\alpha,C>0$,
\end{itemize}
where $|\cdot|_1$ refers to the $\ell^1$-norm on $\RR^d$.
\noindent For each $J$, we consider the associated random walk $(X_k)_{k\geq 0}$ with step distribution given by
\begin{equation*}
    \mathbb P[X_{k+1}=y\: |\: X_k=x]=\frac{J_{x,y}}{|J|}
\end{equation*}
where $|J|:=\sum_{x\in \mathbb Z^d}J_{0,x}$. The infrared bound relates the two-point function of a reflection positive model to the Green's function of the associated random walk $X$. In the case that the random walk is transient, one concludes that the condition \eqref{eq: condition for continuity} from Theorem~\ref{thm:main} holds. We then obtain the following result--- see Section~\ref{sec: continuity} for the details.

\begin{thm}[Continuity of the phase transition]\label{thm: continuity} Let $d \geq 1$. Let $J$ be admissible, reflection positive, and such that the associated random walk on $\ZZ^d$ is transient. Then, 
\begin{equation*}
    m^*(\beta_c)=0.
\end{equation*}
In particular, $m^*(\beta_c) = 0$ for the nearest-neighbour $\phi^4$ model on $\ZZ^d$ when $d \geq 3$.
\end{thm}

\begin{rem}
Theorem \textup{\ref{thm: continuity}} gives that the magnetisation is continuous at $\beta_c$ when $J$ is any admissible and reflection positive interaction in dimension at least $3$. In dimension $2$, it yields the same conclusion when $J$ has power law decay of the form $J_{x,y}=C|x-y|_1^{-d-\alpha}$ with $\alpha\in (0,2)$. Finally, in dimension $1$, provided that the model undergoes a non-trivial phase transition\footnote{We expect this to be true whenever  $J_{x,y}\geq c|x-y|^{-1-\alpha}$ for some $\alpha\in (0,1)$, as in the case of the one-dimensional Ising model \cite{dyson1969}.}, we get that the magnetisation is continuous at $\beta_c$ when $J$ has power law decay with $\alpha\in (0,1)$. 
\end{rem}

\subsection{Open problems}

We now give what, in our opinion, are some beautiful open problems that naturally lead on from Theorems \ref{thm:main} and \ref{thm: continuity}. In what follows, we fix $G = \ZZ^2$ and only consider the nearest-neighbour case.

The most striking omission from Theorem \ref{thm: continuity} is the behaviour at $\beta_c$ for the nearest-neighbour model in $d=2$. Indeed, the infrared bound is not useful in this context and the analogous result for the Ising model relies on exact integrability \cite{onsager1944} or planar duality \cite{W09}, which only works in the nearest-neighbour case; we stress that the case of the finite range Ising model is open. It is widely believed that the phase transition for $\phi^4$, just as for Ising, is continuous in dimension $2$.
\begin{op1}[Continuity in $d=2$]
Let $d=2$ and $J$ be the nearest-neighbour interaction. Prove that 
\begin{equation*}
m^*(\beta_c)
=
0.
\end{equation*}
\end{op1}

The only rigorous evidence that we are aware of universality with respect to changes in the Hamiltonian\footnote{This is in constrast to universality results with respect to changing the underlying \textit{planar} graph for nearest-neighbour models, i.e.\ beyond the square lattice, for which much progress has been made. We refer to the review by Chelkak \cite{C18} for further details.} for the Ising model in $d=2$ are the celebrated works of Pinson and Spencer \cite{S00} and Giuliani, Greenblatt, and Mastropietro \cite{GGM12}, which treat the scaling limit of energy-energy correlations for critical finite range Ising models (with a perturbative condition on the range) via renormalisation group techniques; and the work of Aizenman, Duminil-Copin, Tassion, and Warzel \cite{ADCTW2019}, which proves asymptotic Pfaffian relations for boundary point correlations for any critical finite range model (with no perturbative condition) via stochastic geometric techniques. Both cases are coherent with universality, i.e.\ with the corresponding results established for the critical nearest-neighbour model; see \cite{SML64,GBK78} for the energy-energy and boundary point correlations, respectively. We defer to the articles cited above for precise statements in the case of Ising.
\begin{op2}[Pfaffian relations in $d=2$]
For the nearest-neighbour $\phi^4$ model on $\ZZ^2$, prove that the boundary spin correlations satisfy the Pfaffian relations asymptotically at criticality. 
\end{op2}

Our final problem concerns the set of \textit{all} Gibbs measures in $d=2$. For the Ising model, it is known \cite{A80, H79} that in $d=2$, there are no non-translation invariant Gibbs measures. This is in stark contrast to the existence of (non-translation invariant) Dobrushin measures in $d \geq 3$ as established in \cite{D73}. For $\phi^4$, the existence of Dobrushin measures is known in $d \geq 3$, see \cite{SV07}. The following problem is thus natural.

\begin{op3}[Aizenman--Higuchi result in $d=2$]
Let $d=2$ and $J$ be the nearest-neighbour interaction. Prove that for all $\beta > \beta_c$,
\begin{equation*}
\cG(\beta)
=
\cG_T(\beta),
\end{equation*}
where $\cG_T(\beta)$ is the set of translation invariant Gibbs measures.
\end{op3}


\section{Gibbs measures and correlation functions}\label{sec: gibbs states}

In this section, we prove important properties about Gibbs measures. First, we construct the plus and minus measures, and show that they are mixing and extremal. Then, we provide a Lebowitz-type characterisation on the structure of $\Gamma$-invariant Gibbs measures that reduces the proof of Theorem \ref{thm:main} to checking that even correlations under the plus and free measures coincide. We fix $\beta>0$.

\subsection{Regularity of Gibbs measures}

It is crucial to our techniques that $\langle \cdot \rangle^+_\beta$ and $\langle \cdot \rangle^-_\beta$ are realised as infinite volume limits of measures with appropriate boundary conditions at $h=0$. This is not a priori clear from their definitions via limits as $h \downarrow 0$ and $h \uparrow 0$ of $\langle \cdot\rangle_{\beta,h}^{0}$, respectively. Indeed, it is not even clear that they are even Gibbs measures at $h=0$. As such, we show that any Gibbs measure satisfies a regularity condition that puts a constraint on the growth of the configurations that it is supported on. This allows us to define natural analogues of plus and minus boundary condition measures obtained from finite volume measures that are mixing and extremal, which we shall do in the next subsection.

\begin{defn} \label{def: regular measure}
Let $\nu$ be a measure on $\RR^V$. For $\gamma, \delta > 0$, we say that $\nu$ is $(\gamma,\delta)$-regular if, for every $\Lambda \subset V$ finite, the density of the projection measure satisfies
\begin{equation}\label{eq: regular estimate}
\frac{\textup{d}\nu}{\textup{d}\phi_\Lambda}(\phi_{\Lambda})
\leq
e^{-\sum_{x \in \Lambda} (\gamma \phi_x^4 - \delta)},
\end{equation}
where $\textup{d}\phi_\Lambda$ is the Lebesgue measure on $\RR^\Lambda$.
\end{defn}

The following proposition states that (tempered) Gibbs measures are regular. The proof is technical and relies on estimates developed in \cite{LP76}, adapted to our more general setting on graphs of polynomial growth. We therefore defer the proof to Appendix \ref{app: LP}, c.f.\ Proposition \ref{appendix prop: regular}.
\begin{prop} \label{prop: gibbs is regular}
There exist constants $\gamma,\delta>0$ depending on $G,g,a,J,\beta,h$ such that $\nu$ is  $(\gamma,\delta)$-regular for every $\nu \in \cG(\beta)$. Moreover, the constants $\gamma$ and $\delta$ are continuous in $g,a,J,\beta,h$.
\end{prop}

Regular measures concentrate on weakly growing configurations. 
\begin{lem} \label{lem: regular bc}
Let $\nu$ be $(\gamma,\delta)$-regular. For every $M=M(\gamma,\delta)>0$ sufficiently large, we have
\begin{equation*}
\nu [\Xi(M)]
=
1,
\end{equation*}
where 
\begin{equation*}
    \Xi(M)
    :=
    \{ \phi \in \RR^V: \exists  \Lambda_\phi \subset V \text{ finite},~ \phi^4_x \leq M \log d_G(o,x)  ~~\forall x \in V\setminus \Lambda_\phi \}.
\end{equation*}
\end{lem}
\begin{proof}
This is a direct consequence of Definition \ref{def: regular measure}, a union bound, and Borel-Cantelli.
\end{proof}

In light of Lemma \ref{lem: regular bc}, there is a natural notion of  ``maximal boundary condition''.

\begin{defn}
For $M > 0$, we define the configurations $\PLUS= \PLUS(M)$ and $\MINUS = \MINUS(M)$ as
\begin{equation*}
\PLUS_x = \sqrt[4]{ M \log \big(d_G(o,x)\vee 1\big) },
\qquad \forall x\in V,
\end{equation*}
and $\MINUS=-\PLUS$. We refer to $\langle\cdot\rangle^{\PLUS}_{\Lambda,\beta}$ and $\langle\cdot\rangle^{\MINUS}_{\Lambda,\beta}$ as the plus and minus (finite volume) measures. 
\end{defn}

\begin{rem}\label{rem: Gibss regular limit}
Proposition \textup{\ref{prop: gibbs is regular}} and Lemma \textup{\ref{lem: regular bc}} imply that any Gibbs measure is a limit of finite volume (regular) $\phi^4$ measures $\langle\cdot\rangle_{\Lambda,\beta}^\eta$ with (possibly random) boundary conditions $\eta$ satisfying $|\eta_x|\leq \PLUS_x(M)~~\forall x\in V\setminus \Lambda$, provided $M$ is taken sufficiently large. This is established in Appendix \textup{\ref{app: LP}}. One of the key results to proving Proposition \textup{\ref{prop: gibbs is regular}} is that finite volume measures with boundary condition satisfying $|\eta| \leq \PLUS$ satisfy an analogous estimate to \eqref{eq: regular estimate} with $(\gamma,\delta)$ uniform over subsets sufficiently far away from the boundary, see Corollary \textup{\ref{app cor: finite vol reg}} for a precise statement. Conversely, any sequence of finite volume measures with boundary conditions satisfying $|\eta| \leq \PLUS$ is tight and any accumulation point is a Gibbs measure, see Propositions \textup{\ref{appendix prop: tightness}} and \textup{\ref{appendix prop: dlr}}. 
\end{rem}

\subsection{Extremality and mixing of the plus and minus measures}

Recall that $\langle \cdot \rangle^+_\beta$ and $\langle \cdot \rangle^-_\beta$ were defined as the respective monotone limits $\lim_{h \downarrow 0}\langle \cdot \rangle^{0}_{\beta,h}$ and $\lim_{h \uparrow 0} \langle \cdot \rangle^{0}_{\beta,h}$. We now show that they can be obtained by taking infinite volume limits of measures with maximal boundary conditions at $(\beta, h=0)$, thus giving a proof that they are in $\cG(\beta)$.

\begin{prop} \label{prop: construction of plus and minus}
For $M$ sufficiently large, one has that $\langle\cdot\rangle_{\Lambda,\beta}^\PLUS\to \langle\cdot\rangle_\beta^+$ and $\langle\cdot\rangle_{\Lambda,\beta}^\MINUS\to \langle\cdot\rangle_\beta^-$ as $\Lambda \uparrow V$. Moreover, the measures are maximal and minimal, respectively, in the sense that, for any $\langle \cdot \rangle \in \cG(\beta)$ and any increasing and integrable function $f$,
\begin{equation}\label{eq: maximality}
\langle f \rangle^-_\beta 
\leq
\langle f \rangle
\leq
\langle f \rangle^+_\beta.
\end{equation}
\end{prop}

\begin{proof}
We only prove the convergence results; the monotonicity \eqref{eq: maximality} is a direct consequence of the convergence together with Remark \ref{rem: Gibss regular limit}, and the monotonicity on boundary conditions in finite volume (see Proposition~\ref{prop: increasing in increasing bc}). Without loss of generality, we only consider the case of $\langle \cdot \rangle^+_\beta$. 
First, since the measures $\langle\cdot\rangle_{\Lambda,\beta}^{\PLUS}$ are uniformly  regular (see Remark~\ref{rem: Gibss regular limit}), it is straightforward to prove that they form a tight sequence of measures. Now, let $\Lambda_n\uparrow V$ be any sequence such that $\langle \cdot \rangle^\PLUS_{\Lambda_n,\beta}$ converges to a probability measure, which we denote by $\langle \cdot \rangle^\PLUS_\beta$. Note that this is a Gibbs measure. We prove that $\langle \cdot \rangle^\PLUS_\beta=\langle \cdot \rangle^+_\beta$, which in particular implies that the full sequence $\langle\cdot\rangle_{\Lambda,\beta}^{\PLUS}$ converges.

By monotonicity\footnote{This follows by differentiating in $h$ and using the FKG inequality of Proposition \ref{prop: FKG phi4} with the observation that the identity map $\phi_x \mapsto \phi_x$ is increasing.} in $h$, for any $f$ local, increasing and bounded, and $h>0$,
\begin{equation*}
\langle f \rangle^\PLUS_{\beta}
=
\lim_{n \rightarrow \infty}\langle f \rangle^\PLUS_{\Lambda_n, \beta}
\leq
\lim_{n \rightarrow \infty}\langle f \rangle^{\PLUS}_{\Lambda_n,\beta,h}
=
\nu_{\beta,h} [f],
\end{equation*}
where $\nu_{\beta,h}:=\langle\cdot\rangle_{\beta,h}$ is the unique measure in $\cG(\beta,h)$ for $h\neq0$: this is a consequence of the Lee--Yang theorem for the $\phi^4$ model \cite{GS} and \cite[Theorem 5.6]{LP76}. Since by definition $\nu_{\beta,h}\to \langle\cdot\rangle_{\beta}^+$ as $h\to 0$, we conclude that
$
\langle f \rangle^\PLUS_{\beta}
\leq
\langle f \rangle_{\beta}^+.
$

For the reverse inequality, let $f$ be local, increasing, and bounded. For $h > 0$ and $\Lambda \subset V$ finite, let
$
    \varepsilon(\Lambda,\beta,h)
    :=
    \Vert f \Vert_\infty \nu_{\beta,h}\big[\exists x \notin \Lambda: \: |\eta_x|> \PLUS_x(M)\big].
$
Provided $M$ is large enough, by continuity of the regularity coefficients and concentration on regular boundary conditions, $\sup_{h \in [0,1]}\varepsilon(\Lambda,h) =o(1)$ as $\Lambda \uparrow V$. Thus, by the DLR equations and monotonicity in boundary conditions (see Proposition \ref{prop: increasing in increasing bc}), we have for $n$ sufficiently large
\begin{equation*}
\begin{aligned}
    \langle f\rangle^+_\beta
    &:=
    \lim_{h\rightarrow 0^+}\nu_{\beta,h} [f]
    =
    \lim_{h\rightarrow 0^+} \int \langle f \rangle_{\Lambda_n}^{\eta} \, \textup{d} \nu_{\beta,h}(\eta) 
    \\&
    \leq  \lim_{h\rightarrow 0^+}\left(\varepsilon(\Lambda_n,\beta,h)+\int \langle f\rangle_{\Lambda_n,\beta,h}^{\PLUS}\textup{d}\nu_{\beta,h}(\eta)\right)
    \\&
    \leq  \sup_{h \in [0,1]}\varepsilon(\Lambda_n,\beta,h) + \lim_{h\rightarrow 0^+} \langle f\rangle_{\Lambda_n,\beta,h}^{\PLUS}
    \underset{n \rightarrow \infty}{\longrightarrow }
    \langle f \rangle_{\beta}^\PLUS.
\end{aligned}    
\end{equation*}

Therefore, $\langle \cdot \rangle^+_\beta$ and $\langle \cdot \rangle^{\PLUS}_\beta$ agree on all local bounded increasing functions, hence $\langle \cdot \rangle^\PLUS_{\beta} = \langle \cdot \rangle^+_\beta$, as we wanted to prove.
\end{proof}

\begin{rem} 
Unlike the Ising model, we do not have clear monotonicity in the volume of the quantity $\langle f\rangle_{\Lambda,\beta}^\PLUS$ (for $f$ local, increasing and bounded). This is a consequence of the unboundedness of the spins. As a result, we must use regularity to ``create'' a maximal boundary condition $\PLUS$, which comes with an error term, and temperedness to bound the spins inside a large box. As a byproduct of the methods developed above, we see that for $\Lambda_1\subset \Lambda_2$ two finite subsets of $V$, and for $f$ local (with support in $\Lambda_1$), increasing, bounded,
\begin{equation*}
    \langle f\rangle_{\Lambda_2,\beta}^\PLUS\leq \langle f\rangle_{\Lambda_1,\beta}^\PLUS+\varepsilon(\Lambda_1),
\end{equation*}
where $\varepsilon(\Lambda_1)\rightarrow 0$ as $\Lambda_1\uparrow V$.
\end{rem}

In what follows, we fix $M$ sufficiently large so that Lemma \ref{lem: regular bc} and Proposition \ref{prop: construction of plus and minus} apply. 

\begin{cor}[Extremality] \label{cor: extremality}
The measures $\langle \cdot \rangle^+_\beta$ and $\langle \cdot \rangle^-_\beta$ are extremal in the following sense. Let $\# \in \{+,-\}$. If there exists $\nu_1, \nu_2 \in \cG(\beta)$ and $\lambda \in (0,1)$ such that
\begin{equation*}
    \langle \cdot \rangle^\#_\beta
    =
    \lambda \nu_1 + (1-\lambda)\nu_2,
\end{equation*}
then necessarily $\nu_1 = \nu_2 = \langle \cdot \rangle^\#_\beta$.
\end{cor}

\begin{proof}
This is a direct consequence of \eqref{eq: maximality}.
\end{proof}

As a standard consequence of the extremality of the Gibbs measures proved in Corollary \ref{cor: extremality}, we obtain that the measures $\langle \cdot \rangle^+_\beta$ and $\langle \cdot \rangle^-_\beta$ are mixing. See for example \cite[Theorem 7.7 and Proposition 7.9]{georgii2011gibbs}.
\begin{cor}[Mixing]\label{cor: extremal and ergodic}
Let $\# \in \{+,-\}$. Then, for all local bounded functions $f$ and $g$,
\begin{equation*}
\lim_{d_G(o,x) \rightarrow \infty} \langle f  \cdot \gamma_x g \rangle^\#_\beta
=
\langle f \rangle^\#_\beta \langle g \rangle^\#_\beta
\end{equation*}
for all $\gamma_x \in \Gamma$ such that $\gamma_x o = x$.
\end{cor}

\subsection{Ergodicity of the free measure under even functions}

Recall that the free measure, denoted $\langle \cdot \rangle_{\beta}^0$, is the weak limit of $\langle \cdot \rangle_{\Lambda,\beta}^0$ as $\Lambda \uparrow V$. The free measure is a Gibbs measure. Indeed, by Remark \ref{rem: Gibss regular limit}, we have tightness and that the (unique) limit point is tempered and satisfies the DLR equations.

\begin{prop} \label{prop: free plus ergod}
The measure $\langle\cdot\rangle^0_\beta$ is mixing on even functions in the following sense. For all local even functions $f$ and $g$ such that $f,g$ are square integrable,
\begin{equation*}
\lim_{d_G(o,x) \rightarrow \infty} \langle f  \cdot \gamma_x g \rangle^0_\beta
=
\langle f \rangle^0_\beta \langle g \rangle^0_\beta
\end{equation*}
for all $\gamma_x \in \Gamma$ such that $\gamma_x o = x$.
\end{prop}

\begin{proof}
Let $A:V \rightarrow \NN$ be of finite support and such that $\sum_{x} A_x$ is even. Define $\phi_A := \prod_{x \in V}\phi_x^{A_x}$.  Let $(\tilde J_{x,y})_{x,y \in V}$ be a set of interactions supported on finitely many pairs in $V$ such that $-J_{x,y} \leq \tilde J_{x,y} \leq 0$.
We define 
\begin{equation*}
H_{\tilde J}(\phi)
=
-\sum_{\lbrace x,y\rbrace\in \mathcal{P}_2(V)}\tilde J_{x,y} \phi_x \phi_y,
\end{equation*}
where $\mathcal{P}_2(V) = \big\{ \{x,y\} : x,y \in V \big\}$. Note that, by taking infinite volume limits, the measure $\langle \phi_A \rangle_{\beta,J + \gamma_x^{-1}\tilde J}^0$ is well-defined and satisfies
\begin{equation*}
\langle \phi_A \rangle_{\beta,J + \gamma_x^{-1}\tilde J}^0 
=
\frac{\langle \phi_A e^{-\beta  H_{\tilde J}(\gamma_x \phi)} \rangle_{\beta,J }^0}{\langle e^{-\beta H_{\tilde J}(\gamma_x \phi)} \rangle_{\beta,J}^0}.
\end{equation*}

Now for $x\in V$ such that $n:=d_G(o,x)$ is sufficiently large, by monotonicity in the volume and Griffiths' inequality, we have that
\begin{equation*}
\langle \phi_A \rangle^0_{B_{n/2},\beta, J}
=
\langle \phi_A  \rangle_{B_{n/2}, \beta,J+\gamma_x^{-1}\tilde J}^0
\leq
\langle \phi_A \rangle_{\beta,J + \gamma_x^{-1}\tilde J}^0 
\leq 
\langle \phi_A \rangle_{\beta,J}^0.
\end{equation*}

Hence, by taking $d_G(o,x) \rightarrow \infty$ and using convergence of correlations in the infinite volume limit, we have shown that
\begin{equation*}
\lim_{d_G(o,x) \rightarrow \infty}\langle \phi_A e^{-\beta H_{\tilde{J}}(\gamma_x \phi)} \rangle_{\beta,J}^0
=
\langle \phi_A \rangle_{\beta,J}^0 \langle e^{-\beta H_{\tilde{J}}( \phi)} \rangle_{\beta,J}^0.
\end{equation*}
This establishes Proposition \ref{prop: free plus ergod} for $f(\varphi) = \varphi_A$ and $g(\varphi) =e^{-\beta H_{\tilde{J}}( \phi)}$. By standard approximation arguments via Taylor expanding the exponential, resumming, and using the Stone--Weierstrass theorem, the mixing extends to all even functions. 
\end{proof}


\subsection{A characterisation of $\Gamma$-invariant Gibbs measures via correlations}

The main result of this subsection is a characterisation for the structure of the $\Gamma$-invariant Gibbs measures, see Proposition \ref{Lebowitz}. The analogous result for the Ising model was obtained by Lebowitz \cite{L1977}. In our setting, there are significant complications that come from the unboundedness of the spins. 

A central tool in the proof is the following correlation inequality, which was first established in \cite{L1977}. Recall that if $A:V\rightarrow \NN$ is finitely supported, we denote
$
    \varphi_A
    =
    \prod_{x\in V}\varphi_x^{A_x}.
$
\begin{prop}[Ginibre inequality]\label{Ginibre inequality} Let $A,B: V\rightarrow \NN$ be finitely supported. Then, if $\Lambda$ contains the support of $A$ and $B$, for any $\eta,\eta'\in \Xi(M)$ such that $|\eta_x|\leq \eta'_x$ for all $x\notin \Lambda$,
\begin{equation*}
    \langle \varphi_A\varphi_B\rangle^{\eta'}_{\Lambda,\beta}-\langle \varphi_A\varphi_B\rangle^\eta_{\Lambda,\beta}\geq \left|\langle \varphi_A\rangle^{\eta'}_{\Lambda,\beta}\langle \varphi_B\rangle^\eta_{\Lambda,\beta}-\langle \varphi_A\rangle_{\Lambda,\beta}^\eta\langle\varphi_B\rangle^{\eta'}_{\Lambda,\beta}\right|.
\end{equation*}
Moreover, if $\langle \cdot\rangle_\beta\in \cG(\beta)$,
\begin{equation*}
    \langle \varphi_A\varphi_B\rangle^+_{\beta}-\langle \varphi_A\varphi_B\rangle_\beta\geq \left|\langle \varphi_A\rangle^+_{\beta}\langle \varphi_B\rangle_\beta-\langle \varphi_A\rangle_\beta\langle\varphi_B\rangle^+_{\beta}\right|.
\end{equation*}
\end{prop}
\begin{proof} We include a proof in  Appendix \ref{appendix: ginibre} for completeness.
\end{proof}

We make heavy use of the fact that the sign field of the $\phi^4$ measure can be viewed as an Ising model with random coupling constants coming from the absolute value field. Given $x\in \RR$, we define the sign of $x$ as $\textup{sgn}(x)=x/|x|$ for $x\neq 0$, and $\textup{sgn}(0)=1$. We have the following coupling result.
\begin{lem}\label{coupling lemma}
Let $\Lambda\subset V$ be a finite set and let $\eta\in \RR^V$ such that $|\eta|\leq \PLUS$ (outside $\Lambda$). There exists a coupling $(\PP_{\Lambda,\beta},\varphi^\PLUS,\varphi^{\eta})$, $\varphi^\PLUS\sim \langle \cdot \rangle^\PLUS_{\Lambda,\beta}$, $\varphi^{\eta}\sim \langle \cdot \rangle^{\eta}_{\Lambda,\beta}$ and a non-decreasing function $g:(0,\infty)^2\mapsto (0,\infty)$, such that $\PP_{\Lambda,\beta}$-almost surely, 
\begin{equation*}
    |\varphi^\PLUS_x|\geq |\varphi^{\eta}_x| \quad \text{for every $x\in \Lambda$},
\end{equation*}
\begin{equation*}
\mathbb{E}_{\Lambda,\beta}\left[\textup{sgn}(\varphi^\PLUS_x) \textup{sgn}(\varphi^\PLUS_y) \mid |\varphi^\PLUS|,|\varphi^{\eta}|\right]\geq g(|\varphi^\PLUS_x|,|\varphi^\PLUS_y|) \quad \text{for every $x,y\in \Lambda$},
\end{equation*}
and
\begin{equation*}
\mathbb{E}_{\Lambda,\beta}\left[\textup{sgn}(\varphi^\PLUS_x) \textup{sgn}(\varphi^\PLUS_y)-\textup{sgn}(\varphi^{\eta}_x) \textup{sgn}(\varphi^{\eta}_y)\mid |\varphi^\PLUS|,|\varphi^{\eta}|\right]
\geq 0 \quad \text{for every $x,y\in \Lambda$}.
\end{equation*}
\end{lem}

The proof of this result requires the following lemma. 

\newcommand{\sgn}{\textup{sgn}}
\begin{lem}[Stochastic domination of absolute value fields]\label{lem: absolute domination}
Let $\Lambda\subset V$ be a finite set, let $\beta>0$, and let $\eta\in \RR^V$ such that $|\eta|\leq \PLUS$ (outside $\Lambda$). Then $\langle |\cdot| \rangle^{\PLUS}_{\Lambda,\beta}$ stochastically dominates $\langle |\cdot| \rangle^{\eta}_{\Lambda,\beta}$, in the sense that for any local increasing and bounded function $f$,
\begin{equation*}
    \langle f(|\varphi|)\rangle_{\Lambda,\beta}^\PLUS\geq \langle f(|\varphi|)\rangle_{\Lambda,\beta}^\eta.
\end{equation*}
\end{lem}
\begin{proof}
Let 
\begin{equation*}
    J^{\eta}(|\varphi|)_{x,y}= \begin{cases}
    J_{x,y}|\varphi_x \varphi_y|, & \text{for } x,y\in \Lambda \\
    J_{x,y}|\varphi_x \eta_y|, & \text{for } x\in \Lambda, y\in V\setminus \Lambda.
\end{cases}
\end{equation*}
It follows from a direct calculation that 
\begin{equation*}
    \textup{d}\langle \cdot \rangle^\eta_{\Lambda,\beta}(|\varphi|)=\frac{Z^{\ising,\sgn(\eta)}_{\Lambda,J^{\eta}(|\varphi|),\beta}}{Z^{\eta}_{\Lambda,\beta}}\prod_{x\in \Lambda}\textup{d}\rho(|\varphi_x|),
\end{equation*}
where, for any $\tau \in \{ \pm 1\}^V$,
\begin{equation*}
Z^{\ising,\tau}_{\Lambda,J^{\eta}(|\varphi|),\beta}
:=
\sum_{\sigma\in \{\pm 1\}^{\Lambda}}e^{-\beta H^{\ising,\tau}_{\Lambda,J^{\eta}(|\varphi|)}(\sigma)}
\end{equation*}
and for $\sigma\in \{\pm 1\}^\Lambda$
\begin{equation*}
H^{\ising,\tau}_{\Lambda,J^{\eta}(|\varphi|)}(\sigma)
:=
-\sum_{x,y\in \Lambda}J^{\eta}(|\varphi|)_{x,y}\sigma_x \sigma_y-\sum_{\substack{x\in \Lambda \\ y\in V\setminus \Lambda}}J^{\eta}(|\varphi|)_{x,y}\sigma_x \tau_y.
\end{equation*}
For simplicity, we write $+$ for the configuration $\tau_x = 1$, $\forall x \in V$.

Thus, for any decreasing function $F:[0,\infty)^\Lambda\mapsto \RR$,
\begin{equation*}
    \langle F(|\varphi|) \rangle^{\eta}_{\Lambda,\beta}
    =
    \frac{\big\langle F(|\varphi|) G(|\varphi|) \big\rangle^{\PLUS}_{\Lambda,\beta}}{\big\langle G(|\varphi|) \big\rangle^{\PLUS}_{\Lambda,\beta}},
\end{equation*}
where $G(|\varphi|):=\frac{Z^{\ising,\sgn(\eta)}_{\Lambda,J^\eta(|\varphi|),\beta}}{Z^{\ising,+}_{\Lambda,J^\PLUS(|\varphi|),\beta}}.$
Note that $G(|\varphi|)$ is a decreasing function of $|\varphi|$, which follows from taking the logarithmic derivative and using the Ginibre inequality for the Ising model (see \cite{L1977}). We can now use the FKG inequality for the absolute value field (see Proposition~\ref{absolute FKG}) to conclude that 
\begin{equation*}
   \langle F(|\varphi|) \rangle^{\eta}_{\Lambda,\beta} \geq \langle F(|\varphi|) \rangle^{\PLUS}_{\Lambda,\beta},
\end{equation*}
as desired.
\end{proof}

\begin{proof}[Proof of Lemma~\textup{\ref{coupling lemma}}]
Using Lemma~\ref{lem: absolute domination} and Strassen's theorem \cite{Strassen}, we obtain a coupling $(\PP_{\Lambda,\beta},\psi^\PLUS,\psi^{\eta})$, where $\psi^\PLUS\sim \langle |\cdot |\rangle^\PLUS_{\Lambda}$ and $\psi^{\eta}\sim \langle |\cdot |\rangle^{\eta}_{\Lambda}$, such that $\PP_{\Lambda,\beta}$-almost surely, $\psi^\PLUS_x\geq \psi^{\eta}_x$ for every $x\in \Lambda$. Recall the notation introduced in the proof of Lemma~\ref{lem: absolute domination}.
Enlarging our probability space, we can assume that in the same probability space, there is a family $I=\{\sigma^\PLUS(\textbf{J}) \sim \langle \cdot \rangle^{\ising,+}_{\Lambda,\textbf{J},\beta}, \sigma^{\eta}(\textbf{J}) \sim \langle \cdot \rangle^{\ising,\sgn(\eta)}_{\Lambda,\textbf{J},\beta}  \mid \textbf{J}:\mathcal{P}_2(\Lambda)\mapsto [0,\infty)\}$ of independent Ising models that are also independent from $\psi^\PLUS$ and $\psi^{\eta}$. Now let $\varphi^\PLUS:=\psi^\PLUS \cdot \sigma^\PLUS(J^\PLUS(\psi^\PLUS))$ and $\varphi^\eta:=\psi^\eta \cdot\sigma^\eta(J^{\eta}(\psi^\eta))$. It follows that $\varphi^\PLUS\sim \langle \cdot \rangle^\PLUS_{\Lambda}$, $\varphi^{\eta}\sim \langle \cdot \rangle^{\eta}_{\Lambda}$, and 
$|\varphi^\PLUS_x|\geq |\varphi^{\eta}_x|$ for every $x\in \Lambda$.
Moreover, 
\begin{equation*}
\begin{aligned}
\mathbb{E}_{\Lambda,\beta}&\left[\textup{sgn}(\varphi^\PLUS_x) \textup{sgn}(\varphi^\PLUS_y)-\textup{sgn}(\varphi^\eta_x) \textup{sgn}(\varphi^\eta_y) \mid |\varphi^\PLUS|,|\varphi^\eta|\right]
\\
&=\langle \sigma_x \sigma_y \rangle^{\ising,+}_{\Lambda,J^\PLUS(\psi^\PLUS),\beta}-\langle \sigma_x \sigma_y \rangle^{\ising,\sgn(\eta)}_{\Lambda,J^{\eta}(\psi^\eta),\beta} \geq 0,
\end{aligned}
\end{equation*}
where the equality follows from independence and the inequality from the Ginibre
inequality for the Ising model. Similarly, by the monotonicity properties of the Ising model,
\begin{equation*}
\begin{aligned}
\mathbb{E}_{\Lambda,\beta}\left[\textup{sgn}(\varphi^\PLUS_x) \textup{sgn}(\varphi^\PLUS_y)\mid |\varphi^\PLUS|,|\varphi^\eta|\right]
&=\langle \sigma_x \sigma_y \rangle^{\ising,+}_{\Lambda,J^\PLUS(\psi^\PLUS),\beta}
\\&\geq g(\psi^\PLUS_x,\psi^\PLUS_y)
:=\langle \sigma_x \sigma_y \rangle^{\ising,0}_{\{x,y\},J^0(\psi^\PLUS),\beta}.
\end{aligned}
\end{equation*}
The latter is strictly positive by the random current expansion of the Ising model, and non-decreasing by Griffiths' inequality.
\end{proof}
The following Lemma is the crucial place in this section where $\Gamma$-invariance is required.
\begin{lem}\label{lem: inequality even spin} Let $\beta>0$. Let $\langle\cdot \rangle_\beta\in \cG_{\Gamma}(\beta)$. Then, for any $x,y\in V$ such that $J_{x,y}>0$,
\begin{equation}\label{eq: inequalities J neighbors}
    \langle \varphi_x\varphi_y\rangle_{\beta}^+\geq \langle \varphi_x\varphi_y\rangle_\beta\geq \langle \varphi_x\varphi_y\rangle_{\beta}^0.
\end{equation}
\end{lem}
\begin{proof}

The first inequality in \eqref{eq: inequalities J neighbors} is a direct consequence of Proposition \ref{Ginibre inequality}. The second inequality is more involved and requires $\Gamma$-invariance. We show that a change in temperature is more costly than a change in boundary condition. More precisely, define the function $t\mapsto (\tilde{J}_{u,v}(t))_{u,v\in V}$ as $\tilde{J}_{u,v}(t)=tJ_{u,v}$ if there is $\gamma\in \Gamma$ such that $u=\gamma x$ and $v=\gamma y$, and $\tilde{J}_{u,v}(t)=J_{u,v}$ otherwise. For $t \in (0,1)$, define $\alpha(t):= \langle \varphi_x \varphi_y \rangle^+_{\beta, \tilde{J}(t)}$. In addition, define $ b= \langle \varphi_x \varphi_y \rangle_{\beta}$. Assume that, for any $t\in (0,1)$, 
\begin{equation}\label{eq: ineq a b}
    \alpha(t)\leq b.
\end{equation}
We now argue that the second inequality in \eqref{eq: inequalities J neighbors} is satisfied. Indeed, as a consequence of the Ginibre inequality of Proposition \ref{Ginibre inequality}, we have that $\alpha(t)\geq \langle \varphi_x\varphi_y\rangle_{\beta,\tilde{J}(t)}^0$. The desired inequality then follows from the left-continuity\footnote{This follows from the fact that this function can be realised as a increasing limit of increasing functions.} at $t=1$ of $t\mapsto \langle \varphi_x\varphi_y\rangle_{\beta,\tilde{J}(t)}^0$. It remains to prove \eqref{eq: ineq a b}.

Fix $t \in (0,1)$ and let $\alpha = \alpha(t)$. Let $S_n$ be the set of pairs $u,v\in B_n$ such that there exists $\gamma\in \Gamma$ with $\gamma x=u$ and $\gamma y =v$. Note that there exists $c_1>0$ such that $|S_n|\geq c_1|B_n|$. We define 
$
f(n)
:=
\sum_{\substack{\{u,v\}\in S_n}}\varphi_u \varphi_v.
$
Note that by $\Gamma$-invariance, we have
$
    \langle f(n)\rangle_\beta=b|S_n|, \text{ and } \langle f(n)\rangle_{\beta,\tilde{J}(t)}=\alpha|S_n|.
$
Let $\varepsilon \in (0,\alpha)$. Define the event
$
    C_n
    =
    \{f(n)\leq (\alpha-2\varepsilon_0)|S_n|\}.
$
The key result we need to establish is the following large deviation estimate: there exists $c>0$ such that, for $n$ sufficiently large 
\begin{equation}\label{eq: painful claim}
\big\langle\langle \mathbbm{1}_{C_n}\rangle^\eta_{B_n,\beta,J}\mathbbm{1}_{K_n}\big\rangle_\beta\leq e^{-c|B_n|},
\end{equation}
where $K_n=\{|\eta|\leq \PLUS \text{ in } B_n^c\}$. We assume this for now to establish \eqref{eq: ineq a b}.

On the one hand, by the Cauchy--Schwarz inequality and \eqref{eq: painful claim}, we have that
\begin{equation*}
    \big\langle\langle f(n)\mathbbm{1}_{C_n}\rangle^\eta_{B_n,\beta,J}\mathbbm{1}_{K_n}\big\rangle_\beta
    =
    o(1).
\end{equation*}
On the other hand, by a union bound and regularity,
\begin{equation*}
\begin{aligned}
    \langle \mathbbm{1}_{\{\exists x\notin B_n, \: |\eta_x|>\PLUS_x\}}\rangle_\beta^+&
    \leq \sum_{x\notin B_n}\langle \mathbbm{1}_{\{|\eta_x|
    \geq \PLUS_x\}}\rangle_\beta^+ 
    \leq \sum_{x\notin B_n}\frac{1}{d_G(o,x)^{M}} \lesssim \frac{1}{n^{M-d}}.
\end{aligned}
\end{equation*}
Thus, by the Cauchy--Schwarz inequality, for every $M$ large enough,
\begin{equation*}
    \left|\left\langle \langle f(n)\rangle_{B_n,\beta,J}^\eta\mathbbm{1}_{K_n^c}\right\rangle_{\beta}\right|\leq \sqrt{\langle f(n)^2\rangle_\beta}\sqrt{\langle \mathbbm{1}_{\exists x \notin B_n, \: |\eta_x|>\PLUS_x}\rangle_\beta}
    =
    o(|S_n|).
\end{equation*}
Hence, by the above estimates and the definition of $C_n$,
\begin{equation*}
\begin{aligned}
     b|S_n|
     &=
     \langle f(n)\rangle_{\beta} =
     \langle \langle f(n)\rangle_{B_n,\beta,J}^\eta\mathbbm{1}_{K_n}\rangle_{\beta}+\langle \langle f(n)\rangle_{B_n,\beta,J}^\eta\mathbbm{1}_{K_n^c}\rangle_{\beta}
     \\
     &\geq 
     (\alpha-2\varepsilon_0)|S_n|\big\langle \langle \mathbbm{1}_{C_n^c}\rangle^\eta_{B_n,\beta,J}\mathbbm{1}_{K_n}\big\rangle_\beta - o(|S_n|)
     \\&=
     (\alpha-2\varepsilon_0)|S_n|\left(\langle \mathbbm{1}_{K_n}\rangle_\beta-\big\langle\langle \mathbbm{1}_{C_n}\rangle^\eta_{B_n,\beta,J}\mathbbm{1}_{K_n}\big\rangle_\beta\right) - o(|S_n|).
\end{aligned}     
\end{equation*}
As a result, by applying \eqref{eq: painful claim},
$
     b|S_n|
     \geq 
     (\alpha-2\varepsilon_0)|S_n|-o(|S_n|).
$
Sending $n$ to infinity we obtain $b\geq \alpha-2\varepsilon_0$, which implies that $b\geq \alpha$, as desired.

Finally, we turn to the proof of the large deviation estimate \eqref{eq: painful claim}. By the DLR equations we have 
\begin{equation*}
    \langle f(n)\rangle_{B_n,\beta,\tilde{J}(t)}^{\PLUS}
    \geq
    \langle f(n)\rangle^+_{\beta,\tilde{J}(t)}-\varepsilon(n)
    =
    \alpha|S_n| - \varepsilon(n)
 \end{equation*} 
 where, arguing as above,
\begin{equation*}
    \varepsilon(n)
    :=
    \langle |f(n)| \mathbbm{1}_{\{\exists x\notin B_n, \: |\eta_x|>\PLUS_x\}}\rangle^+_{\beta,\tilde{J}(t)}
    = o(|S_n|).
\end{equation*}
Define 
$
    A_n
    =
    \{f(n)\geq (\alpha-\varepsilon_0)|S_n|\}.
$
For every $n$ large enough, the above considerations imply
\begin{equation*}
    \langle f(n)\mathbbm{1}_{A_n}\rangle^\PLUS_{B_n,\beta,\tilde{J}(t)}= \langle f(n)\rangle_{B_n,\beta,\tilde{J}(t)}^\PLUS-\langle f(n)\mathbbm{1}_{A^c_n}\rangle^\PLUS_{B_n,\beta,\tilde{J}(t)}\geq \varepsilon_0 |S_n|-o(|S_n|)\geq \varepsilon_0|S_n|/2.
\end{equation*}
Observe, furthermore, that
\begin{equation*}
\begin{aligned}
    \langle f(n)\mathbbm{1}_{A_n}\rangle^{\PLUS}_{B_n,\beta,\tilde{J}(t)}
    &= 
    \frac{\langle \mathbbm{1}_{A_n}e^{\beta(t-1)f(n)}f(n)\rangle^\PLUS_{\Lambda,\beta,J}}
    {\langle e^{\beta(t-1)f(n)}\rangle^\PLUS_{B_n,\beta,J}}
    \leq 
    \frac{\langle \mathbbm{1}_{A_n}e^{\beta(t-1)f(n)}|f(n)|\rangle^\PLUS_{\Lambda,\beta,J}}
    {\langle \mathbbm{1}_{C_n}e^{\beta(t-1)f(n)}\rangle^\PLUS_{B_n,\beta,J}}
    \\
    &\leq 
    e^{\beta\varepsilon_0(t-1)|S_n|}\frac{\langle |f(n)|\rangle^{\PLUS}_{B_n,\beta,J}}{\langle \mathbbm{1}_{C_n}\rangle^\PLUS_{B_n,\beta,J}}.
\end{aligned}    
\end{equation*}
Since $\langle \varphi_x\varphi_y \rangle^\PLUS_{B_n,\beta,J}\rightarrow \langle \varphi_x\varphi_y\rangle_{\beta,J}^+ $, we have 
$\langle |f(n)|\rangle^{\PLUS}_{B_n,\beta,J}\leq |S_n|\langle \varphi_x\varphi_y\rangle_{B_n,\beta,J}^\PLUS\leq C_1|S_n|$ for some constant $C_1>0$ independent of $n$. Re-arranging, this yields a bound on $\langle \mathbbm{1}_{C_n} \rangle^\PLUS_{B_n,\beta,J}$. We want to replace the boundary $\PLUS$ by an arbitrary boundary condition in $K_n$. 

Let $\eta \in K_n$. Define
\begin{equation*}
    g_n(\eta)
    :=
    \sum_{\substack{x\in B_n\\y\notin B_n}}J_{x,y}\varphi_x(\PLUS_y-\eta_y).
\end{equation*}
Using the above inequalities, we have that
\begin{equation}\label{eq: proof increasing on two point}
    \frac{\langle \mathbbm{1}_{C_n}e^{\beta g_n(\eta)}\rangle^\eta_{B_n,\beta,J}}{\langle e^{\beta g_n(\eta)}\rangle^\eta_{B_n,\beta,J}}
    =
    \langle \mathbbm{1}_{C_n}\rangle^\PLUS_{B_n,\beta,J}
    \leq 
    \frac{2C_1}{\varepsilon_0}e^{-\frac{\beta\varepsilon_0}{c_1}(1-t)|B_n|}.
\end{equation}
Note that by the Cauchy--Schwarz inequality
\begin{equation*}
\Big(\langle \mathbbm{1}_{C_n}\rangle^\eta_{B_n,\beta,J}\Big)^2\leq \langle \mathbbm{1}_{C_n}e^{\beta g_n(\eta)}\rangle^\eta_{B_n,\beta,J} \langle e^{-\beta g_n(\eta)}\rangle^\eta_{B_n,\beta,J},
\end{equation*}
so that,
\begin{equation}\label{eq: proof increasing on two point 2}
    \langle \mathbbm{1}_{C_n}\rangle^\eta_{B_n,\beta,J}\leq \sqrt{\frac{2C_1}{\varepsilon_0}}e^{-\frac{\beta\varepsilon_0}{2c_1}(1-t)|B_n|}\langle e^{\beta |g_n(\eta)|}\rangle^\eta_{B_n,\beta,J}.
\end{equation}
We seek an estimate on $\langle e^{\beta |g_n(\eta)|}\rangle^\eta_{B_n,\beta,J}$ which is uniform over $\eta \in K$. To this end, observe that, by the definition of $\PLUS$, and by the condition $(\textbf{C4})$ satisfied by $J$,
\begin{equation*}
\begin{aligned}
    |g_n(\eta)|
    &\leq
    2\sum_{\substack{x\in B_n\\y\notin B_n}}J_{x,y}|\varphi_x|\PLUS_y
    &\leq 
    2C\sum_{x\in B_n} |\varphi_x|\sum_{y\notin B_n}\frac{\sqrt[4]{M\log d_G(o,y)}}{d_G(x,y)^{d+\varepsilon}}.
\end{aligned}    
\end{equation*}
Let $\delta>0$ be such that $|B_n\setminus B_{n-1}|=O(n^{-\delta} |B_n|)$ for every $n\geq 1$ (such a $\delta$ exists by \cite{RT}--- see Remark~\ref{rem: delta} in the Appendix).
Note that for every $x\in B_n$ and $y\not\in B_n$, using that $d_G(x,y) \geq 1$ and $d_G(o,x) \leq n$, we have $d_G(o,y)\leq d_G(o,x)+d_G(x,y) \leq (n+1) \, d_G(x,y)$.

Moreover, for every $x\in B_n$ and $y\not\in B_n$, if $d_G(o,x)\leq n-n^{\delta/2}$, then since $d_G(o,y) \geq n$ we have $d_G(x,y)\geq n^{\delta/2}$. Therefore, with the same condition on $d_G(o,x)$, we have
\begin{equation*}
\frac{\sqrt[4]{M\log d_G(o,y)}}{d_G(x,y)^{d+\varepsilon}}\leq \frac{\sqrt[4]{M\log (n+1)+M\log d_G(x,y)}}{d_G(x,y)^{d+\varepsilon}}\leq \frac{1}{d_G(x,y)^{d+\varepsilon/2}}
\end{equation*}
provided that $n$ is large enough. By Lemma~\ref{lem: summable}, the sum
$
\sum_{y\in V\setminus\{x\}}d_G(x,y)^{-d-\varepsilon/2}
$ is finite.
Thus, combining the above estimates and summing over $y$, we find
\begin{equation}\label{eq: o(1)}
    |g_n(\eta)|\leq C'\sum_{x\in B_n} |\varphi_x|\left(\sqrt[4]{M\log n}\mathbbm{1}_{x\in B_n \setminus B_{c_n}}+o(1)\right)
\end{equation}
for some constant $C'>0$, where $c_n=\lfloor n-n^{\delta/2}\rfloor$ and $o(1)$ does not depend $\eta \in K_n$. Plugging this inequality into \eqref{eq: proof increasing on two point 2} and integrating over $K_n$ yields
\begin{equation*}
    \big\langle\langle \mathbbm{1}_{C_n}\rangle^\eta_{B_n,\beta,J}\mathbbm{1}_{K_n}\big\rangle_\beta\leq \sqrt{\frac{2C_1}{\varepsilon_0}}e^{-\frac{\beta\varepsilon_0}{2c_1}(1-t)|B_n|}\left\langle e^{\beta C' \sqrt[4]{M\log n}\sum_{x\in B_n\setminus B_{c_n}}|\varphi_x|+o(1)\sum_{x\in B_n}|\varphi_x|}\right\rangle_{\beta}.
\end{equation*}
By regularity of $\langle \cdot\rangle_\beta$,
\begin{equation*}
    \left\langle e^{\beta C' \sqrt[4]{M\log n}\sum_{x\in B_n\setminus B_{c_n}}|\varphi_x|+o(1)\sum_{x\in B_n}|\varphi_x|}\right\rangle_{\beta}=e^{o(|B_n|)},
\end{equation*}
thus establishing \eqref{eq: painful claim} to complete the proof of \eqref{eq: inequalities J neighbors}.
\end{proof}

We now turn to the main result of this section.
\begin{prop}[Lebowitz-type characterisation]\label{Lebowitz} Let $\beta>0$.
The following are equivalent.
\begin{enumerate}
\item For any $\Gamma$-invariant Gibbs measure $\langle \cdot \rangle_{\beta}\in \cG_\Gamma(\beta)$, there exists $\lambda\in [0,1]$ such that
$\langle \cdot \rangle_{\beta} = \lambda \langle \cdot \rangle^+_\beta + (1-\lambda) \langle \cdot \rangle^-_\beta$.
\item For any $\Gamma$-invariant Gibbs measure $\langle \cdot \rangle_{\beta}\in \cG_\Gamma(\beta)$, and for any bounded and measurable even function $f$, we have
	$\langle f  \rangle^+_\beta
	=
	\langle f  \rangle_{\beta}$.
\item For any $\Gamma$-invariant Gibbs measure $\langle \cdot \rangle_{\beta}\in \cG_\Gamma(\beta)$, and for any $x,y\in V$ such that $J_{x,y}>0$, any vertices $z_1,z_2,\ldots,z_k\in V$, $k\geq 0$, and any positive integers $a_1,a_2,\ldots,a_k$, we have
\begin{equation*}
	\left\langle \varphi_x \varphi_y\prod_{i=1}^k \varphi_{z_i}^{2a_i}  \right\rangle^+_\beta
	=
	\left\langle \varphi_x \varphi_y\prod_{i=1}^k \varphi_{z_i}^{2a_i}  \right\rangle_\beta.
\end{equation*}

\item 
For any $x,y\in V$ such that $J_{x,y}>0$,
$
    \langle \varphi_x \varphi_y  \rangle^+_\beta
	=
	\langle \varphi_x \varphi_y  \rangle^0_\beta.
$
\end{enumerate}
\end{prop}

\begin{proof}
Since $\langle \cdot \rangle_\beta^+$ and $\langle \cdot \rangle_\beta^-$ coincide on even functions, it is straightforward that $1\Rightarrow 4$.

Let us prove that $4\Rightarrow 3$. To this end, by Lemma~\ref{lem: inequality even spin} we have that for any $\langle \cdot \rangle_{\beta}\in \cG_\Gamma(\beta)$,
$
\langle \varphi_x \varphi_y  \rangle^+_{\beta} 
=
\langle \varphi_x \varphi_y  \rangle_{\beta}.
$
Consider some $\varepsilon\in (0,1)$, and a finite set $\Lambda$ that contains $x,y,z_1,\ldots,z_k$.  Let 
\begin{equation*}
\begin{aligned}
    A_{\Lambda}
    &:=
    \{\eta : |\eta_z|\leq \PLUS_z \: \forall z\in \Lambda^c, \:  \langle \varphi_x \varphi_y  \rangle^\PLUS_{\Lambda,\beta} 
    \leq 
    \langle \varphi_x \varphi_y  \rangle^{\eta}_{\Lambda,\beta} +\varepsilon 
    \\
    &\qquad \text{ and } 
    \langle |\varphi_w|  \rangle^\PLUS_{\Lambda,\beta}
    \leq
    \langle |\varphi_w|  \rangle^\eta_{\Lambda,\beta}+\varepsilon, \:\forall \, w=x,y,z_1,\ldots,z_k\}.
\end{aligned}    
\end{equation*}
We first claim that $\langle \mathbbm{1}_{A_{\Lambda}} \rangle_{\beta} \rightarrow 1$ as $\Lambda \uparrow V$. Indeed, it suffices to show that $\langle |\varphi_x|  \rangle^\PLUS_{\Lambda,\beta}
    \leq
    \langle |\varphi_x|  \rangle^\eta_{\Lambda,\beta}+\varepsilon$
with probability tending to $1$ as $\Lambda \uparrow V$. To this end, let 
\begin{equation*}
S_{\Lambda}=\{\eta :|\eta_z|\leq \PLUS_z \: \forall z\in \Lambda^c, \; \langle \varphi_x \varphi_y  \rangle^\PLUS_{\Lambda,\beta} 
\leq \langle \varphi_x \varphi_y \rangle^{\eta}_{\Lambda,\beta} +\delta_1\},
\end{equation*}
where $\delta_1>0$ is a constant to be determined, and note that $\langle \mathbbm{1}_{S_{\Lambda}} \rangle\rightarrow 1$ as $\Lambda \uparrow V$. Fix some $\eta\in S_{\Lambda}$. Now recall the coupling $(\PP_{\Lambda,\beta},\varphi^\PLUS,\varphi^{\eta})$ of Lemma \ref{coupling lemma}. Then,
\begin{equation*}
\langle \varphi_x \varphi_y  \rangle^\PLUS_{\Lambda,\beta}-
\langle \varphi_x \varphi_y  \rangle^\eta_{\Lambda,\beta}\geq 
\mathbb{E}_{\Lambda,\beta}\left[\left(|\varphi^\PLUS_x|-|\varphi^\eta_x|\right)\, |\varphi^\PLUS_y| \, g(|\varphi^\PLUS_x|,|\varphi^\PLUS_y|)\right].
\end{equation*}
Letting $E=\{|\varphi^\PLUS_x|-|\varphi^\eta_x|\geq \delta_2, \, |\varphi^\PLUS_x|\geq \delta_2, \, |\varphi^\PLUS_y|\geq \delta_2\}$, where $\delta_2>0$ is a constant to be determined, we obtain that
$
\delta_1\geq \delta_2^2 \, g(\delta_2,\delta_2) \, \PP_{\Lambda,\beta}[E].
$
By decomposing and applying the Cauchy--Schwarz inequality we obtain 
\begin{equation*}
\langle |\varphi_x| \rangle^\PLUS_{\Lambda,\beta}\leq \sqrt{\langle \varphi^2_x \rangle^\PLUS_{\Lambda,\beta}} \sqrt{\PP_{\Lambda,\beta}\left[E\cup \{|\varphi^\PLUS_x|\leq \delta_2\}\cup \{|\varphi^\PLUS_y|\leq \delta_2\}\right]}+\langle |\varphi_x| \rangle^\eta_{\Lambda,\beta}+\delta_2.    
\end{equation*}
Since $\PP_{\Lambda,\beta}\left[\{|\varphi^\PLUS_x|\leq \delta_2\}\cup \{|\varphi^\PLUS_y|\leq \delta_2\}\right]$ tends to $0$ uniformly in $\Lambda$ as $\delta_2$ goes to $0$, and $\langle \cdot \rangle^\PLUS_{\Lambda,\beta}$ is regular, we can tune the values of $\delta_1$ and $\delta_2$ so that 
$\langle |\varphi_x|  \rangle^\PLUS_{\Lambda,\beta}
\leq \langle |\varphi_x| \rangle^\eta_{\Lambda,\beta}+\varepsilon$, as desired. This completes the proof of the claim.

Since we have
\begin{equation*}
	\mathrm{C}:=\left\langle \varphi_x \varphi_y\prod_{i=1}^k \varphi_{z_i}^{2a_i}  \right\rangle_{\Lambda,\beta}^\PLUS
	-
	\left\langle \varphi_x \varphi_y\prod_{i=1}^k \varphi_{z_i}^{2a_i}  \right\rangle_{\Lambda,\beta}^{\eta}\geq 0
\end{equation*} 
for every $\eta$ such that $|\eta|\leq \PLUS$ outside $\Lambda$ by the Ginibre inequality,
it suffices to obtain an upper bound for $\mathrm{C}$
which converges to $0$ uniformly in $\eta\in A_{\Lambda}$ as $\Lambda\to V$ and $\varepsilon\to 0$. 

Now let 
$
\Phi^\PLUS=\varphi^\PLUS_x \varphi^\PLUS_y \prod_{i=1}^k (\varphi^\PLUS_{z_i})^{2a_i}$ and
$
\Phi^\eta=\varphi^\eta_x \varphi^\eta_y \prod_{i=1}^k (\varphi^\eta_{z_i})^{2a_i}.
$
Then,
\begin{equation}\label{romain use equs} 
\begin{aligned}
     \mathrm{C}
	 =
	 \mathbb{E}_{\Lambda,\beta}\left[\Phi^\PLUS -\Phi^\eta\right]
	 =
	\mathbb{E}_{\Lambda,\beta}\left[\left(\Phi^\PLUS -\Phi^\eta\right)\mathbbm{1}_B\right]+\mathbb{E}_{\Lambda,\beta}\left[\left(\Phi^\PLUS -\Phi^\eta \right)\mathbbm{1}_{B^c}\right],
\end{aligned}    
\end{equation}
where $B$ is the event $\left\{|\varphi^\PLUS_w|\leq |\varphi^\eta_w|+\sqrt{\varepsilon} \text{ for every } w\in \{x,y,z_1,\ldots,z_k\}\right\}$. Letting
\begin{equation*}
f^\PLUS(x,y):=\mathbb{E}_{\Lambda,\beta}\left[\text{sgn}(\varphi^\PLUS_x) \text{sgn}(\varphi^\PLUS_y)\mid |\varphi^\PLUS|,|\varphi^\eta|\right],
\end{equation*}
and
\begin{equation*}
f^\eta(x,y):=\mathbb{E}_{\Lambda,\beta}\left[\text{sgn}(\varphi^\eta_x) \text{sgn}(\varphi^\eta_y)\mid |\varphi^\PLUS|,|\varphi^\eta|\right],
\end{equation*}
we have
\begin{eqnarray*}
\mathbb{E}_{\Lambda,\beta}\left[\left(\Phi^\PLUS -\Phi^\eta \right)\mathbbm{1}_B\right]
&\leq& 
\mathbb{E}_{\Lambda,\beta}\left[|\Phi^\eta| \left(f^\PLUS(x,y)-f^\eta(x,y)\right)\right]+O(\varepsilon^{1/4}) 
\\ &\leq& 
\sqrt{\mathbb{E}_{\Lambda,\beta}\left[(\Phi^\eta)^2\right]}
\sqrt{\mathbb{E}_{\Lambda,\beta}\left[\left(f^\PLUS(x,y)-f^\eta(x,y)\right)^2\right]}
+O(\varepsilon^{1/4}) 
\\ &\leq& 
\sqrt{\mathbb{E}_{\Lambda,\beta}\left[(\Phi^\eta)^2\right]}
\sqrt{2\mathbb{E}_{\Lambda,\beta}\left[f^\PLUS(x,y)-f^\eta(x,y)\right]}
+O(\varepsilon^{1/4}),
\end{eqnarray*}
where the constant in the $O(\varepsilon^{1/4})$ term is uniform in $\Lambda$. In the first line, we used the definition of the event $B$, and the fact that $f^\PLUS(y,z)\geq f^\eta(y,z)$ to remove $\mathbbm{1}_B$. In the second line, we used the Cauchy--Schwarz inequality, and in the third line we used that $0\leq f^\PLUS(y,z)-f^\eta(y,z)\leq 2$. 

Now, define $T$ to be the event $\left\{|\varphi^\eta_x \varphi^\eta_y|\geq \sqrt{\varepsilon}\right\}$, and note that
\begin{eqnarray*}
\varepsilon&\geq& \langle \varphi_x \varphi_y \rangle_{\Lambda,\beta}^\PLUS 
     -
     \langle \varphi_x \varphi_y   \rangle_{\Lambda,\beta}^\eta
	 \geq
	 \mathbb{E}_{\Lambda,\beta}\left[|\varphi^\eta_x \varphi^\eta_y|\left(f^\PLUS(x,y)-f^\eta(x,y)\right)\right]
	 \\ &\geq& \mathbb{E}_{\Lambda,\beta}\left[|\varphi^\eta_x \varphi^\eta_y|\left(f^\PLUS(x,y)-f^\eta(x,y)\right) \mathbbm{1}_T\right]\geq \sqrt{\varepsilon}\mathbb{E}_{\Lambda,\beta}\left[\left(f^\PLUS(x,y)-f^\eta(x,y)\right)\mathbbm{1}_T\right].
\end{eqnarray*}
In the second line, we used that $|\varphi_x^\PLUS \varphi_y^\PLUS|\geq |\varphi_x^\eta \varphi_y^\eta|$ and that $f^\PLUS(x,y)\geq 0$, and in the third line, we used that $f^\PLUS(x,y)\geq f^\eta(x,y)$. It follows that 
$
\mathbb{E}_{\Lambda,\beta}\left[\left(f^\PLUS(x,y)-f^\eta(x,y)\right)\mathbbm{1}_T\right]
\leq \sqrt{\varepsilon}.
$
On the other hand, 
$
\mathbb{E}_{\Lambda,\beta}\left[\left(f^\PLUS(x,y)-f^{\eta}(x,y)\right)\mathbbm{1}_{T^c}\right]
\leq 
2\PP_{\Lambda,\beta}[T^c],
$
and the right-hand side converges to $0$ uniformly in $\Lambda$ and $\eta$ as $\varepsilon$ goes to $0$. Since $\mathbb{E}_{\Lambda,\beta}\left[(\Phi^\eta)^2\right]$ remains bounded uniformly in $\eta$ by regularity of $\langle \cdot \rangle^{\eta}_{\Lambda,\beta}$, we can first send $\Lambda$ to $V$ and then $\varepsilon$ to $0$ to obtain that $\mathbb{E}_{\Lambda,\beta}\left[\left(\Phi^\PLUS -\Phi^\eta \right)\mathbbm{1}_B\right]$ converges to $0$ uniformly in $\eta$.

To handle $\mathbb{E}_{\Lambda,\beta}\left[\left(\Phi^\PLUS -\Phi^\eta \right)\mathbbm{1}_{B^c}\right]$, note that 
\begin{equation*}
\mathbb{E}_{\Lambda,\beta}\left[\left(\Phi^\PLUS -\Phi^\eta \right)\mathbbm{1}_{B^c}\right]
\leq 
\sqrt{\mathbb{E}_{\Lambda,\beta}\left[\left(\Phi^\PLUS -\Phi^\eta \right)^2\right]}\sqrt{\PP_{\Lambda,\beta}[B^c]}
\end{equation*}
by the Cauchy--Schwarz inequality. The first term in the right-hand side is bounded uniformly in $\Lambda$. For the second term, using a union bound we get,
\begin{eqnarray*}
\PP_{\Lambda,\beta}[B^c]
\leq \sum_{w\in\{x,y,z_1,\ldots,z_k\}}\PP_{\Lambda,\beta}\left[|\varphi^\PLUS_w|>|\varphi^\eta_w|+\sqrt{\varepsilon}\right]
\leq (k+2)\sqrt{\varepsilon},
\end{eqnarray*}
where we used Markov's inequality on the second line, together with the inequalities $\langle |\varphi_w|\rangle^\PLUS_{\Lambda,\beta}\leq\langle |\varphi_w| \rangle^\eta_{\Lambda,\beta}+\varepsilon$ and $|\varphi^\PLUS_w|\geq |\varphi^\eta_w|$. Taking first the limit as $\Lambda \uparrow V$ and then $\varepsilon\rightarrow 0$, we can now conclude that  $\mathbb{E}_{\Lambda,\beta}\left[\left(\Phi^\PLUS -\Phi^\eta \right)\mathbbm{1}_{B^c}\right]$ converges to $0$, which implies that 
\begin{equation*}
	\left\langle \varphi_x \varphi_y\prod_{i=1}^k \varphi_{z_i}^{2a_i}  \right\rangle_\beta^+
	=
	\left\langle \varphi_x \varphi_y\prod_{i=1}^k \varphi_{z_i}^{2a_i}  \right\rangle_\beta,
\end{equation*}
as desired.

We now prove $3\Rightarrow 2$. Notice that it is sufficient\footnote{Indeed, if this holds, going to Laplace transforms, one finds that any finite (even) dimensional vector $(\varphi_{x_1},\ldots,\varphi_{x_{2k}})$ has the same law under the two measures which in turns yields that the measures $\langle \cdot \rangle_\beta^+$ and $\langle \cdot \rangle_\beta$ coincide on the set of even measurable functions.} to prove that for any $A:V\rightarrow \NN$ finitely supported with $\sum_{x\in V}A_x$ even, one has
$
    \langle \varphi_A\rangle_\beta^+=\langle\varphi_A\rangle_\beta.
$
First, we prove that for any $x,y\in V$, any $z_1,\ldots,z_k$ with $k\geq 0$, one has,
\begin{equation}\label{equ charact lebo}
    \left\langle \varphi_x \varphi_y\prod_{i=1}^k \varphi_{z_i}^{2a_i}  \right\rangle_\beta^+
	=
	\left\langle \varphi_x \varphi_y\prod_{i=1}^k \varphi_{z_i}^{2a_i}  \right\rangle_\beta.
\end{equation}
Let us start with the case where $x$ and $y$ are such that there exists $z \in V$ such that $J_{x,z}J_{z,y}>0$. Apply Ginibre's inequality to $A=\mathbbm{1}_{x}+\mathbbm{1}_{y}+\sum_{i=1}^k 2a_i \mathbbm{1}_{z_i}$ and $B=\mathbbm{1}_{z}+\mathbbm{1}_y$, and use the hypothesis to get
\begin{eqnarray*}
    0&=&\left\langle \varphi_x \varphi_z\varphi_y^2\prod_{i=1}^k \varphi_{z_i}^{2a_i}  \right\rangle_\beta^+-\left\langle \varphi_x\varphi_z\varphi_y^2\prod_{i=1}^k \varphi_{z_i}^{2a_i}  \right\rangle_\beta
    \\&\geq& \left|\left\langle \varphi_x \varphi_y\prod_{i=1}^k \varphi_{z_i}^{2a_i}\right\rangle^+_{\beta}\langle \varphi_{z}\varphi_{y}\rangle_{\beta}
    -\left\langle \varphi_x \varphi_y\prod_{i=1}^k \varphi_{z_i}^{2a_i}\right\rangle_{\beta}\langle\varphi_{z}\varphi_{y}\rangle^+_{\beta}\right|,
\end{eqnarray*}
so that $\left\langle \varphi_x \varphi_y\prod_{i=1}^k \varphi_{z_i}^{2a_i}  \right\rangle_\beta^+=\left\langle \varphi_x \varphi_y\prod_{i=1}^k \varphi_{z_i}^{2a_i}  \right\rangle_\beta$. 
To get the result in the general case, we use the irreducibility $(\textbf{C3})$ of $J$ and proceed inductively on the number $p$ such that there exists $(x_i)_{0\leq i \leq p}$ with $x_0=x$, $x_p=y$, and $J_{x_0,x_1}\ldots J_{x_{p-1},x_p}>0$. Notice that (\ref{equ charact lebo}) also holds with $x=y$.

To conclude, it is sufficient to prove that for any $x_1,\ldots,x_{2p}, z_1,\ldots,z_k\in V$ with $p\geq 1$ and $k\geq 0$, and any $a_1,\ldots, a_k\in \mathbb N$, one has,
\begin{equation*}
    \left\langle \varphi_{x_1}\ldots \varphi_{x_{2p}}\prod_{i=1}^k \varphi_{z_i}^{2a_i}  \right\rangle_\beta^+
	=
	\left\langle \varphi_{x_1}\ldots \varphi_{x_{2p}}\prod_{i=1}^k \varphi_{z_i}^{2a_i}  \right\rangle_\beta.
\end{equation*}
Once again, let us proceed inductively on $p$. The case $p=1$ was proved above. Assume the result holds for any collection of $2p-2$ points with $p\geq 1$ and consider $x_1,\ldots,x_{2p}\in V$ with $z_1,\ldots,z_k\in V$ and $k\geq 0$. Apply Ginibre's inequality to $A=\sum_{i=1}^{2p}\mathbbm{1}_{x_i}+\sum_{i=1}^k 2a_i\mathbbm{1}_{z_i}$ and $B=\mathbbm{1}_{x_{2p-1}}+\mathbbm{1}_{x_{2p}}$ to get the result.

We now prove $2\Rightarrow 1$. Let $\langle \cdot\rangle_\beta \in \cG_{\Gamma}(\beta)$. We can assume that $\langle \cdot \rangle_\beta$ is extremal. Define the measure $\langle \cdot \rangle^{\textup{sym}}_\beta$ as follows: for any bounded measurable fonction $X: \mathbb R^V\rightarrow \mathbb R$,
\begin{equation*}
    \langle X(\varphi) \rangle^{\textup{sym}}_\beta=\frac{1}{2}\langle X(\varphi) \rangle_\beta+\frac{1}{2}\langle X(-\varphi) \rangle_\beta.
\end{equation*}
Then, notice that for any odd function $X$, one has $\langle X(\varphi) \rangle^{\textup{sym}}_{\beta}=\langle X(\varphi) \rangle^{0}_{\beta}=0$. If $X$ is even, then $\langle X(\varphi) \rangle^{\textup{sym}}_\beta=\langle X(\varphi) \rangle^0_\beta$ by the hypothesis, hence $\langle \cdot \rangle^{\textup{sym}}_\beta=\langle \cdot \rangle^0_\beta$. Now, the hypothesis applied to $\langle \cdot\rangle_\beta^0$ yields the relation $\langle \cdot \rangle^0_\beta=\frac{1}{2}\langle \cdot \rangle^+_\beta+\frac{1}{2}\langle \cdot \rangle^-_\beta$. The uniqueness of the decomposition into extremal states then allows to conclude that $\langle  \cdot \rangle_{\beta}\in \lbrace \langle \cdot \rangle^+_\beta,\langle \cdot \rangle^-_\beta\rbrace$.
\end{proof}

\section{Switching principle for random tangled currents} \label{sec: switching}

In this section, we develop a random tangled current representation for $\phi^4$ which is analogous to the classical random current representation for Ising models. In comparison to the latter representation, an additional complexity appears through the notion of tanglings which we define below. When clear from context, we sometimes write $\langle \cdot\rangle_{\Lambda,\beta,h}=\langle \cdot\rangle_{\Lambda,\beta,h}^0$ to denote the free measure at $(\beta,h)$.



\subsection{Current expansion}
Let $G=(V,E)$ be a countably infinite, locally finite graph. Let $\Lambda$ be a finite subset of $V$.
To include boundary conditions, we introduce a ghost vertex $\fg$ with the condition $\phi_{\fg}\equiv 1$ and extend the model on the augmented set $\Lambda_{\fg} = \Lambda \cup \{ \fg \}$. 	
Let $\Omega_{\Lambda_\fg}:= \NN^{\mathcal{P}_2(\Lambda_\fg)}$ be the space of currents on $\Lambda_{\fg}$, where $\mathbb N=\lbrace 0,1,2,\ldots\rbrace$. For $\n \in \Omega_{\Lambda_\fg}$ and $\lbrace x,y\rbrace \subset \Lambda_\fg$, we use the following notations to denote the value of $\n$ on the pair $\lbrace x,y\rbrace$: $\n(\lbrace x,y\rbrace)=\n_{x,y}=\n(x,y)$.

Given $\n \in \Omega_{\Lambda_\fg}$, let 
$
\partial \n
:=
\lbrace x \in \Lambda : \Delta\n(x) \text{ is odd}\rbrace
$ be the set of sources of $\n$, where
$
\Delta\n(x)
:=
\sum_{y\in \Lambda_\fg} \n_{x,y}	
$
is the $\n$-degree of $x$. Additionally, let $|\n|$ denote the $L^1$ norm of $\n$, i.e.
\begin{equation*}
|\n| 
=
\sum_{\lbrace x,y\rbrace \in \mathcal{P}_2(\Lambda_\fg)} \n_{x,y}.
\end{equation*}
Let
\begin{equation*}
\cM(\Lambda_\fg)
:=
\Big\lbrace A \in \mathbb N^{\Lambda_\fg}: \: A_\fg\leq 1, \: \sum_{x\in \Lambda_\fg}A_x \text{ is even}\Big\rbrace
\end{equation*}
be the set of admissible moments on $\Lambda_{\fg}$. Given $A,B\in \cM_\fg$, define $A+B \in \cM(\Lambda_\fg)$ by
$
(A+B)_x
:=
A_x+B_x
$
for all $x \in \Lambda$, and
\begin{equation*}
(A+B)_\fg
:=
A_\fg+B_\fg \mod 2.
\end{equation*}
Given $A \in \cM(\Lambda_\fg)$, denote by $\partial A:=\lbrace x \in \Lambda
:\: A_x \text{ is odd}\rbrace$ its set of sources (note $\fg$ is not included by convention). We denote by $A=\emptyset$ the element of $\mathcal{M}(\Lambda_\fg)$ that satisfies $A_x=0$ for all $x\in \Lambda_\fg$.
The above definitions also apply naturally to $\cM(\Lambda)$, the set of admissible moments on $\Lambda$.

In the following proposition, we obtain an expression for the spin correlations in terms of currents, and as a corollary we obtain the first Griffiths' inequality. We recall that for $A\in \mathcal{M}(\Lambda_\fg)$, $\varphi_A=\prod_{x\in \Lambda} \varphi_x^{A_x}$.

\begin{prop}[Current expansion for $\phi^4$] Let $\beta>0$ and $h=(h_x)_{x\in \Lambda}\in \mathbb R^\Lambda$. For any $A \in \cM(\Lambda_\fg)$, 
\begin{equation}\label{eq:current_expansion}
\left\langle \varphi_A\right\rangle_{\Lambda,\beta,h}
=
\dfrac{\sum_{\sn=\partial A}w_{\beta,h}^A(\n)}{\sum_{\sn=\emptyset}w_{\beta,h}^\emptyset(\n)},	
\end{equation}
where
\begin{equation}\label{eq:weights_def}
w_{\beta,h}^A(\n)
:=
\prod_{\lbrace x,y\rbrace\subset \Lambda}\dfrac{(\beta J_{x,y})^{\n_{x,y}}}{\n_{x,y}!}\prod_{x\in \Lambda}\dfrac{(\beta h_x)^{\n_{x,\fg}}}{\n_{x,\fg}!}\left\langle \varphi^{\Delta\n(x)+A_x}\right\rangle_{0},
\end{equation}
and we recall that $\langle \cdot \rangle_{0}$ is integration with respect to the single-site measure $\rho_{g,a}$. 

In particular, Griffiths' first inequality holds: if $h_x\geq 0$ for every $x\in \Lambda$, then 
\begin{equation*}
    \left\langle \varphi_A\right\rangle_{\Lambda,\beta,h}\geq 0.
\end{equation*}
\end{prop}
\begin{proof}

Expanding the partition function of the $\phi^4$ model in the same fashion as for the Ising model, we obtain the current representation of the $\phi^4$ model. More precisely, 
\begin{eqnarray*}
Z_{\Lambda,\beta,h}&=& \int_{\mathbb R^\Lambda}\exp\left(\beta\sum_{\lbrace x,y\rbrace \subset \Lambda  }J_{x,y}\varphi_x\varphi_y+\beta\sum_{x\in \Lambda}h_x\varphi_x\right)\prod_{x\in \Lambda}\rho_{g,a}(\text{d}\varphi_x) 
\\ &=& \int_{\mathbb R^\Lambda}\sum_{\n \in \Omega_{\Lambda_\fg}} \prod_{\lbrace x,y\rbrace \subset \Lambda}\dfrac{(\beta J_{x,y})^{\n_{x,y}}}{\n_{x,y}!}\prod_{x\in \Lambda}\dfrac{(\beta h_x)^{\n_{x,\fg}}}{\n_{x,\fg}!}\prod_{x\in \Lambda}\varphi_x^{\Delta\n(x)}\prod_{x\in \Lambda}\rho_{g,a}(\text{d}\varphi_x)
\\ &=&\sum_{\n \in \Omega_{\Lambda_\fg}}\prod_{\lbrace x,y\rbrace\subset \Lambda}\dfrac{(\beta J_{x,y})^{\n_{x,y}}}{\n_{x,y}!}\prod_{x\in \Lambda}\dfrac{(\beta h_x)^{\n_{x,\fg}}}{\n_{x,\fg}!}\prod_{x\in \Lambda}\int_{\mathbb R}\varphi^{\Delta\n(x)}\rho_{g,a}(\text{d}\varphi).
\end{eqnarray*}
Now, note that for $z \in \sn$, 
$
\int_{\mathbb R} \varphi^{\Delta\n(x)}\rho_{g,a}(\text{d}\varphi)=0,
$
since $\text{d}\rho_{g,a}$ is an even measure.
Then,
\begin{equation*}
Z_{\Lambda,\beta,h}=z_{g,a}^{|\Lambda|}\sum_{\substack{\n \in \Omega_{\Lambda_\fg}\\ \sn=\emptyset}}\prod_{\lbrace x,y\rbrace\subset \Lambda}\dfrac{(\beta J_{x,y})^{\n_{x,y}}}{\n_{x,y}!}\prod_{x\in \Lambda}\dfrac{(\beta h_x)^{\n_{x,\fg}}}{\n_{x,\fg}!}\left\langle \varphi^{\Delta\n(x)}\right\rangle_{0}.
\end{equation*}
By the same argument as above, we see that for $A\in \mathcal{M}_\fg$, 
\begin{equation*}
\int_{\mathbb R^\Lambda}\varphi_A \exp\left(\beta\sum_{\lbrace x,y\rbrace \subset \Lambda}J_{x,y}\varphi_x\varphi_y+ \beta\sum_{x\in \Lambda}h_x\varphi_x\right)\prod_{x\in \Lambda}\rho_{g,a}(\text{d}\varphi_x)=z_{g,a}^{|\Lambda|}\sum_{\substack{\n \in \Omega_{\Lambda_\fg}\\ \sn=\partial A}}w_{\beta,h}^A(\n).
\end{equation*}
\end{proof} 

\begin{rem}\label{rem: b.c. as mag. field}
Note that we can also obtain a current expansion for $\langle \varphi_A\rangle_{\Lambda,\beta,h}^\eta$ by viewing the boundary condition as an external field $h'_x:=\sum_{y\notin \Lambda}J_{x,y}\eta_y$ and then writing $\langle \varphi_A\rangle_{\Lambda,\beta,h}^\eta=\langle \varphi_A\rangle_{\Lambda,\beta,h+h'}^0$.
\end{rem}

\begin{rem}
Due to presence of the moments $\langle \phi^{\Delta\n(x)} \rangle_0$, all known proofs of the classical switching lemma for the Ising model do not apply to the above current weights. In order to recover a switching principle, we need the notion of tangled currents, which is introduced in the next section.
\end{rem}

\begin{rem} When $h\equiv 0$, we let $w^A_{\beta}:=w^A_{\beta,0}$.
We may drop the $h$ or $\beta$ dependency on $w_{\beta,h}^A$ when clear from context.
\end{rem}

\subsection{Tangled currents} \label{subsec: tangled currents}

In this section, we introduce the notion of tanglings for single and double currents. Let us start with the former case.

\paragraph{Single tangled current.} We fix $A \in \cM(\Lambda_\fg)$, and let $\n \in \Omega_{\Lambda_\fg}$ be such that $\partial \n = \partial A$. For each  $z\in \Lambda$, we define the block $\cB_z^A(\n)$ as follows : for each $y\in \Lambda_\fg$, it contains $\n_{z,y}$ points labelled $(zy(k))_{1\leq k \leq \n_{z,y}}$, and $A_z$ points labelled $(za(k))_{1\leq k \leq A_z}$. Note that $\cB_z^A(\n)$ has cardinality $\Delta \n(z)+A_z$, which is an even number. We also write $\cB_\fg^A(\n)=\lbrace \fg\rbrace$, and when $A_\fg=1$ we may also write $\fg a(1)=\fg$. For each $z \in \Lambda$, let $\cT_{\n}^A(z)$ be the set of {\it even partitions} of the block $\cB_z^A(\n)$, i.e.\ for every $P=\{P_1,P_2,\ldots,P_k\}\in \cT_{\n}^A(z)$, each $|P_i|$ is even.
Define
\begin{equation*}
\cT_{\n}^A
:=
\underset{z \in \Lambda}{\bigotimes} \,\cT_{\n}^A(z).
\end{equation*}
An element $\ct \in \cT_{\n}^A$ is called a \textit{tangling}. We sometimes refer to an element $\ct_z$ of $\cT_{\n}^A(z)$ as a tangling (of $\cB_z^A(\n)$). This will be clear from the context. We call the pair $(\n, \ct)$ a \textit{tangled current}.

Let $\cH^A(\n)$ be the graph consisting of vertex set 
$
\bigcup_{z \in \Lambda_\fg} \cB_z^A(\n)	
$
and edge set 
\begin{equation*}
\left(\bigcup_{\lbrace x,y\rbrace\subset \Lambda}\left\lbrace \lbrace xy(k),yx(k) \rbrace: \: 1\leq k \leq \n_{x,y}\right\rbrace\right)\cup\left(\bigcup_{x\in \Lambda}\lbrace \lbrace x\fg(k),\fg\rbrace: \: 1\leq k\leq \n_{x,\fg}\rbrace\right).	
\end{equation*}
Given any tangling $\ct \in \cT_{\n}^A$, the graph $\cH^A(\n)$  naturally induces a multigraph $\cH^A(\n,\ct)$ defined as follows. For each block $\cB_z^A(\n)$, if $\ct_z=\{P_1,\ldots,P_k\}\in \cT_\n^A(z)$, for every $1\leq i \leq k$, we add $\binom{|P_i|}{2}$ edges connecting pairwise the vertices lying in $P_i$. In words, we add to $\mathcal H^A(\n)$ the complete graph on the elements of each partition induced by $\ct$.
 This gives a canonical notion of connectivity in a tangled current $(\n,\ct)$. See Figure \ref{fig:my_label} for an illustration.
\begin{rem}
 In the case $A=\emptyset$, we write $\cH(\n,\ct)$ instead of $\cH^\emptyset(\n,\ct)$.
\end{rem}
\begin{rem}\label{extension without ghost tanglings}
We can also extend the definitions for $\n\in \Omega_{\Lambda}$ and $A\in \mathcal{M}(\Lambda)$ by simply noticing that $\n$ can be seen as an element of $\Omega_{\Lambda_\fg}$ and $A$ can be seen as an element of $\mathcal{M}(\Lambda_\fg)$. For a tangling $\ct\in \cT^A_\n$, the graph $\cH^A(\n,\ct)$ is obtained via the same procedure as above, removing $\cB_\fg^A(\n)$.
\end{rem}
\begin{rem} Above, we connected all the elements in a given partition $P_i$ by adding to $\mathcal{H}^A(\n)$ a complete graph on $P_i$. This choice is completely arbitrary and we could have chosen any connected spanning subgraph of $P_i$. However, we find that our choice is the most convenient solution as it preserves a fundamental property: in the graph $\mathcal{H}^A(\n,\ct)$, the only vertices of odd degree are the ones labelled by $A$. As a consequence, $\mathcal{H}^A(\n,\ct)$ must connect these elements pairwise.
    
\end{rem}
In the perspective of constructing a switching principle for tangled currents, it is natural to extend the above definitions to the case where $\n$ is replaced by the sum of two currents.

\begin{figure}
    \centering
    \includegraphics[width=\textwidth]{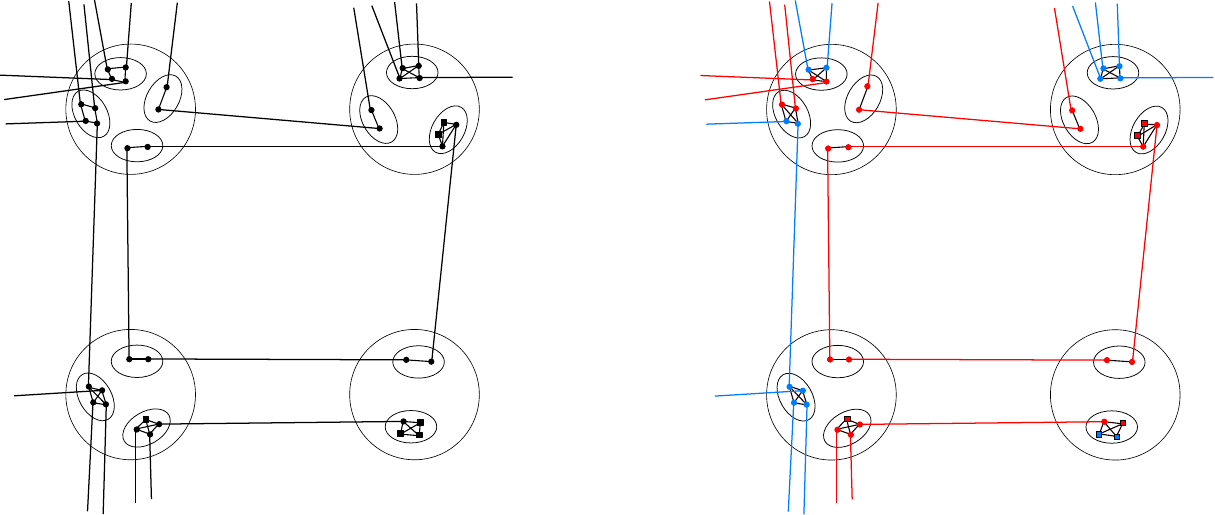}
    \put(-58,10){$\mathcal{B}_{x_4}$}
    \put(-150,10){$\mathcal{B}_{x_3}$}
    \put(-58,95){$\mathcal{B}_{x_2}$}
    \put(-150,95){$\mathcal{B}_{x_1}$}
    \put(-258,10){$\mathcal{B}_{x_4}$}
    \put(-352,10){$\mathcal{B}_{x_3}$}
    \put(-258,95){$\mathcal{B}_{x_2}$}
    \put(-352,95){$\mathcal{B}_{x_1}$}
    \caption{On the left: An example of a single tangled current multigraph $\mathcal{H}^A(\n,\ct)$ with $A_{x_1}=0$, $A_{x_2}=2$, $A_{x_3}=1$ and $A_{x_4}=3$. The circles represent the blocks $\cB^A(\n)$; the small ellipses inside represent the elements of the partitions $\ct$. The vertices $\{za(k):k\geq 1, \: z\in \Lambda\}$ are represented with squares. Notice that $x_3,x_4$ are sources of $\n$. On the right: An example of a double tangled current multigraph $\cH^{A,B}(\n_1,\n_2,\ct)$. The vertices in blue are labelled by $(\n_1,A)$ and the ones in red by $(\n_2,B)$. In this example, $A_{x_1}=A_{x_2}=A_{x_3}=0$ and $A_{x_4}=2$; and $B_{x_1}=0$, $B_{x_2}=2$, and $B_{x_3}=B_{x_4}=1$.}
    \label{fig:my_label}
\end{figure} 

\paragraph{Double tangled current.} Let $\Lambda'\subset \Lambda$ be two finite subsets of $V$.
Fix $A\in \cM(\Lambda_\fg)$ and $B\in \cM(\Lambda'_\fg)$. Let $\n_1 \in \Omega_{\Lambda_\fg}$ and $\n_2 \in \Omega_{\Lambda'_\fg}$ satisfy $\partial \n_1 = \partial A$ and $\partial \n_2 = \partial B$. We trivially view $\n_2$ as an element of $\Omega_{\Lambda_\fg}$ by setting $\n_2(x,y) = 0$ for all pairs $\lbrace x,y\rbrace$ with $x\in \Lambda$ or $y\in \Lambda$. Additionally, we view $B$ as an element of $\cM(\Lambda_\fg)$. To define a notion of tangled currents for the pair $(\n_1,\n_2)$, we use a similar procedure as above. We define $\mathcal{B}_z^{A,B}(\n_1,\n_2)$ as the disjoint union of $\mathcal{B}_z^{A}(\n_1)$ and $\mathcal{B}_z^{B}(\n_2)$.
We write $\cT_{\n_1,\n_2}^{A,B}(z)$ for the set of even partitions of the set $\cB_z^{A,B}(\n_1,\n_2)$ whose restriction to $\cB_z^{A}(\n_1)$ and $\cB_z^{B}(\n_2)$ is also an even partition, i.e.\ for every $P=\{P_1,P_2,\ldots, P_k\}\in \cT_{\n_1,\n_2}^{A,B}(z)$, both $|P_i\cap \cB_z^{A}(\n_1)|$ and $|P_i \cap \cB_z^B(\n_2)|$ are even numbers. The elements of $\cT_{\n_1,\n_2}^{A,B}(z)$ are called {\it admissible} partitions.
As above, we define
\begin{equation*}
\cT_{\n_1,\n_2}^{A,B}
:=
\underset{z \in \Lambda}{\bigotimes} \,\cT_{\n_1,\n_2}^{A,B}(z).
\end{equation*}
As above, we may define a graph $\mathcal H^{A,B}(\n_1,\n_2)$. 
For any tangling $\ct$ in $\cT_{\n_1,\n_2}^{A,B}$, we obtain a multigraph $\cH^{A,B}(\n_1,\n_2,\ct)$ by the following procedure: as for the single tangled current, for each block $\mathcal B_z^{A,B}(\n_1,\n_2)$, if $\ct_z=\{P_1,\ldots,P_k\}\in \cT_{\n_1,\n_2}^{A,B}(z)$, for every $1\leq i \leq k$, we add $\binom{|P_i|}{2}$ edges connecting pairwise the vertices lying in $P_i$. This provides a natural notion of connectivity in $\mathcal{H}^{A,B}(\n_1,\n_2)$. We let $\mathcal{H}^{A,B}_{\Lambda'}(\n_1,\n_2,\ct)$ be the induced subgraph of $\cH^{A,B}(\n_1,\n_2,\ct)$ with vertex set all the vertices in blocks labelled by elements of $\Lambda'$.

Finally, we introduce some useful terminology. Recall that each block $\mathcal B_z^{A,B}(\n_1,\n_2)$ contains four types of vertices: the \emph{internal vertices} which are either labelled by $A$ or by $B$, and the \emph{external vertices} which are labelled either by $\n_1$ or $\n_2$. For the switching principle stated in the next section, the relevant geometric object to consider is the graph $\overline{\mathcal{H}}^{A,B}(\n_1,\n_2,\ct)$ (resp. $\overline{\mathcal{H}}_{\Lambda'}^{A,B}(\n_1,\n_2,\ct)$)  obtained from $\mathcal{H}^{A,B}(\n_1,\n_2,\ct)$ (resp. $\mathcal{H}^{A,B}_{\Lambda'}(\n_1,\n_2,\ct)$) by removing the labels of all the external vertices. 

\begin{rem}
As in Remark \textup{\ref{extension without ghost tanglings}}, we can obtain a notion of tangled currents for pairs of currents $(\n_1,\n_2)$ such that $\n_1\in \Omega_{\Lambda_\fg}$ and $\n_2\in \Omega_{\Lambda}$, or $\n_1\in \Omega_\Lambda$ and $\n_2\in \Omega_{\Lambda'}$. Moreover, if $A=B=\emptyset$, we write $B_z(\n_1,\n_2)=B_z^{\emptyset,\emptyset}(\n_1,\n_2)$ and  $\mathcal{H}(\n_1,\n_2,\ct)=\mathcal{H}^{\emptyset,\emptyset}(\n_1,\n_2,\ct)$ (resp. $\overline{\mathcal{H}}(\n_1,\n_2,\ct)=\overline{\mathcal{H}}^{\emptyset,\emptyset}(\n_1,\n_2,\ct)$).
\end{rem}

\paragraph{The space of (unlabelled) double tangled currents.}
For $A\in \mathcal M(\Lambda)$ and $B\in \mathcal M(\Lambda')$, we define $ \Omega_{\Lambda,\Lambda'}^{\cT,A,B}$ to be the set of double tangled currents on $\Lambda$ with source functions $A,B$, i.e the set of triples $(\n_1,\n_2,\ct)$ where: $(\sn_1,\sn_2)=(\partial A,\partial B)$, and $\ct \in \cT^{A,B}_{\n_1,\n_2}$.  We sometimes identify such a triple with its associated multigraph $\mathcal H^{A,B}(\n_1,\n_2,\ct)$. We set
 \begin{equation}
 	\Omega^\cT_{\Lambda,\Lambda'}:=\bigcup_{(A,B)\in \mathcal M(\Lambda)\times \mathcal M(\Lambda')}\Omega_{\Lambda,\Lambda'}^{\cT,A,B}.
 \end{equation}
Moreover, we define $\overline{\Omega}_{\Lambda,\Lambda'}^{\mathcal{T},A,B}$ to be the image of $\Omega_{\Lambda,\Lambda'}^{\mathcal{T},A,B}$ under the map $\mathcal{H}^{A,B}(\n_1,\n_2,\ct)\mapsto \overline{\mathcal{H}}^{A,B}(\n_1,\n_2,\ct)$.
\begin{rem}\label{rem: bijection (A,B) (A+B,0)}
Note that there exists a (natural) bijection between $\overline{\Omega}_{\Lambda,\Lambda'}^{\mathcal{T},A,B}$ and $\overline{\Omega}_{\Lambda,\Lambda'}^{\mathcal{T},A+B,\emptyset}$ which can be described as follows.
If $\mathcal{H}\in \overline{\Omega}_{\Lambda,\Lambda'}^{\mathcal{T},A+B,\emptyset}$, each block $\mathcal{B}_z$ of $\mathcal{H}$ contains $A_z+B_z$ internal vertices, declare that the first $A_z$ internal vertices are labelled by $A$ and that the last $B_z$ internal vertices are labelled by $B$. This defines uniquely an element of $\overline{\Omega}_{\Lambda,\Lambda'}^{\mathcal{T},A,B}$ that we denote by $\iota(\mathcal{H})$. It is not hard to construct the inverse of $\iota$.
If $F:\overline{\Omega}^{\mathcal{T},A,B}_{\Lambda,\Lambda'}\rightarrow \mathbb R$, we will often abuse notation and write $F=F\circ \iota$, which allows us to view $F$ as a function on $\overline{\Omega}^{\mathcal{T},A+B,\emptyset}_{\Lambda,\Lambda'}$ as well.
\end{rem}

\subsection{A switching principle for tangled currents}
 Let $\Lambda'\subset\Lambda$ be finite subsets of $V\cup \lbrace \fg\rbrace$. Fix $\beta>0$ and $h=(h_x)_{x\in V}\in (\mathbb R^+)^{V}$. Before stating the main result of the section, we need a definition.

\begin{defn}[Pairing event]\label{def: pairing event} Let $C\in \mathcal M(\Lambda)$ and  $B\in \cM(\Lambda')$. Assume that $B\leq C$, i.e.\ $B_x\leq C_x$ for every $x\in \Lambda'$. Define $\cF_{\Lambda'}^{C}(B)$ to be the subset of $\Omega_{\Lambda,\Lambda'}^{\cT,C,\emptyset}$ consisting of all double tangled currents $(\n_1,\n_2,\ct)$ satisfying the following condition: each connected component of $\mathcal H_{\Lambda'}^{C,\emptyset}(\n_1,\n_2,\ct)$ intersects $\{zc(k): z\in \Lambda', \: C_z-B_z+1\leq k \leq C_z\}$ an even number of times.
Notice that $\mathcal{F}^C_{\Lambda'}(B)$ only depends on $\overline{\mathcal H}_{\Lambda'}^{C,\emptyset}(\n_1,\n_2,\ct)$.
When $\Lambda'=\Lambda$, we write $\mathcal F^C(B)=\mathcal F^C_{\Lambda}(B)$. 
\end{defn}
\begin{rem}
The above definition is to be compared with the pairing event which naturally appears in the switching lemma of the Ising model.
\end{rem}



\begin{thm} \label{thm: switching lemma} For every $z\in \Lambda$, every $(C,D)\in \mathcal M(\Lambda)\times \mathcal M(\Lambda')$, every $(\n_1,\n_2)\in \Omega_\Lambda\times \Omega_{\Lambda'}$ with $(\sn_1,\sn_2)=(\partial C,\partial D)$, there exists a probability measure $\rho_{z,\n_1,\n_2}^{C,D}$ on $\cT_{\n_1,\n_2}^{C,D}(z)$ such that the following holds.
Let $A \in \cM(\Lambda)$ and $B \in \cM(\Lambda')$.  
For every bounded $F:\overline{\Omega}_{\Lambda,\Lambda'}^{\cT,A,B} \rightarrow \RR$, 
\begin{equation*}
\begin{aligned}
\sum_{\substack{\partial \n_1 = \partial A \\ \partial \n_2 = \partial B}} 
&w_{\beta,h}^A(\n_1) w_{\beta,h}^B(\n_2) \rho_{\n_1,\n_2}^{A,B}\big[ F[\overline{\mathcal H}^{A,B}(\n_1,\n_2,\ct)] \big]
\\
&=
\sum_{\substack{\partial \n_1 = \partial(A+B) \\ \partial \n_2 = \emptyset}} 
w_{\beta,h}^{A+B}(\n_1)w_{\beta,h}^\emptyset(\n_2)  \rho_{\n_1,\n_2}^{A+B,\emptyset} \big[ F[\overline{\mathcal H}^{A+B,\emptyset}(\n_1,\n_2, \ct)] \mathbbm{1}_{\cF_{\Lambda'}^{A+B}(B)} \big],
\end{aligned}
\end{equation*}
where
\begin{equation*}
\rho_{\n_1,\n_2}^{A,B}
:=
\bigotimes_{z \in \Lambda} \rho_{z,\n_1, \n_2}^{A,B},\quad \text{ and } \quad \rho_{\n_1,\n_2}^{A+B,\emptyset}:=\bigotimes_{z \in \Lambda} \rho_{z,\n_1, \n_2}^{A+B,\emptyset},
\end{equation*}
are measures on $\cT_{\n_1,\n_2}^{A,B}$ and $\cT_{\n_1,\n_2}^{A+B,\emptyset}$ respectively, and $\cF_{\Lambda'}^{A+B}(B)$ is the event defined above.
\end{thm}

\begin{proof}
See Section \ref{section: proof of switching lemma}.	
\end{proof}

\begin{rem}
As we will see in the proof, for any $z$, the measures $\rho_{z,\n_1,\n_2}^{A,B}$ only depend on $(g,a,\Delta\n_1(z),\Delta\n_2(z),A_z,B_z)$, and not on $\Lambda$, $\beta$, $h$ or the rest of $\n_1,\n_2$.
\end{rem}

We can also state the switching lemma in a probabilistic way. We start by introducing the random tangled current measures of interest. We define $\rho_{z,\n}^{A}:=\rho^{A,\emptyset}_{z,\n,0}$ and
\begin{equation*}
    \rho_\n^{A}:=\bigotimes_{z\in \Lambda}\rho_{z,\n}^{A}.
\end{equation*}
The measure $\rho_\n^{A}$ can be viewed as a measure on single tangled currents.

With a small abuse of notation, we view the set of single tangled currents on $\Lambda$ as a subset of $\Omega^\cT_{\Lambda,\Lambda'}$ that we denote by $\Omega_\Lambda^{\cT}$. Similarly, we define for $A\in \mathcal M(\Lambda)$ the set $\Omega_\Lambda^{\cT,A}\subset \Omega_{\Lambda,\Lambda'}^{\cT,A,\emptyset}$.

\begin{defn}
Let $\Lambda'\subset \Lambda$ be finite subsets of $V\cup \lbrace \fg\rbrace$. Let $A\in \cM(\Lambda)$ and $B\in \cM(\Lambda')$. We define a measure $\bfP^{A}_{\Lambda,\beta,h}$ on the $\Omega_{\Lambda}^{\cT,A}$: for any $(\n,\ct)\in \Omega_{\Lambda}^{\cT,A}$,
\begin{equation*}
    \bfP^{A}_{\Lambda,\beta,h}[(\n,\ct)]=\frac{w_{\beta,h}^A(\n)\rho_\n^{A}[\ct]}{\sum_{\partial \m=\partial A}w_{\beta,h}^A(\m)}.
\end{equation*}
We also define a measure $\bfP_{\Lambda,\Lambda',\beta,h}^{A,B}$ on $\Omega^{\cT,A,B}_{\Lambda,\Lambda'}$: for $(\n_1,\n_2,\ct)\in \Omega^{\cT,A,B}_{\Lambda,\Lambda'}$,
\begin{equation*}
    \bfP^{A,B}_{\Lambda,\Lambda',\beta,h}[(\n_1,\n_2,\ct)]=\frac{w_{\beta,h}^A(\n_1)w_{\beta,h}^B(\n_2)\rho^{A,B}_{\n_1,\n_2}[\ct]}{\sum_{\substack{\partial \m_1=\partial A\\ \partial \m_2=\partial B}}w_{\beta,h}^A(\m_1)w_{\beta,h}^B(\m_2)}.
\end{equation*}
We write $\mathbf{E}^A_{\Lambda,\beta,h}$ (resp. $\mathbf{E}^{A,B}_{\Lambda,\Lambda',\beta,h}$) the expectation with respect to the measure $\mathbf{P}^A_{\Lambda,\beta,h}$ (resp. $\mathbf{P}^{A,B}_{\Lambda,\Lambda',\beta,h}$).
\end{defn}
Using these definitions, we can reformulate the switching lemma in a more probabilistic way.
\begin{cor}[Probabilistic version of the switching lemma]\label{thm: probabilistic switching}
 Let $A \in \cM(\Lambda)$, $B \in \cM(\Lambda')$. For any bounded $F:\overline{\Omega}_{\Lambda,\Lambda'}^{\cT,A,B}\rightarrow \mathbb R$,
\begin{equation*}
    \frac{\langle \varphi_A\rangle_{\Lambda,\beta,h}\langle \varphi_B\rangle_{\Lambda',\beta,h}}{\langle \varphi_{A+B}\rangle_{\Lambda,\beta,h}}\mathbf{E}^{A,B}_{\Lambda,\Lambda',\beta,h}[F[\overline{\mathcal H}^{A,B}(\n_1,\n_2,\ct)]]=\mathbf{E}^{A+B,\emptyset}_{\Lambda,\Lambda',\beta,h}[F[\overline{\mathcal H}^{A+B,\emptyset}(\n_1,\n_2,\ct)]\mathbbm{1}_{\cF_{\Lambda'}^{A+B}(B)}].
\end{equation*}
\end{cor}
In what follows, $\beta,h$ are often fixed and omitted to simplify the notation. In particular, we omit them in the notations of the tangling measures.

\begin{rem}
Taking $a>0$ and letting $g \rightarrow 0$, the $\varphi^4$ model on a finite graph converges to a (massive) Gaussian Free Field. It is therefore natural to expect the switching principle to pass through the limit. This is indeed the case, and one may additionally show that the tangling measures converge to the uniform measure on pairings, which is directly related to Wick's formula.
Similarly, one may take the Ising limit as in \eqref{eq: ising limit} and ask a similar question. In this case, one can prove that the tangling measures converge to the Dirac measure supported on the trivial partition consisting of one element, thus recovering the standard switching lemma of the Ising model. We refer to \textup{\cite[Section~4]{krachun2023scaling}} for more details.
\end{rem}

\subsection{Application: correlation inequalities}

We fix $\beta>0$ and drop it from the notation. The first result gives a new exact probabilistic representation for the ratio of spin correlations.
\begin{prop}\label{prop: first consequence sl} Let $\Lambda$ be a finite subset of $V$. Let $h=(h_x)_{x\in \Lambda}\in (\mathbb R_+)^\Lambda$. For $A,B\in \mathcal{M}(\Lambda_\fg)$, 
\begin{equation*}
    \dfrac{\langle \varphi_A\rangle_{\Lambda,h} \langle \varphi_B\rangle_{\Lambda,h}}{\langle \varphi_{A+B}\rangle_{\Lambda,h}}=\mathbf P_{\Lambda_g,\Lambda_\fg}^{A+B,\emptyset}[(\n_1,\n_2,\ct) \in \mathcal{F}^{A+B}(B)]
\end{equation*}
\end{prop}
\begin{proof}
This is a direct consequence of Corollary \ref{thm: probabilistic switching} applied to $F=1$.
\end{proof}

We will now apply Proposition \ref{prop: first consequence sl} to give new proofs of standard correlation inequalities for the $\varphi^4$ model. Alternative proofs can be found in \cite[Chapter 4]{glimm2012quantum}.

\begin{prop}[Second Griffiths' inequality]\label{prop: second consequence sl} Let $\Lambda$ be a finite subset of  $V$, and $h=(h_x)_{x\in \Lambda}\in (\mathbb R_+)^\Lambda$. Let $A,B\in \cM(\Lambda_\fg)$. Then
\begin{equation*}
    \langle \varphi_A\varphi_B\rangle_{\Lambda,h}\geq \langle \varphi_A\rangle_{\Lambda,h}\langle \varphi_B\rangle_{\Lambda,h}.
\end{equation*}
\end{prop}
\begin{proof}
This is a direct consequence of Proposition \ref{prop: first consequence sl}
\end{proof}
\begin{prop}[Monotonicity in volume] \label{prop: monotonicity in vol} Let $\Lambda'\subset \Lambda$ be finite subsets of $V$, and $h=(h_x)_{x\in \Lambda}\in (\mathbb R_+)^\Lambda$. Let $B\in \cM(\Lambda_\fg)$. Then,
\begin{equation*}
    \langle \varphi_B\rangle_{\Lambda,h}\geq \langle \varphi_B\rangle_{\Lambda',h}.
\end{equation*}
\end{prop}

\begin{proof}
Apply Corollary \ref{thm: probabilistic switching} to $A=0$ and $F=1$ to obtain that
\begin{equation*}
    \langle \varphi_B\rangle_{\Lambda',h}=\langle \varphi_B\rangle_{\Lambda,h}\bfP_{\Lambda_\fg,\Lambda'_\fg}^{B,\emptyset}[\cF_{\Lambda'_\fg}^{B}(B)].
\end{equation*}
\end{proof}

\begin{prop} \label{prop: applications monotonicity switching}
Let $\Lambda$ be a finite subset of $V$, $h=(h_x)_{x\in \Lambda}\in (\mathbb R_+)^\Lambda$, and $A\in \cM(\Lambda_\fg)$. Then,
\begin{enumerate}
    \item[$(i$)] $\beta\mapsto \langle \varphi_A\rangle_{\Lambda,\beta,h,g,a}$ is non-decreasing on $\RR_+$,
    \item[$(ii)$] $g\mapsto \langle \varphi_A\rangle_{\Lambda,\beta,h,g,a}$ is non-increasing on $\RR$,
    \item[$(iii)$] $a\mapsto \langle \varphi_A\rangle_{\Lambda,\beta,h,g,a}$ is non-increasing on $\RR$.
\end{enumerate}
\end{prop}
\begin{proof}
Notice that
\begin{equation*}
    \frac{\textup{d}}{\textup{d}\beta} \langle \varphi_A\rangle_{\Lambda,\beta,h,g,a}=\sum_{\lbrace x,y\rbrace \subset\Lambda}J_{x,y}\langle \varphi_A;\varphi_x\varphi_y\rangle_{\Lambda,\beta,h},
\end{equation*}
\begin{equation*}
    \frac{\textup{d}}{\textup{d}g} \langle \varphi_A\rangle_{\Lambda,\beta,h,g,a}=-\sum_{x\in \Lambda}\langle \varphi_A;\varphi_x^4\rangle_{\Lambda,\beta,h,g,a},
\end{equation*}
and
\begin{equation*}
    \frac{\textup{d}}{\textup{d}a} \langle \varphi_A\rangle_{\Lambda,\beta,h,g,a}=-\sum_{x\in \Lambda}\langle \varphi_A;\varphi_x^2\rangle_{\Lambda,\beta,h,g,a},
\end{equation*}
where we used the notation $\langle \varphi_A;\varphi_B\rangle_{\Lambda,\beta,h,g,a}:=\langle \varphi_A\varphi_B\rangle_{\Lambda,\beta,h}-\langle \varphi_A\rangle_{\Lambda,\beta,h}\langle \varphi_B\rangle_{\Lambda,\beta,h}$.
\end{proof}

\begin{rem} \label{rem: monotoncity griffiths}
Note that there is also monotonicity in $h$ and $J$ analogous to the statements in Proposition \textup{\ref{prop: applications monotonicity switching}}. They follow by similar considerations.
\end{rem}

\subsection{Proof of switching principle}\label{section: proof of switching lemma}

In this section, we use the Griffiths--Simon approximation \cite{GS} to construct a random tangled current representation of $\phi^4$ on a finite subset $\Lambda$ of $V$ which satisfies a switching principle. The prelimit is an Ising model on the product of $\Lambda$ with the complete graph on $N$ vertices. We use the switching principle of the random current representation of the Ising model to obtain the switching principle for $\phi^4$ by taking limits, as appropriate. Along the way, we identify the underlying geometric structures that persist in the scaling limit.

\subsubsection{Griffiths--Simon approximation}\label{section: gs approx}

Let $N \in \mathbb N$. Define $\Lambda_N:=\Lambda\times \lbrace 1,\ldots,N\rbrace$. For $x \in \Lambda$, define
$
B_{x,N}
:=
\lbrace (x,i)\rbrace_{1\leq i \leq N},
$
and set $B_{\fg,N} := \lbrace \fg\rbrace$.

Let $\cG_N$ be the graph with vertex set $V(\cG_N):=\Lambda_N\cup \{\fg\}$ and edge set 
\begin{equation*}    E(\cG_N)
:=
E^{\textup{int}}(\cG_N) \cup E^{\textup{ext}}(\cG_N),
\end{equation*}
where
\begin{equation*}
\begin{aligned}
E^{\textup{int}}(\cG_N)
&:=
\bigcup_{x\in \Lambda}\mathcal{P}_2(B_{x,N}),
\\
E^{\textup{ext}}(\cG_N)
&:=
\Big\{ \{ u,v\}  \text{ such that there exists }  \lbrace x,y\rbrace \in E(\Lambda_\fg) \text{ with } u \in B_{x,N} \text{ and }v \in B_{y,N} \Big\},
\end{aligned}
\end{equation*}
and, above, $\cP_2(B_{x,N})$ denotes the set of pairs in $B_{x,N}$. 

\begin{rem}
An edge of the form $\lbrace (x,j),(x,k)\rbrace \in E^{\textup{int}}$ is called \textit{internal}, whereas an edge of the form $\lbrace (x,j),(y,k)\rbrace \in E^{\textup{ext}}$  is called \textit{external}.
\end{rem}

We now introduce the Griffiths--Simon model. Let $\tilde g = (12g)^{1/4}$ and $\tilde a = 2a {\tilde g}^{-2}$. Define the coupling constants  
\begin{equation}\label{eq: def c_N and d_N}
c_N:=\tilde{g}^{-1}N^{-3/4},
\qquad 
d_N:=\frac{1}{N}\left(1-\frac{\tilde{a}}{\sqrt{N}}\right).
\end{equation}
For $\beta>0$, $h\in (\RR_+)^\Lambda$ and $\sigma$ in $\lbrace \pm 1\rbrace^{\Lambda_N}$, define the Hamiltonian 

\begin{equation*}
\begin{aligned}
\mathbf{H}^h_{\Lambda,\beta,N}(\sigma)
&:=
-\beta c_N^2\sum_{\substack{\lbrace x,y\rbrace \subset \Lambda\\j,k\in \lbrace   1,\ldots,N\rbrace}}J_{x,y}\sigma_{(x,j)}\sigma_{(y,k)}
\\
&\qquad\quad\quad
-\beta c_N \sum_{(x,j) \in \Lambda_N} h_x \sigma_{(x,j)}-d_N\sum_{\substack{x \in \Lambda\\  j,k \in \lbrace 1,\ldots,N\rbrace}}\sigma_{(x,j)}\sigma_{(x,k)}.
\end{aligned}
\end{equation*}

\begin{defn}\label{def: measure gs graph}
Let $\beta > 0$. Let $\mu^h_{\Lambda,N,\beta}$ be the probability measure on $\lbrace \pm 1\rbrace^{\Lambda_N}$ defined for $\sigma \in \lbrace \pm 1\rbrace^{\Lambda_N}$ by
\begin{equation*}
\mu_{\Lambda,N,\beta}^h(\sigma)=\dfrac{\exp(- \mathbf{H}^h_{\Lambda,\beta,N}(\sigma))}{\mathbf{Z}(\Lambda,N,\beta,h)}
\end{equation*}
where $\mathbf{Z}(\Lambda,N,\beta,h)$ is the partition function. In the specific case where $\Lambda$ is reduced to a point $x$, we denote this measure $\mu_{x,N,\beta}^h$. Note that this measure does not depend on $x$.
\end{defn} 

Observe that $\mu^h_{\Lambda,N,\beta}$ is an Ising model on $\Lambda_N$. The coupling constants are such that the model has a scaling limit as $N \rightarrow \infty$, which is the $\phi^4$ model on $\Lambda$ with parameters $(g,a)$ at $(\beta,h)$. The natural observable to state the convergence is the block-spin average, defined for $x \in \Lambda$ by
$
\Phi_{x,N}
:=
c_N\sum_{j=1}^N\sigma_{(x,j)}.
$
The following result was first derived in \cite{GS} to obtain, amongst other things, the Lee--Yang theorem for the $\varphi^4$ model on $\mathbb Z^d$. It was later reproved for a larger class of models in \cite{dunlopnewman1975}.

\begin{prop}\label{prop: CV griffiths} Let $\beta > 0$ and $h \in (\RR_+)^\Lambda$. Then, for any $p \in \NN^\Lambda$, 
\begin{equation*}
\lim_{N \rightarrow \infty} \mu^h_{\Lambda,N,\beta}\left[\prod_{x \in \Lambda}\Phi_{x,N}^{p_x}\right]
=
\Big\langle \prod_{x \in \Lambda} \varphi_x^{p_x}
\Big \rangle_{\Lambda,\beta,h}.
 \end{equation*}
\end{prop}

\begin{rem}
Proposition \textup{\ref{prop: CV griffiths}} can be seen as a generalisation of the central limit theorem to a class of dependent random variables. The different rescaling factor $c_N \sim N^{-3/4}$ from the usual $N^{-1/2}$ scaling for independent random variables is intrinsically linked to the non-Gaussian limiting random variable. A more sophisticated viewpoint is that the coupling constants are chosen such that the $N \rightarrow \infty$ limit corresponds to a near-critical scaling limit of the Curie--Weiss model. Thus, the convergence of moments of the block-spin magnetisations to the $\phi^4$ moments can be seen as a statement about the model lying in a $\phi^4$ universality class. Indeed, generalisations of these results leading to densities with respect to $\phi^{2k}$ measures, i.e.\ distinct universality classes, can be found in \textup{\cite{ellisnewman1978}}.
\end{rem}

We strengthen the above convergence result for the specific case of a single block-spin. More precisely, we show that in this instance we can neglect diagonal terms. 
\begin{lem}\label{lem: GS} For any $p\in \mathbb N$ and $x\in \Lambda$, 
\begin{equation*}
    (c_N N)^p\mu^0_{x,N,\beta}\left[\prod_{j=1}^p\sigma_{(x,j)}\right]\underset{N\rightarrow \infty}\longrightarrow \left\langle\varphi^p\right\rangle_0.
\end{equation*}
\end{lem}
\begin{proof}
For $p=1$, there is nothing to show because both the sequence and $\left\langle\varphi\right\rangle_0$ are equal to $0$, so let us assume that $p\geq 2$. Let $N\geq p$. Expanding $\Phi_{x,N}^{p}$ and taking cases according to whether some index repeats or not we obtain
\begin{equation*}
\begin{aligned}
\mu^0_{x,N,\beta}\big[\Phi_{x,N}^{p}\big]=c_N^p\sum_{\substack{(i_1,\ldots,i_p)\in \lbrace 1,\ldots,N\rbrace^p\\ \forall k\neq \ell, \: i_k\neq i_\ell}}\mu^0_{x,N,\beta}\left[\prod_{j=1}^p\sigma_{(x,i_j)}\right] \\+c_N^p\sum_{\substack{(i_1,\ldots,i_p)\in \lbrace 1,\ldots,N\rbrace^p\\ \exists k<\ell, \: i_k=i_\ell}}\mu^0_{x,N,\beta}\left[\prod_{j=1}^p\sigma_{(x,i_j)}\right].
\end{aligned}
\end{equation*}
Note that
\begin{equation*}
\begin{aligned}
c_N^p\sum_{\substack{(i_1,\ldots,i_p)\in \lbrace 1,\ldots,N\rbrace^p\\ \forall k\neq \ell, \: i_k\neq i_\ell}}\mu^0_{x,N,\beta}\left[\prod_{j=1}^p\sigma_{(x,i_j)}\right] &=c^p_N p!\binom{N}{p}\mu^0_{x,N,\beta}\left[\prod_{j=1}^p\sigma_{(x,j)}\right] \\
&= (1+o(1)) (c_N N)^p\mu^0_{x,N,\beta}\left[\prod_{j=1}^p\sigma_{(x,j)}\right]
\end{aligned}
\end{equation*}
by the symmetries of the complete graph, and
\begin{equation*}
0\leq c_N^p\sum_{\substack{(i_1,\ldots,i_p)\in \lbrace 1,\ldots,N\rbrace^p\\ \exists k<\ell, \: i_k=i_\ell}}\mu^0_{x,N,\beta}\left[\prod_{j=1}^p\sigma_{(x,i_j)}\right]\leq c_N^2 {p \choose 2}  N \mu^0_{x,N,\beta}\big[\Phi_{x,N}^{p-2}\big]\underset{N\rightarrow \infty}{\longrightarrow} 0
\end{equation*}
where we have used that each summand is non-negative by Griffiths' inequality. The desired assertion then follows from Proposition \ref{prop: CV griffiths}.
\end{proof}

\subsubsection{A conditional proof}

We now introduce the random current representation for $\mu^h_{\Lambda,N,\beta}$. For ease of notation, we adopt the following conventions. First, we write $\Omega_{\Lambda_\fg}^N:=\Omega_{\cG_N}$ to denote the space of currents on $\cG_N$. Second, for $\tilde\n \in \Omega^N_{\Lambda_\fg}$, $x,y\in \Lambda$ and $j,k\in \lbrace 1,\ldots,N\rbrace$, we set 
\begin{equation*}
\tilde\n(x,j,y,k)
:=
\tilde\n\big(( x,j),( y,k)\big),\qquad \tilde\n(x,j,\fg):=\n\big(( x,j),\fg\big).
\end{equation*}
The random current expansion for $\mu^h_{\Lambda,\beta,N}$ takes the following form: for any $S\subset \Lambda_N$,
\begin{equation*}
    \mu_{\Lambda,N,\beta}^h[\sigma_S]=\frac{\sum_{\partial\tilde{\n}=S}w_{N,\beta}(\tilde\n)}{\sum_{\partial\tilde{\n}=\emptyset}w_{N,\beta}(\tilde\n)},
\end{equation*}
where
\begin{equation*}
\begin{aligned}
    w_{N,\beta}(\tilde\n)
    &:=
    \prod_{\substack{\lbrace x,y\rbrace\subset\Lambda\\j,k\in \lbrace 1,\ldots,N\rbrace}}\dfrac{(\beta c_N^2 J_{x,y})^{\tilde\n(x,j,y,k)}}{\tilde\n(x,j,y,k)!}
    \prod_{\substack{x\in \Lambda\\ j,k \in \lbrace 1,\ldots,N\rbrace}}\dfrac{( d_N)^{\tilde\n(x,j,x,k)}}{\tilde\n(x,j,x,k)!}
    \\
    &\qquad\qquad\quad \times 
    \prod_{(x,j)\in \Lambda_N}\dfrac{(\beta c_N h_x)^{\tilde\n(x,j,\fg)}}{\tilde\n(x,j,\fg)!}.
\end{aligned}    
\end{equation*}
Let $\Theta_N : \Omega_{\Lambda_\fg}^N\rightarrow \Omega_{\Lambda_\fg}$ be the projection map defined,
for all $\tilde{\n} \in \Omega_{\Lambda_\fg}^N$ and $\lbrace x,y\rbrace\subset\Lambda_\fg$, by
\begin{equation*}
\Theta_N(\tilde{\n})_{\lbrace x,y\rbrace}
:=
\sum_{j,k\in \lbrace 1,\ldots,N\rbrace}\tilde{\n}(x,j,y,k).
\end{equation*}
We then have that $\partial \Theta_N(\tilde\n)=\partial A$, where $A$ is defined for $x\in\Lambda$ by $A_x=|\partial \tilde\n\cap B_{x,N}|$.

Conversely, we can lift any $\n \in \Omega_{\Lambda_\fg}$ to a set of currents in $\Omega_{\Lambda_\fg}^N$ once we have fixed an injection of admissible moments $A\in \cM(\Lambda_\fg)$ to subsets of $\Lambda_N$. For convenience, we fix the natural injection: for $A\in \mathcal{M}(\Lambda_\fg)$, define
\begin{equation*}
    \tilde{A}
    :=
    \bigcup_{x\in \Lambda}\lbrace (x,j): \: 1\leq j \leq A_x\rbrace.
\end{equation*}
Then, the lift of $\n \in \Omega_{\Lambda_\fg}$ with source set $\partial A$ is the set of $\tilde \n \in \Omega_{\Lambda_\fg}^N$ such that $\Theta_N(\tilde \n) = \n$ and $\partial \tilde \n = \tilde A$. 

Motivated by this observation and Lemma \ref{lem: GS}, we define the renormalised weight for $\n \in \Omega_{\Lambda_\fg}$ by
\begin{equation*}
    \overline{w}_{N,A,\beta}(\n)
    :=
    \dfrac{(c_N N)^{|A|}}{\mathbf{Z}_N^{|\Lambda|}}\sum_{\substack{\tilde{\n}\in \Omega^N_{\Lambda_\fg}\\ \Theta_N(\tilde{\n})=\n \\ \partial \tilde{\n}=\tilde{A}}}w_{N,\beta}(\tilde{\n}),
\end{equation*}
where $|A|=\sum_{x\in \Lambda}A_x$ and
\begin{equation*}
\mathbf{Z}_N
:=
\sum_{\substack{\tilde{\n} \in \Omega_{B_{o,N}}\\ \sn=\emptyset}}w_{N,\beta}(\tilde{\n}).
\end{equation*}
We emphasise that the renormalisation factor $c_NN \sim N^{1/4}$ is directly linked to Proposition \ref{prop: CV griffiths}.

An important step in the proof of the switching principle is that the renormalised weights converge to those of the $\phi^4$ current expansion. To prove this, we need to understand the geometric structure of the currents that persists in the limit. It consists of currents which belong to the set $W_N$ defined below.

\begin{defn}\label{def: def w_N}
Let $W_N(\Lambda_\fg)=W_N:=W_{1,N}\cap W_{2,N}$, where
\begin{equation*}
\begin{aligned}
    W_{1,N}
    &:=
    \left\lbrace \tilde{\n}\in \Omega^N_{\Lambda_\fg}: \: \forall (x,j)\in \partial \tilde{\n}\setminus \lbrace \fg \rbrace, \: \forall e\in E^{\textup{ext}} \text{ with }(x,j)\in e, \: \tilde{\n}_e=0\right\rbrace,
    \\
    W_{2,N}
    &:=
    \Big\{ \tilde{\n}\in \Omega^N_{\Lambda_\fg}: \: \forall (x,j)\in \Lambda_N, \: \sum_{e \in E^{\textup{ext}}:\: (x,j)\in e}\tilde{\n}_e\leq 1\Big\}.
\end{aligned}    
\end{equation*}
\end{defn}
In words, the currents in $W_{1,N}$ consists of those whose sources are separated from exterior block-spins by other internal vertices; whereas $W_{2,N}$ consists of currents whose weight on external edges does not exceed $1$. We have the following result.
\begin{prop}\label{prop: stronger weight convergence}
For all $A\in \cM(\Lambda_\fg)$ and $\n\in \Omega_{\Lambda_\fg}$,
\begin{equation*}
    \dfrac{(c_N N)^{|A|}}{\mathbf{Z}_N^{|\Lambda|}}\sum_{\substack{\tilde{\n}\in W_N\\ \Theta_N(\tilde{\n})=\n \\ \partial \tilde{\n}=\tilde{A}}}w_{N,\beta}(\tilde{\n})\underset{N\rightarrow \infty}{\longrightarrow}w_{\beta}^A(\n),
\end{equation*}
and
\begin{equation}\label{eq: second half}
    \dfrac{(c_N N)^{|A|}}{\mathbf{Z}_N^{|\Lambda|}}\sum_{\substack{\tilde{\n}\in W_N^c\\ \Theta_N(\tilde{\n})=\n \\ \partial \tilde{\n}=\tilde{A}}}w_{N,\beta}(\tilde{\n})\underset{N\rightarrow \infty}{\longrightarrow}0.
\end{equation}
\end{prop}
\begin{proof}
See Section \ref{proof of stronger weight prop}.
\end{proof}
\begin{cor}
For all $A\in \mathcal{M}(\Lambda_\fg)$ and $\n \in \Omega_{\Lambda_\fg}$,
\begin{equation*}
\overline{w}_{N,A,\beta}(\n)
\underset{n\rightarrow \infty}\longrightarrow
w_{\beta}^A(\n).
\end{equation*}
\end{cor}

Guided by this intuition and Proposition \ref{prop: stronger weight convergence}, we isolate how these geometric structures intertwine in the statement of the usual switching lemma for $\mu^h_{\Lambda,\beta,N}$, which takes the following form. 

In what follows we fix $\Lambda'\subset\Lambda$ two finite subsets of $V\cup \lbrace \fg\rbrace$. This is slightly different from the convention we had above in which $\Lambda$ was a subset of $V$.
\begin{lem}
Let $S\subset \Lambda_N$, $T\subset \Lambda'_N$, and let $F:\Omega^N_{\Lambda_{\fg}} \rightarrow \RR$ be a bounded measurable function. Then,
\begin{equation}\label{eq: switching for GS}
\begin{aligned}
& \sum_{\substack{\tilde \n_1\in \Omega^N_{\Lambda}, \tilde\n_2 \in \Omega^N_{\Lambda'}  \\ \partial \tilde\n_1 = S, \partial \tilde\n_2 = T}}
    w_N(\tilde\n_1) w_N(\tilde\n_2) F(\tilde \n_1 + \tilde \n_2) 
    \\
    &\quad\quad\quad=
    \sum_{\substack{\tilde \n_1\in \Omega^N_{\Lambda}, \tilde\n_2 \in \Omega^N_{\Lambda'} \\ \partial \tilde\n_1 = S \Delta T, \partial \tilde\n_2 =\emptyset}}	
    w_N(\tilde\n_1) w_N(\tilde\n_2) F(\tilde \n_1 + \tilde \n_2)\mathbbm{1}[\mathcal{F}_T^{\Lambda'_N}],
\end{aligned}    
\end{equation}
where $\mathcal{F}_{T}^{\Lambda'}$ is the event that every cluster of the multigraph induced by $(\tilde \n_1+ \tilde \n_2)_{\Lambda'_N}$ intersects $T$ an even number of times. 
\end{lem}

If we proceed naively in expressing the left-hand side and the right-hand side in terms of projected currents and renormalised weights, we encounter a homogeneity issue. Indeed, let $\n_1\in \Omega_{\Lambda}$, and $\n_2 \in \Omega_{\Lambda'}$ with source sets $\partial A, \partial B$ associated to $A\in \cM(\Lambda)$, and $B\in \cM(\Lambda')$ respectively. Suppose we fix injections such that $\tilde A$ and $\tilde B$ may intersect. Then, the left-hand side of the switching lemma has a scaling factor $(c_N N)^{|A|+|B|}$, whereas the right-hand side has a scaling factor $(c_NN)^{|\tilde A \Delta \tilde B|}$, which is strictly smaller. 

To resolve this issue, we require a consistent convention (note that this choice is consistent with Remark \ref{rem: bijection (A,B) (A+B,0)}) for injecting the sources of fixed admissible moments, $A$ and $B$, in such a way that their lifts $\tilde A$ and $\tilde B$ satisfy $\tilde A \cap \tilde B = \emptyset$. The natural such injection of the ordered pair $(A,B)$ is given by
\begin{equation*}
\begin{aligned}
\tilde A
&:=
\bigcup_{x\in \Lambda}\lbrace (x,j):  1\leq j \leq A_x\rbrace,
\\
\tilde B
&:=
\bigcup_{x\in \Lambda}\lbrace (x,j):  A_x + 1\leq j \leq A_x + B_x\rbrace.
\end{aligned}
\end{equation*}
Consider the lefthand side of \eqref{eq: switching for GS} for $S=\widetilde{A}$ and $T=\widetilde{B}$. By conditioning on the underlying projected current in $\Omega_{\Lambda}$ or $\Omega_{\Lambda'}$, and then splitting the sum over $W_N$ and $W_N^c$, we obtain
\begin{equation*}
\begin{aligned}
\text{(LHS)} 
&:=
\sum_{\substack{\tilde \n_1\in \Omega^N_{\Lambda}, \tilde\n_2 \in \Omega^N_{\Lambda'} \\ \partial \tilde\n_1 = \widetilde{A}, \partial \tilde\n_2 = \widetilde{B}}}	
w_N(\tilde\n_1) w_N(\tilde\n_2) F(\tilde \n_1 + \tilde \n_2)
\\
&=
\sum_{\substack{\n_1\in \Omega_\Lambda, \:\partial \n_1 = \partial A\\\n_2\in \Omega_{\Lambda'},\: \partial\n_2 = \partial B}} 
\sum_{\substack{\Theta_N(\tilde \n_1) = \n_1, \Theta_N(\tilde\n_2) = \n_2 \\ \partial \tilde\n_1 = \widetilde{A}, \partial \tilde\n_2 = \widetilde{B}}}
w_N(\tilde\n_1) w_N(\tilde\n_2) F(\tilde \n_1 + \tilde \n_2) 
\\
&=
\sum_{\substack{\n_1\in \Omega_\Lambda, \:\partial \n_1 = \partial A\\\n_2\in \Omega_{\Lambda'},\: \partial\n_2 = \partial B}} (\Sigma_{1,N}(\n_1,\n_2) + \Sigma_{2,N}(\n_1,\n_2)),
\end{aligned}
\end{equation*}
where
\begin{equation*}
\begin{aligned}
\Sigma_{1,N}(\n_1,\n_2)
&:=
\sum_{\substack{\tilde\n_1\in W_N(\Lambda), \tilde\n_2 \in W_N(\Lambda') \\ \Theta_N(\tilde\n_1) = \n_1, \Theta_N(\tilde\n_2) = \n_2 \\ \partial\tilde\n_1 = \widetilde{A}, \partial\tilde\n_2 = \widetilde{B}}}
w_N(\tilde\n_1)w_N(\tilde\n_2) F(\tilde\n_1 + \tilde\n_2),	
\\
\Sigma_{2,N}(\n_1,\n_2)
&:=
\sum_{\substack{\tilde\n_1\notin W_N(\Lambda) \textup{ or } \tilde\n_2 \notin W_N(\Lambda') \\ \Theta_N( \tilde\n_1) = \n_1, \Theta_N(\tilde\n_2) = \n_2 \\ \partial\tilde\n_1 = \widetilde{A}, \partial\tilde\n_2 = \widetilde{B}}}
w_N(\tilde\n_1)w_N(\tilde\n_2) F(\tilde\n_1 + \tilde\n_2).
\end{aligned}
\end{equation*}

As a consequence of Proposition \ref{prop: stronger weight convergence} , we find that (with the proper scaling) $\Sigma_{2,N}$ converges to $0$ as $N \rightarrow \infty$.
\begin{lem}
Let $F:\Omega_{\Lambda}^N \rightarrow \RR$ be bounded. Then, for any $A \in \cM(\Lambda)$, $B\in \cM(\Lambda')$ and $(\n_1,\n_2)\in \Omega_{\Lambda}\times \Omega_{\Lambda'}$ such that $(\sn_1,\sn_2)=(\partial A,\partial B)$, 
\begin{equation*}
\lim_{N \rightarrow \infty}\frac{(c_NN)^{|A|+|B|}}{\mathbf{Z}_N^{|\Lambda|+|\Lambda'|}} \Sigma_{2,N}(\n_1,\n_2)
=
0.
\end{equation*}
\end{lem}

\begin{proof}
Since $F$ is bounded,
\begin{equation*}
    \left|\frac{(c_NN)^{|A|+|B|}}{\mathbf{Z}_N^{2|\Lambda|+|\Lambda'|}} \Sigma_{2,N}(\n_1,\n_2)\right| 
    \leq 
    \| F \|_\infty \left( (1)+(2)\right),
\end{equation*}
where
\begin{equation*}
\begin{aligned}
(1)
&=
\frac{(c_NN)^{|A|}}{\mathbf{Z}_N^{|\Lambda|}} \sum_{\substack{\tilde\n_1 \in W_N^c(\Lambda) \\ \Theta_N(\tilde\n_1) = \n_1 \\ \partial \tilde \n_1 = \widetilde{A} }} w_N(\tilde \n_1)\frac{(c_NN)^{|B|}}{\mathbf{Z}_N^{|\Lambda'|}} \sum_{\substack{\tilde\n_2\in \Omega^N_{\Lambda'}\\ \Theta_N(\tilde\n_2) = \n_2 \\ \partial \tilde \n_2 = \widetilde{B} }} w_N(\tilde \n_2),
\\
(2)
&=
\frac{(c_NN)^{|A|}}{\mathbf{Z}_N^{|\Lambda|}} \sum_{\substack{\tilde\n_1 \in \Omega^N_{\Lambda} \\ \Theta_N(\tilde\n_1) = \n_1 \\ \partial \tilde \n_1 = \widetilde{A} }} w_N(\tilde \n_1)\frac{(c_NN)^{|B|}}{\mathbf{Z}_N^{|\Lambda'|}} \sum_{\substack{\tilde\n_2\in W_N^c(\Lambda')\\ \Theta_N(\tilde\n_2) = \n_2 \\ \partial \tilde \n_2 = \widetilde{B} }} w_N(\tilde \n_2).
\end{aligned}
\end{equation*}
The result then follows from rewriting the weights in terms of renormalised weights and applying Proposition \ref{prop: stronger weight convergence}.
\end{proof}

The above lemma implies that it is sufficient to restrict our attention to pairs of currents that belong to $W_N(\Lambda)\times W_N(\Lambda')$ in the switching lemma of \eqref{eq: switching for GS}. In order to understand the geometry of the sum of two such currents, we further decompose into internal and external components. 

\begin{defn}\label{def: def w^int}
Let
$
W^{\textup{ext}}_N(\Lambda)
=
\{ \tilde\n \in W_{2,N} : \tilde\n(x,i,x,j) = 0 \text{ for all $x \in \Lambda$} \}$, and
$
W^{\textup{int}}_N(\Lambda)
=
\{ \tilde\n \in \Omega_{\Lambda}^N : \tilde\n(x,i,y,j) = 0 \text{ if } x\neq y \}$.
\end{defn}
Note that any $\tilde\n \in W_N(\Lambda)$ admits the decomposition $\tilde\n = \tilde\n^{\textup{ext}} + \tilde\n^{\textup{int}}$ with $\tilde\n^{\textup{ext}} \in W^{\textup{ext}}_N(\Lambda)$ and $\tilde\n^{\textup{int}} \in W^{\textup{int}}_N(\Lambda)$ such that $\partial\tilde\n^{\textup{ext}} \cap \partial\tilde\n = \emptyset$ and $\partial\tilde\n^{\textup{int}} = \partial\tilde\n^{\textup{ext}} \cup \partial\tilde\n$. Furthermore, we can write $\tilde{\n}^{\textup{int}}=\sum_{x\in \Lambda} \tilde{\n}^{\textup{int}}(x)$, where $\tilde{\n}^{\textup{int}}(x)\in W^{\textup{int}}_{x,N}(\Lambda)$ for 
\begin{equation*}
W^{\textup{int}}_{x,N}(\Lambda)
=
\{ \tilde\n \in W^{\textup{int}}_N(\Lambda) : \tilde\n(y,i,z,j) = 0 \text{ for all $(y,z)\neq (x,x)$}\}.
\end{equation*}
These definitions extend to $\Lambda'$ and since a current on $\Omega_{\Lambda'}$ can naturally be extended to a current on $\Omega_{\Lambda}$, one has that  $W_{x,N}^{\textup{int}}(\Lambda')=W_{x,N}^{\textup{int}}(\Lambda)$ for all $x\in \Lambda'$.

In order to identify the scaling limit of $\Sigma_{1,N}$, we have to encode how the connectivity properties inside the block-spins intertwine when summing over the internal components. We formalise this as follows.

\begin{defn}\label{def: block current measures}
Let $N \in \NN$. For $S_x,T_x \in \NN$, denote by $\tilde S_x:= \{ (x,i) : 1 \leq i \leq S_x \}\subset B_{x,N}$ and $\tilde T_x:= \{ (x,i) : S_x+1 \leq i \leq S_x+T_x \}$. Let $\pi_{x,N}^{\tilde S_x}$ be the measure on $\Omega_{B_{x,N}}$ defined by
\begin{equation*}
    \pi_{x,N}^{\tilde S_x}(\tilde{\n}^x)
    :=
    (c_NN)^{S_x} \,\frac{w_N(\tilde\n^x)\mathbbm{1}_{\partial \tilde\n^x= \tilde S_x}}{\underset{\tilde\n^x \in W^{\textup{int}}_{x,N}, \partial\tilde\n^x = \emptyset}{\sum} w_N(\tilde\n^x)}.	
\end{equation*}
Furthermore, denote by $\pi_{x,N}^{\tilde S_x,  \tilde T_x}$ the pushforward of the product measure $\pi_{x,N}^{\tilde S_x} \otimes \pi_{x,N}^{\tilde T_x}$ on $\Omega_{B_{x,N}} \times \Omega_{B_{x,N}}$ by the mapping $(\n^x, \m^x) \mapsto \n^x + \m^x$.
\end{defn}
\begin{rem}\label{rem : total weight of pi measures}
Note that
\begin{equation*}
    \pi^{\tilde{S}_x}_{x,N}(\Omega_{B_{x,N}})=(c_NN)^{S_x}\mu^0_{x,N,\beta}\left[\prod_{i=1}^{S_x}\sigma_{(x,i)}\right].
\end{equation*}
\end{rem}
Let $A_x \in \NN$. Consider the following equivalence relation on $\{ \tilde\n^x \in \Omega_{B_{x,N}} : \partial\tilde\n_x = \tilde A_x\}$. We say that $\tilde{\n}^x$ and $\tilde{\m}^x$ are equivalent, and write $\tilde\n^x \simeq \tilde\m^x$, if for all $1\leq  i,j\leq A_x$,
\begin{equation*}
\{(x,i) \overset{\tilde{\n}^x}{\longleftrightarrow} (x,j)	\}
\Leftrightarrow
\{(x,i) \overset{\tilde{\m}^x}{\longleftrightarrow} (x,j)	\}.
\end{equation*}
Following the notations of Definition \ref{def: block current measures}, denote by $\tilde\rho_{x,N}^{\tilde{S}_x}$ and $\tilde\rho_{x,N}^{\tilde{S}_x,\tilde{T}_x}$ the induced measures corresponding to $\pi_{x,N}^{\tilde{S}_x}$ and $\pi_{x,N}^{\tilde{S}_x,\tilde{T}_x}$, respectively, on the quotient $\sigma$-algebra. Note that these measures are finite due to the finiteness of the underlying space. Let $\rho_{x,N}^{\tilde S_x}$ and $\rho_{x,N}^{\tilde S_x, \tilde T_x}$ denote the corresponding probability measures, which can be identified as probability distribution on even partitions of $K_{S_x}$ and $K_{S_x+T_x}$, respectively. 

\begin{lem}\label{lem: tightness of the single-site tanglings}
Let $S_x,T_x \in \NN$. The two sequences of probability measures $(\rho_{x,N}^{\tilde S_x})_{N \in \NN}$ and $(\rho_{x,N}^{\tilde S_x,\tilde T_x})$ are tight. In both cases, the convergence along subsequences is pointwise on events. We denote their respective limits by $\rho_x^{\tilde S_x}$ and $\rho_x^{\tilde S_x, \tilde T_x}$.
\end{lem}
\begin{proof}
Since the measures considered here are defined on a finite space, it suffices to prove that the measure of the entire space converges in each case. But this is a direct consequence of Remark \ref{rem : total weight of pi measures} and Lemma \ref{lem: GS}.
\end{proof}


\begin{defn}
If $P,Q$ are two partitions of a set $\cS$, we say that $P$ is coarser than $Q$, and write $P\succ Q$, if any element of $P$ can be written as a union of elements of $Q$.
\end{defn}
The following proposition is a direct consequence of the constructions above.
\begin{prop}[The measure $\rho^{\tilde{S}_x,\tilde{T}_x}_x$ stochastically dominates $\rho^{\tilde{S}_x}_x\sqcup\rho^{\tilde{T}_x}_x$] 
Let $S_x,T_x\in 2\NN$. 
There exists a coupling of $\rho^{\tilde{S}_x,\tilde{T}_x}_x$, $\rho^{\tilde{S}_x}_x$ and $\rho^{\tilde{T}_x}_x$ such that if $(X,Y,Z)\sim (\rho^{\tilde{S}_x,\tilde{T}_x}_x,\rho^{\tilde{S}_x}_x,\rho^{\tilde{T}_x}_x)$ and if $Y\sqcup Z$ is the partition whose partition classes are the partition classes of $Y$ and $Z$,
then 
$
    X \succ Y \sqcup Z, \text{almost surely.}
$
We write, 
\begin{equation}\label{eq: domination}
    \rho^{\tilde{S}_x,\tilde{T}_x}_x 
    \succ
    \rho^{\tilde{S}_x}_x\sqcup\rho^{\tilde{T}_x}_x.
\end{equation}
\end{prop}

\begin{defn}
Let $A \in \cM(\Lambda)$, $B\in \cM(\Lambda')$ and let $(\n_1, \n_2) \in \Omega_{\Lambda}\times \Omega_{\Lambda'}$ with $(\sn_1,\sn_2)=(\partial A,\partial B)$. Let $F:\Omega_{\Lambda} \rightarrow \RR$ be bounded and measurable with respect to the quotient $\sigma$-algebra induced by the equivalence relation $\simeq$. We define 
\begin{equation*} 
\tilde\rho^{A,B}_{\n_1,\n_2, N}[F]
:=
\frac {1}{|W_N^{\textup{ext}}(\n_1,\n_2; A,  B)|}
\Big(\sum_{(\tilde\n_1^{\textup{ext}}, \tilde\n_2^{\textup{ext}}) \in W_N^{\textup{ext}}(\n_1,\n_2;  A,  B)}
\bigotimes_{z \in \Lambda}
\tilde\rho_{z,N}^{\partial\tilde\n_1^{\textup{ext}}(z)\sqcup \widetilde{A_z}, \partial\tilde\n_2^{\textup{ext}}(x)\sqcup \widetilde{B_z} } 
\Big)
[ F ],
\end{equation*}
where
\begin{equation*}
W_N^{\textup{ext}}(\n_1,\n_2; A, B)
:=
\Bigg\{
(\tilde\n_1^{\textup{ext}}, \tilde\n_2^{\textup{ext}}) \in W^{\textup{ext}}(\Lambda)\times W^{\textup{ext}}(\Lambda') :
\begin{array}{c}
\Theta_N (\tilde\n_1^{\textup{ext}}) = \n_1, 
\, \Theta_N (\tilde\n_2^{\textup{ext}}) = \n_2 \\
\partial\tilde\n_1^{\textup{ext}} \cap \widetilde{A} = \emptyset, \,
\partial\tilde\n_2^{\textup{ext}} \cap \widetilde{B} = \emptyset
\end{array}
\Bigg\},
\end{equation*}
and we use the convention that $\tilde A_z,\tilde B_z = \emptyset$ if $A_z,B_z = 0$. We denote by $\rho_{\n_1,\n_2,N}^{A,B}$ the associated probability measures.
\end{defn}

\begin{lem}
Let $A \in \cM(\Lambda)$, $B\in \cM(\Lambda')$ and let $(\n_1, \n_2) \in \Omega_{\Lambda}\times \Omega_{\Lambda'}$ with $(\sn_1,\sn_2)=(\partial A,\partial B)$. Then, the sequence of probability measures $(\rho_{\n_1,\n_2,N}^{A,B})_{N \in \NN}$ is tight. 
\end{lem}
\begin{proof}
As in Lemma \ref{lem: tightness of the single-site tanglings}, this result is a consequence of the finiteness of the underlying probability space together with the convergence 
\begin{equation*}
    \tilde{\rho}_{\n_1,\n_2,N}^{A,B}[1] \underset{N\rightarrow \infty}\longrightarrow \prod_{z \in \Lambda}\left\langle \varphi^{\Delta\n_1(z) + A_z} \right\rangle_{0} \left\langle \varphi^{\Delta\n_2(z) + B_z} \right\rangle_{0}.
\end{equation*}
\end{proof}
\begin{cor}\label{cor:subseq}
There exists an increasing sequence $(N_k)_{k\geq 1}$ such that for any $A\in \mathcal{M}(\Lambda)$, $B\in \cM(\Lambda')$, and for any $(\n_1,\n_2)\in \Omega_{\Lambda}\times \Omega_{\Lambda'}$ satisfying $\sn_1=\partial A$ and $\sn_2=\partial B$, the sequence of measures $(\rho_{\n_1,\n_2,N_k}^{A,B})_{k\geq 1}$ weakly converges to a measure $\rho^{A,B}_{\n_1,\n_2}$.
\end{cor}

\begin{rem}
Although for our purposes subsequential limits suffice, identification of the full limit has recently been established \textup{\cite{krachun2023scaling}}.
\end{rem}

The following estimate on the cardinality of $W_N^{\textup{ext}}(\n_1,\n_2; A, B)$ will be useful throughout this section.

\begin{lem}\label{lem:card}
Let $A \in \cM(\Lambda)$, $B\in \cM(\Lambda')$ and let $(\n_1, \n_2) \in \Omega_{\Lambda}\times \Omega_{\Lambda'}$ with $(\sn_1,\sn_2)=(\partial A,\partial B)$. Then we have
\begin{equation*}
 |W_N^{\textup{ext}}(\n_1,\n_2; A, B)| =(1+o(1))\prod_{\lbrace x,y\rbrace\subset\Lambda}\frac{N^{2\n_1(x,y)+2\n_2(x,y)}}{\n_1(x,y)!\n_2(x,y)!}\prod_{x\in \Lambda}\frac{N^{\n_1(x,\fg)+\n_2(x,\fg)}}{\n_1(x,\fg)!\n_2(x,\fg)!}.
\end{equation*}
\end{lem}
\begin{proof}
It suffices to estimate the cardinality of $W_N^{\textup{ext}}(\n; A):=W_N^{\textup{ext}}(\n,0; A, 0)$. To this end, we introduce the following procedure. We first enumerate the elements of $\mathcal{P}_2(\Lambda)$. We will count inductively the number of lifts $\tilde{\n}$ of $\n$ which satisfy $\tilde{\n}\in W^{\textup{ext}}_N$ and $\partial \tilde{\n}\cap \tilde{A}=\emptyset$.

Starting from the first element $\{x,y\}$ according to this ordering and taking into account that we cannot use any point in $\tilde{A}$, we see that we have 
$
    \binom{N-A_x}{\n(x,y)}\binom{N-A_y}{\n(x,y)}\n(x,y)!
$
choices for the piece of current between the blocks $B_{x,N}$ and $B_{y,N}$ when $\{x,y\}\subset \Lambda$, and 
$
\binom{N-A_x}{\n(x,\fg)}
$
choices if $x\in \Lambda$ and $y=\fg$. We then proceed to the second element $\{z,w\}$. In that case, the only additional condition we need to take into account is that we have attributed $\n(x,y)$ vertices of $B_{x,N}$ and $B_{y,N}$ in the preceding step that we cannot re-use (because the lifts we construct belong to $W_{2,N}$). If $z,w \in \Lambda$, then we have
\begin{equation*}
    \binom{N-A_z-\mathbbm{1}_{z\in\{x,y\}}\n(x,y)}{\n(z,w)}\binom{N-A_w-\mathbbm{1}_{w\in\{x,y\}}\n(x,y)}{\n(z,w)}\n(z,w)!
\end{equation*}
choices for the piece of lift between $x$ and $z$. We can iterate this procedure to compute the number of suitable lifts of $\n$.

It follows from the above procedure that 
\begin{multline*}
    \prod_{\lbrace x,y\rbrace\subset\Lambda}\binom{N-\Delta\n(x)-A_x}{\n(x,y)}\binom{N-\Delta\n(y)-A_y}{\n(x,y)}\n(x,y)!\prod_{x\in \Lambda}\binom{N-\Delta\n(x)-A_x}{\n(x,\fg)}\leq \\|W_N^{\textup{ext}}(\n; A)|\leq \prod_{ \lbrace x,y\rbrace\in \Lambda}\binom{N}{\n(x,y)}^2\n(x,y)!\prod_{x\in \Lambda}\binom{N}{\n(x,\fg)},
\end{multline*}
hence 
\begin{equation}\label{eq: card}
|W_N^{\textup{ext}}(\n; A)|=(1+o(1))\prod_{\lbrace x,y\rbrace\subset\Lambda}\frac{N^{2\n(x,y)}}{\n(x,y)!}\prod_{x\in \Lambda}\frac{N^{\n(x,\fg)}}{\n(x,\fg)!},    
\end{equation}
as desired.
\end{proof}

In what follows, we fix a sequence $(N_k)_{k\geq 1}$ as in Corollary~\ref{cor:subseq}. 

\begin{lem} Let $A \in \cM(\Lambda)$, $B\in \cM(\Lambda')$ and let $(\n_1, \n_2) \in \Omega_{\Lambda}\times \Omega_{\Lambda'}$ with $(\sn_1,\sn_2)=(\partial A,\partial B)$. Then, along the subsequence $(N_k)_{k\geq 1}$,
\begin{equation*}
\lim_{k\to\infty}\frac{(c_{N_k}N_k)^{|A|+|B|}}{\mathbf{Z}_{N_k}^{|\Lambda|+|\Lambda'|}} \Sigma_{1,N_k}(\n_1,\n_2)
=
w_{\beta}^A(\n_1) w_{\beta}^B(\n_2) \rho^{A,B}_{\n_1,\n_2}[F].
\end{equation*}	
\end{lem}

\begin{proof}
By definition
 \begin{equation*}
     \frac{(c_N N)^{|A|+|B|}}{\mathbf{Z}_N^{|\Lambda|+|\Lambda'|}} \Sigma_{1,N}(\n_1,\n_2)=\frac{(c_N N)^{|A|+|B|}}{\mathbf{Z}_N^{|\Lambda|+|\Lambda'|}}\sum_{\substack{\tilde \n_1\in W_N(\Lambda), \tilde\n_2 \in W_N(\Lambda') \\ \Theta_N(\tilde\n_1) = \n_1, \Theta_N(\tilde\n_2) = \n_2 \\ \partial \tilde\n_1 = \widetilde{A}, \partial \tilde\n_2 = \widetilde{B}}}	w_N(\tilde\n_1) w_N(\tilde\n_2) F(\tilde \n_1 + \tilde \n_2).
 \end{equation*}
For $i=1,2$, write $\tilde{\n}_i=\tilde{\n}_i^{\textup{ext}}+\tilde{\n_i}^{\textup{int}}$ with $\tilde{\n}_i^{\textup{ext}}\in W^\textup{ext}_N$ satisfying $\partial\tilde{\n}_i^{\textup{ext}}\cap \partial\tilde{\n}_i=\emptyset$, and $\tilde{\n_i}^{\textup{int}}\in W^\textup{int}_N$. Then, if $\n_1^{\textup ext}, \tilde \n_2^{\textup ext} \in W^{\textup ext}_N$ are such that $\Theta_N( \tilde\n_1^{\textup ext}) = \n_1, \Theta_N( \tilde\n_2^{\textup ext})=\n_2$, one has
\begin{equation*}
     w_N(\tilde\n_1^{\textup{ext}})w_N(\tilde\n_2^{\textup{ext}})=\prod_{\lbrace x,y\rbrace\subset\Lambda}(\beta c_N^2J_{x,y})^{\n_1(x,y)+\n_2(x,y)}\prod_{x\in \Lambda}(\beta c_N h_x)^{\n_1(x,\fg)+\n_2(x,\fg)}.
\end{equation*}
As a consequence,
\begin{equation*}
\begin{aligned}
     \frac{(c_N N)^{|A|+|B|}}{\mathbf{Z}_N^{|\Lambda|+|\Lambda'|}} \Sigma_{1,N}(\n_1,\n_2)=\frac{|W_N^{\textup{ext}}(\n_1,\n_2; A, B)|}{(c_NN)^{|\n_1|+|\n_2|}}\prod_{\lbrace x,y\rbrace\subset\Lambda}(\beta c_N^2J_{x,y})^{\n_1(x,y)+\n_2(x,y)}\\ \cdot\prod_{x\in \Lambda}(\beta c_N h_x)^{\n_1(x,\fg)+\n_2(x,\fg)}\tilde{\rho}_{\n_1,\n_2,N}^{A,B}[F].
\end{aligned}     
\end{equation*}
Recall now Lemma~\ref{lem:card} and that
\begin{equation*}
    \tilde{\rho}^{A,B}_{\n_1,\n_2}[1]=\prod_{x\in \Lambda}\left\langle \varphi_x^{\Delta\n_1(x)+A_x}\right\rangle_0\left\langle \varphi_x^{\Delta\n_2(x)+B_x}\right\rangle_0.
\end{equation*}
The desired convergence along the subsequence $(N_k)_{k\geq 1}$ follows readily.
\end{proof}

Using the same method as above, we also obtain the following.

\begin{lem} Let $A \in \cM(\Lambda)$, $B\in \cM(\Lambda')$ and let $(\n_1, \n_2) \in \Omega_{\Lambda}\times \Omega_{\Lambda'}$ with $(\sn_1,\sn_2)=(\partial A,\partial B)$. Then, along the subsequence $(N_k)_{k\geq 1}$,
\begin{multline}
    \frac{(c_N N)^{|A|+|B|}}{\mathbf{Z}_N^{|\Lambda|+|\Lambda'|}}\sum_{\substack{\tilde\n_1\in W_N(\Lambda), \tilde\n_2 \in W_N(\Lambda') \\ \Theta_N( \tilde\n_1) = \n_1, \Theta_N(\tilde\n_2) = \n_2 \\ \partial\tilde\n_1 = \widetilde{A}\sqcup \widetilde{B}, \partial\tilde\n_2 = \emptyset}}w_N(\tilde\n_1)w_N(\tilde\n_2) F(\tilde\n_1 + \tilde\n_2)\mathbbm{1}_{\mathcal{F}_{\tilde{B}}^{\Lambda'_N}}\\\underset{}{\longrightarrow}	 w_{\beta}^{A+B}(\n_1)w_{\beta}^\emptyset(\n_2) \rho^{A+B,\emptyset}_{\n_1,\n_2} [ F[\overline{\mathcal{H}}^{A+B,\emptyset}(\n_1, \n_2, \ct)] \mathbbm{1}_{\cF_{\Lambda'}^{A+B}(B)} ],
\end{multline}
where $\mathcal{F}_{\Lambda'}^{A+B}(B)$ is the event defined in Theorem \textup{\ref{thm: switching lemma}}.
\end{lem}
We are now equipped to complete the conditional proof of Theorem \ref{thm: switching lemma}
\begin{proof}[Proof of Theorem \textup{\ref{thm: switching lemma}}]
We first apply \eqref{eq: switching for GS} for $F\mathbbm{1}_C$, where \begin{equation*}C:=\left\{\sum_{i=1}^N(\Delta\tilde\n_1^{\textup{ext}}(x,i)+\Delta\tilde\n_2^{\textup{ext}}(x,i))\leq M ,\: \forall x\in \Lambda_{\fg}\right\}
\end{equation*}
for some $M\geq 1$. Note that $F\mathbbm{1}_C$ is bounded. Using the two preceding lemmas and taking the limit along an appropriate sequence $(N_k)_{k\geq 1}$ gives
\begin{equation*}
\begin{aligned}
    &\sum_{\substack{\n_1\in \Omega_{\Lambda}:\partial \n_1 = \partial A \\\n_2\in \Omega_{\Lambda'}:\partial \n_2 = \partial B}} 
    \mathbbm{1}_{(\n_1,\n_2)\in C}w_{\beta}^A(\n_1) w_{\beta}^B(\n_2) \rho^{A,B}_{\n_1,\n_2}\big[ F[\overline{\mathcal{H}}^{A,B}(\n_1,\n_2,\ct)] \big]
    \\
    &=
    \sum_{\substack{\n_1\in \Omega_{\Lambda}:\partial \n_1 = \partial(A+B) \\\n_2\in \Omega_{\Lambda'}: \partial \n_2 = \emptyset}} 
    \mathbbm{1}_{(\n_1,\n_2)\in C}w_{\beta}^{A+B}(\n_1)w_{\beta}^\emptyset(\n_2)  \rho^{A+B,\emptyset}_{\n_1,\n_2} \big[ F[\overline{\mathcal{H}}^{A+B,\emptyset}(\n_1,\n_2,\ct)] \mathbbm{1}_{\cF_{\Lambda'}^{A+B}(B)} \big].
\end{aligned}    
\end{equation*}
Sending $M$ to infinity and using the dominated convergence theorem we obtain 
\begin{equation*}
\begin{aligned}
    \sum_{\substack{\n_1\in \Omega_{\Lambda}:\partial \n_1 = \partial A \\\n_2\in \Omega_{\Lambda'}: \partial \n_2 = \partial B}} 
    &w_{\beta}^A(\n_1) w_{\beta}^B(\n_2) \rho^{A,B}_{\n_1,\n_2}\big[ F[\overline{\mathcal{H}}^{A,B}(\n_1,\n_2,\ct)] \big]
    \\
    &=
    \sum_{\substack{\n_1\in \Omega_{\Lambda}:\partial \n_1 = \partial(A+B) \\\n_2\in \Omega_{\Lambda'}: \partial \n_2 = \emptyset}} 
    w_{\beta}^{A+B}(\n_1)w_{\beta}^\emptyset(\n_2)  \rho^{A+B,\emptyset}_{\n_1,\n_2} \big[ F[\overline{\mathcal{H}}^{A+B,\emptyset}(\n_1,\n_2,\ct)] \mathbbm{1}_{\cF_{ \Lambda'}^{A+B}(B)} \big].
\end{aligned}    
\end{equation*}

It remains to prove that for $A \in \cM(\Lambda)$, $B\in \cM(\Lambda')$ and $(\n_1, \n_2) \in \Omega_{\Lambda}\times \Omega_{\Lambda'}$ with $(\sn_1,\sn_2)=(\partial A,\partial B)$, the measure $\rho^{A,B}_{\n_1,\n_2}$ is a product measure. It suffices to show that $\tilde\rho^{A,B}_{\n_1,\n_2}$ is a product measure, and for that, we can assume that $F$ is bounded. We write
$\tilde{\rho}_{\n_1,\n_2,N}^{A,B}[F]= C_1+C_2,
$
where
\begin{equation*}
    C_1:=\frac{1}{|W_N^{\textup{ext}}(\n_1,\n_2; A,  B)|}\sum_{\substack{(\tilde\n_1^{\textup{ext}}, \tilde\n_2^{\textup{ext}}) \in S_1}}\Big(\bigotimes_{z \in \Lambda} \tilde\rho_{z,N}^{\partial\tilde\n_1^{\textup{ext}}(z)\sqcup \widetilde{A_z},\partial\tilde\n_2^{\textup{ext}}(x)\sqcup \widetilde{B_z}}\Big)\left[ F \right]
\end{equation*}
for $S_1:=\{(\tilde\n_1^{\textup{ext}}, \tilde\n_2^{\textup{ext}}) \in W^{\textup{ext}_N} \, : \, \Theta_N(\tilde\n_1^{\textup{ext}}) = \n_1, \Theta_N(\tilde\n_2^{\textup{ext}}) = \n_2,  \partial\tilde\n_1^{\textup{ext}} \cap \widetilde{A} = \emptyset, \partial\tilde\n_2^{\textup{ext}} \cap \widetilde{B} = \emptyset, (\partial \tilde{\n}_1^{\textup{ext}}\sqcup \tilde{A})\cap (\partial \tilde{\n}_2^{\textup{ext}}\sqcup \tilde{B})=\emptyset \}$,
and
\begin{equation*}
    C_2:=\frac{1}{|W_N^{\textup{ext}}(\n_1,\n_2; A,  B)|}\sum_{\substack{(\tilde\n_1^{\textup{ext}}, \tilde\n_2^{\textup{ext}}) \in S_2}}\Big(\bigotimes_{z \in \Lambda} \tilde\rho_{z,N}^{\partial\tilde\n_1^{\textup{ext}}(z)\sqcup \widetilde{A_z},\partial\tilde\n_2^{\textup{ext}}(x)\sqcup \widetilde{B_z}}\Big)\left[ F \right]
\end{equation*}
for $S_2:=\{(\tilde\n_1^{\textup{ext}}, \tilde\n_2^{\textup{ext}}) \in W^{\textup{ext}_N} \, : \, \Theta_N(\tilde\n_1^{\textup{ext}}) = \n_1, \Theta_N(\tilde\n_2^{\textup{ext}}) = \n_2,  \partial\tilde\n_1^{\textup{ext}} \cap \widetilde{A} = \emptyset, \partial\tilde\n_2^{\textup{ext}} \cap \widetilde{B} = \emptyset, (\partial \tilde{\n}_1^{\textup{ext}}\sqcup \tilde{A})\cap (\partial \tilde{\n}_2^{\textup{ext}}\sqcup \tilde{B})\neq \emptyset \}$.
Notice that because of the symmetries of the model, all the product measures appearing in $C_1$ are equal. Thus, for a fixed pair of currents $(\tilde\n_1^{\textup{ext}}, \tilde\n_2^{\textup{ext}})$ satisfying the conditions of the sum of the numerator in $C_1$, one finds,
\begin{equation*}
    C_1=\bigotimes_{z \in \Lambda}\left(\tilde\rho_{z,N}^{\partial\tilde\n_1^{\textup{ext}}(z)\sqcup \widetilde{A_z},\partial\tilde\n_2^{\textup{ext}}(x)\sqcup \widetilde{B_z}}\right)\left[F\right] \frac{1}{|W_N^{\textup{ext}}(\n_1,\n_2; A,  B)|}|S_1|.
\end{equation*}
Arguing as in the proof of Lemma~\ref{lem:card} we obtain
\begin{equation*}
|S_1| 
\geq \prod_{i=1}^2\prod_{\lbrace x,y\rbrace\subset\Lambda}\binom{N_{x,y}}{\n_i(x,y)}^2\n_i(x,y)!\prod_{x\in \Lambda}\binom{N_{x,\fg}}{\n_i(x,\fg)},
\end{equation*}
where $N_{x,y}=N-\Delta\n(x)-\Delta\n(y)-A_x-A_y-B_x-B_y$ for $\n=\n_1+\n_2$, hence
\begin{equation}\label{eq: wext_conv_1}
   \frac{1}{|W_N^{\textup{ext}}(\n_1,\n_2; A,  B)|}|S_1|\underset{N\rightarrow \infty}{\longrightarrow} 1,
\end{equation}
which implies that along $(N_k)_{k\geq 1}$ the term $C_1$ converges to a product measure. To handle $C_2$, notice that
\begin{equation*}
    |C_2|\leq \left(\prod_{z\in \Lambda}\max_{\substack{k\leq \Delta\n_1(z)+A_z\\ p\leq \Delta\n_2(z)+B_z}} \langle \phi_z^k \rangle_{0} \langle \phi_z^p \rangle_{0}+o(1) \right)\Vert F\Vert_\infty  \frac{1}{|W_N^{\textup{ext}}(\n_1,\n_2; A,  B)|}|S_2|.
\end{equation*} 
It follows from \eqref{eq: wext_conv_1} that 
\begin{equation*}
    \frac{1}{|W_N^{\textup{ext}}(\n_1,\n_2; A,  B)|}|S_2|= o(1),
\end{equation*}
hence $C_2$ tends to $0$ as $N$ tends to infinity.
\end{proof}

\subsubsection{Proof of Proposition \ref{prop: stronger weight convergence}}\label{proof of stronger weight prop}
We now turn to the proof of Proposition \ref{prop: stronger weight convergence}. As explained above, we will argue that the main contribution in the weights $\overline{w}_{N,A,\beta}(\n)$ comes from currents $\tilde \n$ which belong to the set $W_N$ introduced in Definition \ref{def: def w_N}. We start by showing that this contribution admits an explicit limit, which is this first part of Proposition \ref{prop: stronger weight convergence}.
\begin{lem}\label{lem: lemma convergence on w_N} For all $A\in \mathcal{M}(\Lambda_\fg)$ and $\n \in \Omega_{\Lambda_\fg}$, 
\begin{equation*}
    \dfrac{(c_N N)^{|A|}}{\mathbf{Z}_N^{|\Lambda|}}\sum_{\substack{\tilde{\n}\in W_N\\ \Theta_N(\tilde{\n})=\n \\ \partial \tilde{\n}=\tilde{A}}}w_{N,\beta}(\tilde{\n})\underset{N\rightarrow \infty}\longrightarrow w_{\beta}^A(\n).
\end{equation*}
\end{lem}
\begin{proof}
Fix $A\in \mathcal{M}(\Lambda_\fg)$ and $\n\in \Omega_{\Lambda_\fg}$. Notice that for every $\tilde{\n}\in W_N$, we have a unique decomposition $\tilde{\n}=\tilde{\n}_1+\tilde{\n}_2$ with $\tilde{\n}_1\in W^{\textup{int}}_N$ and $\tilde{\n}_2\in W^{\textup{ext}}_N$ (these sets were introduced in Definition \ref{def: def w^int}). This yields the factorization
$
    w_{N,\beta}(\tilde{\n})=w_{N,\beta}(\tilde{\n}_1)w_{N,\beta}(\tilde{\n}_2).
$    
Now write,
\begin{align*}
    \sum_{\substack{\tilde{\n}\in W_N\\ \Theta_N(\tilde{\n})=\n \\ \partial \tilde{\n}=\tilde{A}}}w_{N,\beta}(\tilde{\n})&=\sum_{\substack{\tilde{\n}\in \Omega_{\Lambda_\fg}^N\\ \Theta_N(\tilde{\n})=\n \\ \partial \tilde{\n}=\tilde{A}}}w_{N,\beta}(\tilde{\n}_1)w_{N,\beta}(\tilde{\n}_2)\mathbbm{1}_{\tilde{\n}_1+\tilde{\n}_2\in W_N}\\ &=\sum_{\substack{\tilde{\n}_2\in W^{\textup{ext}}_N\\ \Theta_N(\tilde{\n}_2)=\n\\ \partial \tilde{\n}_2\cap \tilde{A}=\emptyset}}w_{N,\beta}(\tilde{\n}_2)\sum_{\substack{\tilde{\n}_1\in W^{\textup{int}}_N\\ \partial \tilde{\n}_1=\tilde{A}\sqcup \partial \tilde{\n}_2}}w_{N,\beta}(\tilde{\n}_1).
\end{align*}
Notice that for $\tilde{\n}_2\in W_N^{\textup{ext}}$ satisfying $\Theta_N(\tilde{\n}_2)=\n$, and $\partial \tilde{\n}_2\cap \tilde{A}=\emptyset$,
\begin{equation*}
    \dfrac{1}{\mathbf{Z}_N^{|\Lambda|}}\sum_{\substack{\tilde{\n}_1\in W^{\textup{int}}_N\\ \partial \tilde{\n}_1=\tilde{A}\sqcup \partial \tilde{\n}_2}}w_{N,\beta}(\tilde{\n}_1)=\prod_{x\in \Lambda}\mu^0_{x,N}\left[\prod_{j=1}^{A_x+\Delta\n(x)}\sigma_{(x,j)}\right].
\end{equation*}
Using this we get,
\begin{multline}\label{eq:intermediate}
    \dfrac{(c_N N)^{|A|}}{\mathbf{Z}_N^{|\Lambda|}}\sum_{\substack{\tilde{\n}\in W_N\\ \Theta_N(\tilde{\n})=\n \\ \partial \tilde{\n}=\tilde{A}}}w_{N,\beta}(\tilde{\n})=(c_N N)^{|A|}\prod_{x\in \Lambda}\mu^0_{x,N}\left[\prod_{j=1}^{A_x+\Delta\n(x)}\sigma_{(x,j)}\right]\sum_{\substack{\tilde{\n}_2\in W^{\textup{ext}}_N\\ \Theta_N(\tilde{\n}_2)=\n\\ \partial \tilde{\n}_2\cap \tilde{A}=\emptyset}}w_{N,\beta}(\tilde{\n}_2).
\end{multline}
Now, recall that $W^\textup{ext}_N\subset W_{2,N}$ so that for $\tilde{\n}_2\in W^{\textup{ext}}_N$ with $\Theta_N(\tilde{\n}_2)=\n$, one has
\begin{equation*}
    w_{N,\beta}(\tilde{\n}_2)=\prod_{\lbrace x,y\rbrace\subset\Lambda}(\beta c_N^2J_{x,y})^{\n_{x,y}}\prod_{x\in \Lambda}(\beta c_N h_x)^{\n_{x,g}}.
\end{equation*}
This implies for the left hand side of \eqref{eq:intermediate} that
\begin{multline*}
    \text{(LHS)}=(c_N N)^{|A|}\prod_{x\in \Lambda}\mu^0_{x,N}\left[\prod_{j=1}^{A_x+\Delta\n(x)}\sigma_{(x,j)}\right]\prod_{\lbrace x,y\rbrace\subset\Lambda}(\beta c_N^2J_{x,y})^{\n_{x,y}}\\ \times\prod_{x\in \Lambda}(\beta c_N h_x)^{\n_{x,g}}|\lbrace \tilde{\n}_2\in W^{\textup{ext}}_N: \: \Theta_N(\tilde{\n}_2)=\n, \: \partial \tilde{\n}_2\cap \tilde{A}=\emptyset\rbrace|.
\end{multline*}
By \eqref{eq: card},
\begin{equation}\label{number of currents in w2N}
    |\lbrace \tilde{\n}_2\in W^{\textup{ext}}_N: \: \Theta_N(\tilde{\n}_2)=\n, \: \partial \tilde{\n}_2\cap \tilde{A}=\emptyset\rbrace|=(1+o(1))\prod_{\lbrace x,y\rbrace\subset\Lambda}\frac{N^{2\n_{x,y}}}{\n_{x,y}!}\prod_{x\in \Lambda}\frac{N^{\n_{x,\fg}}}{\n_{x,\fg}!},
\end{equation}
hence
\begin{multline*}
    \text{(LHS)}=(1+o(1))\prod_{\lbrace x,y\rbrace\subset\Lambda}\frac{(\beta J_{x,y})^{\n_{x,y}}}{\n_{x,y}!}\prod_{x\in \Lambda}\frac{(\beta h_x)^{\n_{x,\fg}}}{\n_{x,\fg}!}\\\times\prod_{x\in \Lambda}(c_N N)^{A_x+\Delta\n(x)}\mu^0_{x,N}\left[\prod_{j=1}^{A_x+\Delta\n(x)}\sigma_{(x,j)}\right].
\end{multline*}
Using Lemma \ref{lem: GS}, we obtain that
\begin{multline*}
    \dfrac{(c_N N)^{|A|}}{\mathbf{Z}_N^{|\Lambda|}}\sum_{\substack{\tilde{\n}\in W_N\\ \Theta_N(\tilde{\n})=\n \\ \partial \tilde{\n}=\tilde{A}}}w_{N,\beta}(\tilde{\n})\underset{N\rightarrow \infty}\longrightarrow \prod_{\lbrace x,y\rbrace\subset\Lambda}\frac{(\beta J_{x,y})^{\n_{x,y}}}{\n_{x,y}!}\prod_{x\in \Lambda}\frac{(\beta h_x)^{\n_{x,\fg}}}{\n_{x,\fg}!}\prod_{x\in \Lambda}\left\langle \varphi^{A_x+\Delta\n(x)}\right\rangle_0
\end{multline*}
which is the desired result.
\end{proof}

We now prove the second part of Proposition \ref{prop: stronger weight convergence} by induction on $|\n|$, namely \eqref{eq: second half}. We only write the proof in the special case $h=0$ in which the ghost $\fg$ plays no role. It is easy to extend the result to the general case $h\geq 0$ (see Remark~\ref{remark case h>0} below).

\begin{proof}[Proof of second half of Proposition \textup{\ref{prop: stronger weight convergence}}]
Recall that $h=0$. We proceed by induction on $|\n|$. For $k\geq 0$, introduce $\mathbf{H}_k$ : For every $\n \in \Omega_{\Lambda_\fg}$ with $|\n|=k$, and every $A \in \mathcal{M}(\Lambda)$, 
\begin{equation*}
    \overline{w}_{N,A,\beta}(\n)\underset{N\rightarrow \infty}\longrightarrow w_{\beta}^A(\n).
\end{equation*}
First, if $|\n|=0$, then, proceeding as in the proof of Lemma \ref{lem: lemma convergence on w_N}, we get 
\begin{equation}\label{eq: convergence weights proof 1}
    \overline{w}_{N,A,\beta}(\n)=\prod_{x\in \Lambda}(c_N N)^{A_x+\Delta\n(x)}\mu_{x,N}^0\left[\prod_{j=1}^{A_x+\Delta\n(x)}\sigma_{(x,j)}\right],
\end{equation}
and the right-hand side in \eqref{eq: convergence weights proof 1} converges to $w_{A,\beta}(\n)$ by Lemma \ref{lem: GS}.

Let $k\geq 0$ and assume that $\mathbf{H}_i$ are true for all $0\leq i \leq k$. Let $\n \in \Omega_{\Lambda}$ such that $|\n|=k+1$. By Lemma \ref{lem: lemma convergence on w_N}, it is sufficient to prove that $\alpha_1(N)$ and $\alpha_2(N)$ vanish where
\begin{equation*}
    \alpha_i(N):=\dfrac{(c_N N)^{|A|}}{\mathbf{Z}_N^{|\Lambda|}}\sum_{\substack{\tilde{\n}\in \Omega^N_{\Lambda}\setminus W_{i,N}\\ \Theta_N(\tilde{\n})=\n \\ \partial \tilde{\n}=\tilde{A}}}w_{N,\beta}(\tilde{\n}).
\end{equation*}
We first prove that $\alpha_1(N)\rightarrow 0$. Notice that 
\begin{equation*}
    \Omega^N_{\Lambda}\setminus W_{1,N}=\bigcup_{\substack{u,v\in \Lambda\\ u\neq v}}\bigcup_{1\leq i \leq A_u}\bigcup_{1\leq j\leq N} V_{i,j}(u,v),
\end{equation*}
where 
$
    V_{i,j}(u,v):=\lbrace \tilde{\n}\in \Omega^N_{\Lambda}, \: \tilde{\n}(u,i,v,j)\geq 1\rbrace.
$    
Then, by substracting $1$ to the weight of the edge $\lbrace (u,i),(v,j)\rbrace$ for every $\tilde{\n} \in V_{i,j}(u,v)$, we get
\begin{eqnarray*}
    \alpha_1(N)&\leq & \sum_{u\neq v}\sum_{i=1}^{A_u}\sum_{j=1}^N\dfrac{(c_N N)^{|A|}}{\mathbf{Z}_N^{|\Lambda|}}\sum_{\substack{\tilde{\n}\in V_{i,j}(u,v)\\ \Theta_N(\tilde{\n})=\n \\ \partial \tilde{\n}=\tilde{A}}}w_{N,\beta}(\tilde{\n})\\ &\leq&\sum_{u\neq v} \sum_{i=1}^{A_u}\sum_{j=1}^N \beta J_{u,v} c_N^2  \dfrac{(c_N N)^{|A|}}{\mathbf{Z}_N^{|\Lambda|}}\sum_{\substack{\tilde{\m}\in \Omega_{\Lambda}^N\\ \Theta_N(\tilde{\m})=\n^{-} \\ \partial \tilde{\m}=\tilde{A}_{i,j}}}w_{N,\beta}(\tilde{\m})
\end{eqnarray*}
where $\n^{-}$ is the current satisfying $\n^-_e=\n_e$ for all edges $e\neq uv$, and $\n^-_{uv}=\n_{uv}-1$; and $\tilde{A}_{i,j}=\tilde{A}\Delta \lbrace (u,i),(v,j)\rbrace$. Notice that $\tilde{A}_{i,j}=\tilde{A}\setminus \lbrace (u,i),(v,j)\rbrace$ if $1\leq j \leq A_v$, and $\tilde{A}_{i,j}=\left(\tilde{A}\setminus \lbrace (u,i)\rbrace\right)\cup \lbrace (v,j)\rbrace$ otherwise. With this in mind, we associate to these new sets of sources their corresponding elements in $\mathcal{M}$: define $A'$ from $A$ by letting $A'_{u}=A_u-1$, $A'_v=A_v-1$, and $A'_z=A_z$ otherwise; also define $A''$ from $A$ by letting $A_{u}''=A_u-1$, $A_v''=A_v+1$, and $A_z''=A_z$ otherwise. Then, using the induction hypothesis ($|\n^{-}|=|\n|-1$),  
\begin{eqnarray*}
    \alpha_{1}(N)&\leq& \sum_{\substack{u\neq v\\ 1\leq i \leq A_u}}\beta J_{u,v}c_N^2\left(\sum_{j=1}^{A_v}(c_NN)^2\overline{w}_{N,A',\beta}(\n^-)+\sum_{j=A_v+1}^N\overline{w}_{N,A'',\beta}(\n^-) \right)\\ &\leq & \sum_{\substack{u\neq v\\ 1\leq i \leq A_u}}\beta J_{u,v}c_N^2\left(\sum_{j=1}^{A_v}(c_NN)^2 O(1)+\sum_{j=A_v+1}^N O(1)\right)\\ &=& O(1/\sqrt{N})\underset{N\rightarrow \infty}\longrightarrow 0.
\end{eqnarray*}
We used the fact that $(c_NN)^2c_N^2\rightarrow 0$ and $c_N^2N\rightarrow 0$. Moreover, to bound the second sum in the right-hand side of the first inequality, we used the symmetries of the model to write that for all $A_v+1\leq j \leq N$, 
\begin{equation*}
    \dfrac{(c_N N)^{|A|}}{\mathbf{Z}_N^{|\Lambda|}}\sum_{\substack{\tilde{\m}\in \Omega_{\Lambda}^N\\ \Theta_N(\tilde{\m})=\n^{-} \\ \partial \tilde{\m}=\tilde{A}_{i,j}}}w_{N,\beta}(\tilde{\m})=\overline{w}_{N,A'',\beta}(\n^-).
\end{equation*}
We now prove that $\alpha_2(N)\rightarrow 0.$ The idea is essentially the same: we seek a non-zero weight between to different blocks that we diminish to apply the induction hypothesis. Notice that if $\tilde{\n} \in \Omega^N_{\Lambda_\fg}\setminus W_{2,N}$ then there exist $u\in \Lambda$ and $i \in \lbrace 1,\ldots,N\rbrace$, such that
\begin{equation*}
    \sum_{\substack{(u,i)\in e\\ e\not\subset B_{u,N}}}\tilde{\n}_e\geq 2.
\end{equation*}
There are two cases: either for some $(v,j)$ with $u\neq v$, one has $\tilde{\n}(u,i,v,j)\geq 2$ and in such case we can subtract $2$ to the weight of this edge (which does not change the set of sources) to gain a $(c_N^2)^2$ which is in competition with a factor $O(N^2)$, and hence vanishes; or we can find two distinct vertices $(v,j)$ and $(w,k)$ with $u\neq v,w$ such that $\tilde{\n}(u,i,v,j)=1$ and $\tilde{\n}(u,i,w,k)=1$ so that the deletion of these weights adds two new sources, $(v,j)$ and $(w,k)$, and gives a factor $(c_N^2)^2$ which is in competition with the factor $(c_NN)^{-2}O(N^3)=O(N^{5/2})$, and once again vanishes. More formally, define $X_{1,N}$ (resp. $X_{2,N}$) to be the subset of currents of $\Omega^N_{\Lambda_g}\setminus W_{2,N}$ that falls into the first case (resp. second case). Then, 
\begin{eqnarray*}
    \dfrac{(c_N N)^{|A|}}{\mathbf{Z}_N^{|\Lambda|}}\sum_{\substack{\tilde{\n}\in X_{1,N}\\ \Theta_N(\tilde{\n})=\n \\ \partial \tilde{\n}=\tilde{A}}}w_{N,\beta}(\tilde{\n})&\leq & \sum_{u\neq v}\sum_{1\leq i \leq N}\sum_{1\leq j \leq N}\dfrac{(c_N N)^{|A|}}{\mathbf{Z}_N^{|\Lambda|}}\sum_{\substack{\tilde{\n}(u,i,v,j)\geq 2\\ \Theta_N(\tilde{\n})=\n \\ \partial \tilde{\n}=\tilde{A}}}w_{N,\beta}(\tilde{\n})\\ &\leq& \sum_{u\neq v}\sum_{1\leq i,j \leq N} (\beta J_{u,v}c_N^2)^2\overline{w}_{N,A,\beta}(\n^{2-})\\ &=&O(N^2c_N^4)\underset{N\rightarrow \infty}\longrightarrow 0,
\end{eqnarray*}
where $\n^{2-}$ is the current  satisfying $\n^{2-}_e=\n_e$ for all edges $e\neq uv$ and $\n^{2-}_{u,v}=\n_{u,v}-2$. We have used the induction hypothesis on the last line (since $|\n^{2-}|=|\n|-2$).

Now, for the second case, bear in mind that if we can find, for $\tilde{\n} \in X_{2,N}$, $(u,i), (v,j)$ and $(w,k)$ satisfying the properties described above, with one of them belonging to $\partial \tilde{\n}$ then $\tilde{\n}$ belongs to $\Omega_{\Lambda}^N\setminus W_{1,N}$ and the contribution of such currents goes to zero since $\alpha_1(N)\rightarrow 0$. Thus, 
\begin{equation*}
\begin{aligned}
    \dfrac{(c_N N)^{|A|}}{\mathbf{Z}_N^{|\Lambda|}}\sum_{\substack{\tilde{\n}\in X_{2,N}\\ \Theta_N(\tilde{\n})=\n \\ \partial \tilde{\n}=\tilde{A}}}&w_{N,\beta}(\tilde{\n})\leq  \alpha_1(N)\\&+ \sum_{\substack{u,v,w\in \Lambda\\v\neq u\\ w\neq u}}\sum_{i=A_u+1}^N\sum_{j=A_v+1}^N\sum_{k=A_w+1}^N \dfrac{(c_N N)^{|A|}}{\mathbf{Z}_N^{|\Lambda|}}\sum_{\substack{\tilde{\n}(u,i,v,j)=1\\ \tilde{\n}(u,i,w,k)=1\\ \Theta_N(\tilde{\n})=\n \\ \partial \tilde{\n}=\tilde{A}}}w_{N,\beta}(\tilde{\n}).
\end{aligned}    
\end{equation*}
Notice that for $i\geq A_u+1$, $j\geq A_v+1$ and $k\geq A_w+1$,
\begin{equation*}
    \dfrac{(c_N N)^{|A|}}{\mathbf{Z}_N^{|\Lambda|}}\sum_{\substack{\tilde{\n}(u,i,v,j)=1\\ \tilde{\n}(u,i,w,k)=1\\ \Theta_N(\tilde{\n})=\n \\ \partial \tilde{\n}=\tilde{A}}}w_{N,\beta}(\tilde{\n})=\beta^2 J_{u,v}J_{u,w}c_N^4(c_NN)^{-2}\overline{w}_{N,B,\beta}(\n^{-,-})
\end{equation*}
where $B\in \mathcal{M}$ is defined from $A$ by only changing its values on $v$ and $w$ where $B_v=A_v+1$ and $B_w=A_w+1$, and $\n^{-,-}$ is obtained from $\n$ by reducing by $1$ the weights of the edges $uv$ and $uw$. We can once again apply the induction hypothesis to get
\begin{equation*}
    \dfrac{(c_N N)^{|A|}}{\mathbf{Z}_N^{|\Lambda|}}\sum_{\substack{\tilde{\n}\in X_{2,N}\\ \Theta_N(\tilde{\n})=\n \\ \partial \tilde{\n}=\tilde{A}}}w_{N,\beta}(\tilde{\n})\leq O(c_N^2N)\underset{N\rightarrow \infty}\longrightarrow 0.
\end{equation*}
We finally obtained $\alpha_2(N)\rightarrow 0$ which yields $\mathbf{H}_{k+1}$.
\end{proof}
\begin{rem}\label{remark case h>0} In the case where we add an external field, the preceding proof may be adapted using the following ideas. For the vanishing of $\alpha_1(N)$ above, adding the external field means considering the case where $v=\fg$ which adds a $O(c_N(c_NN))$ term (because in this case removing $1$ to the weight of the edge reduces by $1$ the number of sources in $\Lambda_N$). For $\alpha_2(N)$, in the first case, consider the situation where $v$ is equal to $\fg$ which turns a $O(N^2c_N^4)$ into a $O(Nc_N^2)$; and in the second case, consider the situation where $v$ or $w$ is equal to $\fg$ which turns a $O((c_N^2)^2N^{5/2})$ into a $O((c_N^3)(c_NN)^{-1}N^2)=O(c_N^2N)$ (because in this case we add exactly one source in $\Lambda_N$).
\end{rem}

\section{Positive probability of admissible partitions}

The goal of this section is to prove that any element of $\cT_{\n_1,\n_2}^{A,B}(z)$ has positive probability under $\rho^{A,B}_{z,\n_1,\n_2}$. In particular, this shows that the distribution of tanglings in each block is not deterministic. This of crucial importance for the proof of Theorem \ref{thm:main}.

Let $\Lambda$ be a finite subset of $V\cup\lbrace \fg\rbrace$. Let $A,B\in \cM(\Lambda)$, and consider a pair of currents $(\n_1,\n_2)$ on $\Lambda$ satisfying $(\sn_1,\sn_2)=(\partial A,\partial B)$. We write $\n=\n_1+\n_2$. Recall that up to relabelling the elements of $\cB_z^{A,B}(\n_1,\n_2)$, the measure $\rho_{z,\n_1,\n_2}^{A,B}$ only depends on $(\Delta\n_1(z),\Delta\n_2(z),A_z,B_z)$. The following proposition is the main result of this section.

\begin{prop}\label{prop: pos prob} 
For every integer $M\geq 1$, there exists a constant $\kappa =\kappa(M)>0$ such that for every $A$, $B$, $\n_1, \n_2$ as above, and $z\in \Lambda$ satisfying $\max(A_z,B_z,\Delta\n(z))\leq M$, we have
\begin{equation*}
    \rho_{z,\n_1,\n_2}^{A,B}(P)
    \geq
    \kappa,
\end{equation*}
for all $P\in \cT_{\n_1,\n_2}^{A,B}(z)$. In particular, 
\begin{equation*}
    \rho_{z,\n_1}^{A}(P)
    \geq
    \kappa,
\end{equation*}
for all $P\in \cT_{\n_1}^A(z)$.
\end{prop}

\begin{rem}
The occurrence of the partition consisting of one element, i.e.\ the partition for which all the points of the block are tangled together, plays the role of an insertion tolerance property inside the block $\cB_z$ of $\cH^{A,B}(\n_1,\n_2,\ct)$. We call this event $\textup{ECT(z)}$. Proposition \ref{prop: pos prob} implies, in particular, that $\textup{ECT}(z)$ occurs with positive probability. 
\end{rem}

Let us recall that $\rho_{z,\n_1, \n_2}^{A,B}$ is constructed as the weak subsequential limit of the probability measures $\rho^{\tilde{S},\tilde{T}}_{z, N}$ where $S=\Delta\n_1(z)+A_z$ and $T=\Delta\n_2(z)+B_z$ (recall that $\tilde{S}\cap \tilde{T}=\emptyset$ by our choice of projection, as explained in Definition \ref{def: block current measures}). This convergence is in fact strong, i.e.\ pointwise on events, since the underlying space is finite dimensional. The measures in the prelimit are themselves defined as the pushforward of the block-spin double current measures  $\PP_{K_N}^{\tilde{S},\tilde{T}}$ under the partitions on sources induced by the clusters of $\tilde \n_1 +\tilde \n_2 \sim \PP_{K_N}^{\tilde{S},\tilde{T}}$. 

A fundamental step to proving Proposition \ref{prop: pos prob} is understanding the large $N$ asymptotics of  the size of the cluster of a source. We then use this to prove two lemmas that are used in the proof of Proposition \ref{prop: pos prob}. The first one asserts that pairings occur with positive probability, where a {\it pairing} of $\tilde{S}$ is a partition of it into sets of cardinality $2$, and the second one can be seen as merging the clusters of the sources at finite cost. In the following and later on this in section, we make the dependency on $a$ explicit as it is of great importance.

\subsection{Large $N$ asymptotics of the size of the cluster}
 Let $S \in \NN$ be finite and recall that, for each $N\geq S$, $\tilde S$ is its natural injection into $K_N$. Given a vertex $x\in \tilde{S}$, we write $\mathcal{C}_x=\mathcal{C}_x(\tilde{\n})$ for the connected component of $x$ in $\tilde{\n}$.

\begin{prop}\label{prop:uni_size}
Let $S=2k$, $k>0$, and let $a_0 \in \RR$.
Then there exists a constant $C=C(k,a_0)>0$ such that for every $x\in \tilde{S}$ and every $N\geq 2k$, we have
\begin{equation*}
    \sup_{a\geq a_0}\mathbb{E}^{\tilde{S},\emptyset}_{K_N,a}
    \left[\frac{|\mathcal{C}_x|}{\sqrt{N}}\right]
    \leq C.
\end{equation*}
\end{prop}
\begin{proof}
Note that for every $N\geq 2k$, we trivially have
\begin{equation*}
    \sup_{a\in I}\mathbb{E}^{\tilde{S},\emptyset}_{K_N,a}
    \left[\frac{|\mathcal{C}_x|}{\sqrt{N}}\right]
    \leq \sqrt{N},
\end{equation*} 
thus it suffices to handle the case of large $N$.
Let $u\not\in \tilde{S}$. Then, for every $a\geq a_0$, we have
\begin{equation*}
\begin{aligned}
    \mathbb E^{\tilde{S},\emptyset}_{K_N,a}\left[\frac{|\mathcal{C}_x|}{\sqrt{N}}\right]
    &= \frac{1}{\sqrt{N}}\left((N-2k)\mathbb P^{\tilde{S},\emptyset}_{K_N,a}[x\longleftrightarrow u]+O(1)\right)\\ 
    &= \frac{N-2k}{\sqrt{N}}\frac{\langle \sigma_{\tilde{S}\triangle \{x,u\}}\rangle_{K_N,a} \langle \sigma_x \sigma_u\rangle_{K_N,a}}{\langle \sigma_{\tilde{S}} \rangle_{K_N,a}}+O(1/\sqrt{N})\\
    &\leq \frac{N-2k}{\sqrt{N}}\langle \sigma_x \sigma_u\rangle_{K_N,a_0} +O(1/\sqrt{N})\\
    &\underset{N\rightarrow \infty}{\longrightarrow}\sqrt{12g}\langle \varphi^2\rangle_{0,a_0}.
\end{aligned}    
\end{equation*}
In the first line, we used the symmetry of the model. In the second line, we used the random current representation and the usual switching lemma for the Ising model. In the third line, we used that $\langle \sigma_{\tilde{S}\Delta \{x,u\}}\rangle_{K_N,a}=\langle \sigma_{\tilde{S}} \rangle_{K_N,a}$ by symmetry (recall that $x \in \tilde S$ but $u \not\in \tilde S$), and the mononoticity of $a\mapsto \langle \sigma_x \sigma_u\rangle_{K_N,a}$, which follows from differentating and applying Griffiths' inequalities. Finally, in the last line we used Lemma~\ref{lem: GS}. The desired assertion follows readily. \end{proof}

\begin{rem}
From the proof above, one can show that $\EE^{\tilde S,\emptyset}_{K_n,a}\left[\frac{|\cC_x|}{(c_N N)^2}\right] \rightarrow \langle \phi^2 \rangle_{0,a}$ as $N \rightarrow \infty$. The scaling limit of the size of the rescaled cluster is obtained in \textup{\cite{krachun2023scaling}}.
\end{rem}

We now make an observation\footnote{This observation is not new, see \cite[Lemma 11.2]{A82}.} that is subtle and means that arguments involving exploration of clusters naturally change the underlying model. 
\begin{rem}\label{shift}
Let $u\in \RR$. The scaling limit of the Ising model on $K_{N-r_N(u)}$ with interaction $d_N(a):=\frac{1}{N}\left(1-\frac{2a}{\sqrt{12gN}}\right)$, where $r_N(u)=\lceil u\sqrt{N} \rceil$, coincides with the $\varphi^4$ single-site measure with parameter $a(u)=a+\sqrt{3g}u$ in place of $a$. Indeed, writing $\tilde{N}:=N-r_N(u)$, we have
\begin{equation*}
    d_N(a)=\frac{1}{\tilde{N}}\left(1-\frac{2(a+\sqrt{3g}u)}{\sqrt{12g\tilde{N}}}+O\left(\frac{1}{\tilde{N}}\right)\right).
\end{equation*}
\end{rem}

\subsection{Pairings occur with positive probability}

Let $\cP$ denote the (random) even partition of the source set $\tilde{S}$ induced by the cluster of $\tilde \n_1 + \tilde \n_2 \sim \PP_{K_N}^{\tilde{S},\emptyset}$. Recall that $a$ is a parameter of the single site measure.
\begin{lem}[Pairings happen with positive probability]\label{lem:pair}
Let $S=2k$, $k>0$, and let $I=[\ell,L]$ be a compact interval. Then, there exists a constant $c=c(I,k)>0$ such that for every pairing $P$ of $\tilde{S}$ and for every $N\geq 2k$,
\begin{equation*}
    \inf_{a\in I}\PP^{\tilde{S},\emptyset}_{K_N,a}[\cP=P]\geq c.
\end{equation*}
\end{lem}

\begin{proof}
We prove the assertion by induction on $k$. The case $k=1$ is trivial. Let us assume that the statement holds for some $k\geq 1$. We will prove it for $k+1$.

Note that for every $N\geq 2k+2$, there exists $\alpha_N>0$ such that 
$
    \inf_{a\in I}\PP^{\tilde{S},\emptyset}_{K_N,a}[\cP=P]\geq \alpha_N,
$
hence we can assume that $N$ is sufficiently large. To simplify notation, we identify $\tilde{S}$ with $\{1,2,\ldots,2k+2\}$.
Consider a parameter $a\in I$, and let us write $P_0=\lbrace \lbrace 1,2\rbrace,\ldots,\lbrace 2k+1,2k+2\rbrace\rbrace$.
Then, conditioning on $\mathcal{C}_x$, and letting $x\in\{1,2\}$,
\begin{equation*}
\PP^{\tilde{S},\emptyset}_{K_N,a}[\cP=P_0]
=\frac{1}{\langle \sigma_{\tilde{S}}\rangle_{K_N,a}}\sum_{\mathcal{S}\subset (K_N\setminus \tilde{S})\cup\{1,2\}}\frac{Z^{\tilde{S},\emptyset}_{K_N,a}[\mathcal{C}_{x}=\mathcal{S},\cP=P_0]}{Z^{\emptyset,\emptyset}_{K_N,a}}.
\end{equation*}
Notice that for $P_0'=\lbrace \lbrace 3,4\rbrace,\ldots,\lbrace 2k+1,2k+2\rbrace\rbrace$,
\begin{eqnarray*}
    Z^{\tilde{S},\emptyset}_{K_N,a}[\mathcal{C}_x=\mathcal{S},\cP=P_0]
    &=&Z^{\{1,2\},\emptyset}_{\mathcal{S},a}[\mathcal{C}_{x}=\mathcal{S}]Z^{\tilde{S}\setminus \lbrace 1,2\rbrace,\emptyset}_{K_N\setminus \mathcal{S},a}[\cP=P_0'] \\
    &=& Z^{\{1,2\},\emptyset}_{K_N,a}[\mathcal{C}_{x}=\mathcal{S}]\frac{Z^{\tilde{S}\setminus \lbrace 1,2\rbrace,\emptyset}_{K_N\setminus \mathcal{S},a}[\cP=P_0']}{Z^{\emptyset,\emptyset}_{K_N\setminus \mathcal{S},a}}.
\end{eqnarray*}
Using the above equalities, we obtain
\begin{equation}\label{partition}
\PP^{\tilde{S},\emptyset}_{K_N,a}[\cP=P_0]=\sum_{\mathcal{S}\subset (K_N\setminus \tilde{S})\cup\{1,2\}} Q_{\mathcal{S},N}(a)\PP^{\{1,2\},\emptyset}_{K_N,a}[\mathcal{C}_{x}=\mathcal{S}]\PP^{\tilde{S}\setminus \lbrace 1,2\rbrace,\emptyset}_{K_N\setminus \mathcal{S},a}[\cP=P'_0],
\end{equation}
where
\begin{equation*}
Q_{\mathcal{S},N}(a):=\frac{\langle \sigma_{1}\sigma_{2}\rangle_{K_N,a}\langle \sigma_{\tilde{S}\setminus \lbrace 1,2\rbrace}\rangle_{K_N\setminus \mathcal{S},a}}{\langle \sigma_{\tilde{S}}\rangle_{K_N,a}}.
\end{equation*}

Let us now consider a constant $C_0>0$ such that 
$
    \PP_{K_N,a}^{\{1,2\},\emptyset}[A]\geq \frac{1}{2} \text{ for every } a\in I,
$
where $A=A(N)=\{|\mathcal{C}_{x}|<C_0\sqrt{N}, \: (\mathcal{C}_{x}\setminus \{1,2\})\cap \tilde{S}=\emptyset\}$. The existence of such a constant follows from Proposition~\ref{prop:uni_size} and the symmetry of the model, provided that $N$ is large enough.
Using Griffiths inequality yields that for every $\mathcal{S}$ such that $|\mathcal{S}|< C_0\sqrt{N}$ and every $a\in I$,
\begin{equation*}
Q_{\mathcal{S},N}(a)\geq Q_N:=\frac{\langle \sigma_{1}\sigma_{2}\rangle_{K_N,L}\langle \sigma_{\tilde{S}\setminus \lbrace 1,2\rbrace}\rangle_{K_{N-\lceil C_0\sqrt{N} \rceil},L}}{\langle \sigma_{\tilde{S}}\rangle_{K_N,\ell}}.
\end{equation*}
Since every summand in \eqref{partition} is non-negative, using the above inequalities we obtain
\begin{equation*}
    \PP^{\tilde{S},\emptyset}_{K_N,a}[\cP=P_0]\geq \frac{1}{2}Q_N\inf_{b\in I} \inf_{ |\mathcal{S}|<C_0\sqrt{N}}\PP^{\tilde{S}\setminus \lbrace 1,2\rbrace,\emptyset}_{K_N\setminus \mathcal{S},b}[\cP=P'_0]
\end{equation*}
for every $a\in I$.

Using Remark~\ref{shift} and Lemma~\ref{lem: GS} we also obtain that
\begin{equation*}
    \lim_{N\rightarrow\infty}Q_N=\frac{\langle \varphi^2\rangle_{0,L}\langle \varphi^{2k}\rangle_{0,L+\sqrt{3g}C_0}}{\langle \varphi^{2k+2}\rangle_{0,\ell}}.
\end{equation*}
Finally, it follows from our induction hypothesis and Remark~\ref{shift} that
\begin{equation*}
   \inf_{b\in I} \inf_{|\mathcal{S}|<C_0\sqrt{N}}\PP^{\tilde{S}\setminus \lbrace 1,2\rbrace,\emptyset}_{K_N \setminus \mathcal{S},b}[\cP=P'_0]\geq c(k,I'),
\end{equation*}
where $I'=[\ell,L+\sqrt{3g}C_0]$. This concludes the proof.
\end{proof}

\subsection{Cluster merging at finite cost}

Given a partition $P=\{P_1,P_2,\ldots,P_n\}$ of $\tilde{S}$, we write $P^{i,j}$ for the partition of $\tilde{S}$ with $P_i$ and $P_j$ merged together, i.e.\ $P^{i,j}=(P\setminus \{P_i,P_j\})\cup \{P_i\cup P_j\}$.

\begin{lem}[Cluster merging at finite cost]\label{lem:merge}
Let $S=2k$, $k>0$. For every $\varepsilon\in (0,1)$, there exists $\delta=\delta(\varepsilon)>0$ such that the following holds for every $N$ large enough.
Assume that for a partition $P$ of $\tilde{S}$ we have
\begin{equation}\label{ECT: nondegen}
\PP_{K_N}^{\tilde{S},\emptyset}[\cP=P]
\geq
\varepsilon.
\end{equation} 
Then for every $i,j$ we have
\begin{equation}\label{ECT: gluing prob}
\PP_{K_N}^{\tilde{S},\emptyset} [\cP=P^{i,j}]
\geq 
\delta \PP_{K_N}^{\tilde{S},\emptyset}[\cP=P].
\end{equation}
\end{lem}

\begin{proof}
Consider a partition $P=\{P_1,P_2,\ldots,P_n\}$ satisfying our assumption \eqref{ECT: nondegen}. Fix $i,j\in \{1,2,\ldots,n\}$ distinct and let $r\in \{i,j\}$. We start by defining $\Pi(P_r)$ to be the (random) set of (vertex-)self-avoiding walks where each walk in 
$\Pi(P_r)$ traverses only edges with non-zero value in $\tilde\n:=\tilde\n_1+\tilde\n_2$, and starts and ends in $P_r$.
We write $\pi(P_r)$ to denote the set of vertices visited by the self-avoiding walks in $\Pi(P_r)$. It is possible that some vertices of $\mathcal{C}_{P_r}:=\bigcup_{x\in P_r}\mathcal{C}_x$ are not contained in $\pi(P_r)$, which can happen when $\tilde\n$ contains a loop starting and ending at some vertex in $\pi(P_r)$. 

We first show that with high probability $\pi(P_r)$ has size $\Omega(\sqrt{N})$. Indeed, let $m>0$ and $x,y\in P_r$, with $x\neq y$. For a self-avoiding walk $\gamma$ in $K_N$ which starts at $x$, ends at $y$, and has length $\gamma$, we write $\{x\overset{\gamma}{\longleftrightarrow}y\}$ for the event that $\tilde \n_1(e)>0$ for all edges $e$ of $\gamma$. When $\tilde{\n}_1\in \{x\overset{\gamma}{\longleftrightarrow}y\}$, the current 
\begin{equation*}
\tilde{\m}_1(e):= \begin{cases}
      \tilde{\n}_1(e)-1, & \text{for } e \text{ edge of } \gamma, \\
      \tilde{\n}_1(e), & \text{otherwise},
\end{cases}
\end{equation*}
has sources $\tilde{S}\setminus \{x,y\}$. Thus, using the inequality $\frac{1}{\tilde{\n}_1(e)}\leq \frac{1}{\tilde{\m}_1(e)}$ and arguing as in the proof of Proposition~\ref{prop:uni_size}, we obtain (for the single current measure)
\begin{equation*}
     \PP^{\tilde{S}}_{K_N}[x\overset{\gamma}{\longleftrightarrow}y]
     \leq d_N^m \frac{Z_{K_N}^{\tilde{S}\setminus\lbrace x,y\rbrace}}{Z^{\tilde{S}}_{K_N}} 
     =
     d_N^m\frac{\langle \sigma_{\tilde{S}\setminus \{x,y\} }\rangle_{K_N}}{\langle\sigma_{\tilde{S}}\rangle_{K_N}}
     \leq
     C_1d_N^m\sqrt{N},
\end{equation*}
where $d_N$ is defined in \eqref{eq: def c_N and d_N}, $m$ is the length of $\gamma$, and $C_1=C_1(k)>0$ is a constant.

Note that there are at most $N^{m-1}$ self-avoiding walks of length $m$ connecting $x$ and $y$ in $K_N$. Moreover, when $|\pi(P_r)|\leq m$, each $x\in P_r$ is connected to some $y\in P_r$ by a self-avoiding walk of length at most $m$. Since adding a second current cannot decrease $|\pi(P_r)|$, for every $1>\eta>0$  
\begin{equation*}
    \mathbb P^{\tilde{S},\emptyset}_{K_N}\left[\{\cP=P\}\cap\left(\{|\pi(P_i)|
    \leq 
    \eta \sqrt{N}\}\cup \{|\pi(P_j)|\leq \eta \sqrt{N}\}\right)\right] 
    \leq
    C_2d_N\eta N 
    \leq
    C_3\eta,
\end{equation*}
where $C_2=C_2(k)>0, C_3=C_3(k)>0$ are constants. Thus, because of our assumption \eqref{ECT: nondegen} there exists a constant $c>0$ such that 
\begin{equation}\label{size gamma}
\PP^{\tilde{S},\emptyset}_{K_N}[\{\cP=P\}\cap \{|\pi(P_i)|,|\pi(P_j)| \geq c\sqrt N \}]
\geq 
\frac 12 \PP^{\tilde{S},\emptyset}_{K_N}[\cP=P]
\end{equation}
for every $N\geq 2k$. 

Let $(\tilde \n_1, \tilde \n_2) \in \{ \cP=P \}\cap \{|\pi(P_i)|,|\pi(P_j)| \geq c\sqrt N \}$. For every $x \in \pi(P_i) \setminus \tilde{S}$ and $y \in \pi(P_j) \setminus \tilde{S}$, define $(\tilde\n_1^{x,y}, \tilde\n_2^{x,y})$ as follows. First, we identify $x$ and $y$ in the natural way, labelling the point obtained from the gluing by $x$. Also, we preserve the labelling of all other vertices. The gluing procedure gives rise to a pair of currents $(\tilde \n_1^{x,y}, \tilde \n_2^{x,y})$ which has the property that $\cP=P^{i,j}$. The caveat is that they are defined on $N-1$ vertices. To fix this, we add a new vertex, labelled $y$, and extend $\tilde\n_1^{x,y}$ and $\tilde\n_2^{x,y}$ in the natural way such that $y$ is of degree $0$ in their sum. 

We define $f$ as the multi-valued map which associates to each $(\tilde\n_1, \tilde\n_2)$ the set of all possible $(\tilde \n_1^{x,y}, \tilde \n_2^{x,y})$. Note that, since $|\pi(P_i)|,|\pi(P_j)| \geq c\sqrt N$, we have
\begin{equation*}
|f(\tilde \n_1, \tilde \n_2)| \geq (c\sqrt{N}-|\tilde{S}|)^2\geq c_1N,
\end{equation*} 
for some constant $c_1>0$, provided that $N$ is large enough. Moreover, since no new edges were added, $\PP^{\tilde{S},\emptyset}_{K_N}[(\tilde\n_1,\tilde\n_2)] = \PP^{\tilde{S},\emptyset}_{K_N}[(\tilde\n_1^{x,y}, \tilde\n_2^{x,y})]$. Thus, for every $(\tilde\n_1, \tilde\n_2) \in \{ \cP=P \}\cap \{|\pi(P_i)|,|\pi(P_j)| \geq c\sqrt N \}$,
\begin{equation*}
\PP_{K_N}^{\tilde{S},\emptyset}[(\tilde\n_1,\tilde\n_2)]
=
\sum_{(\tilde\n,\tilde\m)\in f(\tilde\n_1,\tilde\n_2)}\frac{\PP_{K_N}^{\tilde{S},\emptyset}[(\tilde\n,\tilde\m)]}{|f(\tilde\n_1,\tilde\n_2)|}
\leq
\sum_{(\tilde\n,\tilde\m)\in f(\tilde\n_1,\tilde\n_2)}\frac{\PP_{K_N}^{\tilde{S},\emptyset}[(\tilde\n,\tilde\m)]}{c_1 N}
\end{equation*}
which, by Lemma \ref{size gamma}, yields
\begin{equation}\label{ECT: conclusion}
\begin{aligned}
\frac 12 \PP_{K_N}^{\tilde{S},\emptyset}[\cP=P]
&\leq
\sum_{\substack{(\tilde\n_1,\tilde\n_2) \in \{ \cP=P \} \\ |\pi(P_i)|, |\pi(P_j)| \geq c\sqrt N}} \sum_{(\tilde\n, \tilde\m) \in f(\tilde\n_1, \tilde\n_2)} 
\frac{\PP_{K_N}^{\tilde{S},\emptyset}[(\tilde\n,\tilde\m)]}{c_1 N}
\\
&\leq
\sum_{(\tilde\n, \tilde\m) \in \{ \cP=P^{i,j}\} } \frac{|f^{-1}(\tilde \n, \tilde \m)|}{c_1 N} \PP^{\tilde{S},\emptyset}_{K_N}[(\tilde\n,\tilde\m)].
\end{aligned}
\end{equation}

We now estimate $|f^{-1}(\tilde \n, \tilde \m)|$. Let $(\tilde\n_1,\tilde\n_2) \in f^{-1}(\tilde \n, \tilde \m)$. Then, there exists $x \in \pi(P_i), y \in \pi(P_j)$ such that $\tilde\n = \tilde\n_1^{x,y}$ and $\tilde\m = \tilde\n_2^{x,y}$. Consider the number of self-avoiding walks between vertices of $P_i$ in $(\tilde\n,\tilde\m)$ and the number of self-avoiding walks between vertices of $P_j$ in $(\tilde\n,\tilde\m)$. Notice that \textit{both} numbers strictly decrease when we remove $x$ from $\cC_{P_i\cup P_j}$ and that no other point has this property. Since we do not relabel $x$ in the construction of $(\tilde\n_1^{x,y},\tilde\n_2^{x,y})$, the topological consideration above implies that this point is the same for any $(\tilde\n_3,\tilde\n_4) \in f^{-1}(\tilde\n,\tilde\m)$. See Figure~\ref{fig: merging}. Thus, $\mathcal{C}_{P_i}(\tilde\n_1, \tilde\n_2) = \mathcal{C}_{P_i}(\tilde\n_3,\tilde\n_4)$. 

Since the labels of all other vertices are preserved after gluing except for $y$, this means that any $(\tilde\n_3,\tilde\n_4) \in f^{-1}(\tilde\n,\tilde\m)$ can be mapped bijectively to $(\tilde\n_1,\tilde\n_2)$ by swapping the label of $y$ with a vertex of degree zero in $(\tilde\n_1,\tilde\n_2)$. Thus, $|f^{-1}(\tilde\n,\tilde\m)| \leq N$. Putting this into \eqref{ECT: conclusion} gives the desired result. 
\end{proof}

\begin{figure}[!tbp]
  \centering
    \includegraphics[width=.5\textwidth]{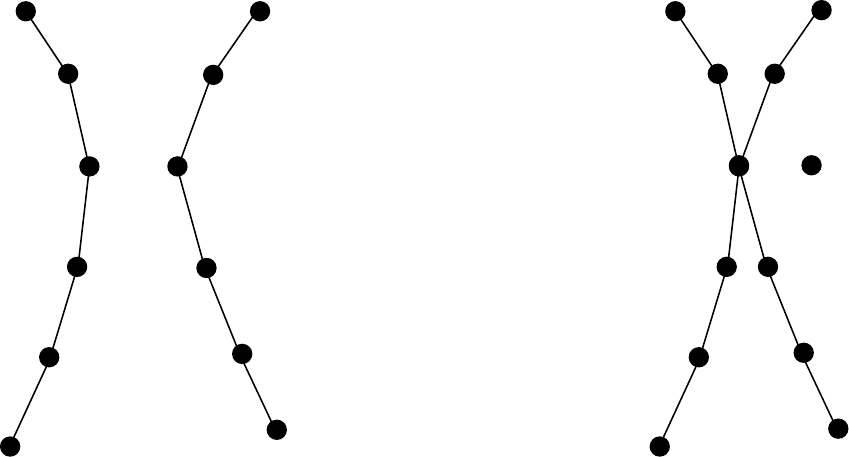}
    \put(-33,60){$x$}
    \put(-3,60){$y$}
    \put(-170,60){$x$}
    \put(-135,60){$y$}
    \put(-180,-13){$\mathcal{C}_{P_i}$}
    \put(-130,-13){$\mathcal{C}_{P_j}$}
    \put(-45,-13){$\mathcal{C}_{P_i}$}
    \put(-3,-13){$\mathcal{C}_{P_j}$}
    \caption{An example illustrating the merging of $\mathcal{C}_{P_i}$ and $\mathcal{C}_{P_j}$.}
    \label{fig: merging}
\end{figure} 

\subsection{Proof of Proposition \ref{prop: pos prob}}
We prove the estimate for the case of a double current; the case of a single current follows similarly. Let $P$ be an admissible partition of $\tilde{S}\sqcup \tilde{T}$. 
Note that,
\begin{equation*}
\begin{aligned}
    \PP_{K_N}^{\tilde{S},\tilde{T}}[P]=\frac{\langle \sigma_{\tilde{S}\sqcup \tilde{T}}\rangle_{K_N}}{\langle \sigma_{\tilde{S}}\rangle_{K_N}\langle \sigma_{\tilde{T}}\rangle_{K_N}}\PP_{K_N}^{\tilde{S}\sqcup \tilde{T},\emptyset}[P, \:\cF_{\tilde{T}}]
    =(1+o(1))\frac{\langle \varphi^{S+T}\rangle_{0}}{\langle \varphi^{S}\rangle_{0}\langle \varphi^{T}\rangle_{0}}
    \PP_{K_N}^{\tilde{S}\sqcup \tilde{T},\emptyset}[P]
\end{aligned}    
\end{equation*}
since any configuration $\tilde{\n}_1+\tilde{\n}_2\sim \PP_{K_N}^{\tilde{S},\tilde{T}}$ that realises $P$ immediately realises $\cF_{\tilde{T}}$ because of the admissibility constraint on $P$. This computation gives Lemmas \ref{lem:pair} and \ref{lem:merge} for the measure $\PP_{K_N}^{\tilde{S},\tilde{T}}$. As a consequence, there exists a constant $c>0$ such that, if $N$ is large enough and if $P$ is an admissible pairing of $\tilde{S}\sqcup \tilde{T}$, 
$\PP^{\tilde{S},\tilde{T}}_{K_N}[\cP=P]\geq c$. Note that by successively merging the partition classes of a pairing, we can obtain any admissible partition. Using Lemma~\ref{lem:merge}, we obtain that $\PP^{\tilde{S},\tilde{T}}_{K_N}[\cP=P]\geq c'$ uniformly in $N$ (sufficiently large) for a constant $c'>0$. Letting $N$ go to infinity, we obtain the desired result. $\hfill \square$

\section{Infinite volume measures} \label{sec: inf vol}
 
\subsection{Infinite volume distributions on currents}

Denote by $(\Omega_{V}, \mathcal{A}_V)$ the space of currents on $V$ equipped with the $\sigma$-algebra generated by cylinder events (i.e.\ events depending on finitely many edges). For $\Lambda \subset V$ finite, denote by $(\Omega_\Lambda, \cA_\Lambda)$ the space of currents on $\Lambda$ equipped with its Borel $\sigma$-algebra. There are natural injections $\Omega_\Lambda \hookrightarrow \Omega_{V}$ and $\mathcal{A}_\Lambda \hookrightarrow \mathcal{A}_V$. Thus, we canonically identify measures on $(\Omega_\Lambda, \cA_\Lambda)$ as measures on $(\Omega_{V}, \mathcal{A}_V)$. Similarly, denote by $(\Omega_{\Lambda_\fg}, \cA_{\Lambda_\fg})$ the space of currents on $\Lambda_\fg$ equipped with its Borel $\sigma$-algebra. We can also canonically identify measures on $(\Omega_\Lambda, \cA_\Lambda)$ as measures on $(\Omega_{V}, \mathcal{A}_V)$. We consider $\beta>0,h=0$ fixed in this section and drop from it from notation when convenient.

\begin{defn}
For $\Lambda \subset V$ finite, let $\mathsf{P}_\Lambda^\emptyset$ and $\mathsf{P}_{\Lambda_\fg}^\emptyset$ be the probability measures on $\Omega_\Lambda$ and $\Omega_{\Lambda_\fg}$, respectively, defined by
\begin{equation*}
\begin{aligned}
    \mathsf{P}_\Lambda^\emptyset[\n]
    =
    \frac{w_{\beta}^\emptyset(\n)}{\sum_{\substack{\m \in \Omega_\Lambda}}w_{\beta}^\emptyset(\m)}, \qquad
    \mathsf{P}_{\Lambda_\fg}^\emptyset[\n]
    =
    \frac{w_{\beta,h_\Lambda(\PLUS)}^\emptyset(\n)}{\sum_{\substack{\m \in \Omega_{\Lambda_\fg}}}w_{\beta, h_\Lambda(\PLUS)}^\emptyset(\m)},
\end{aligned}    
\end{equation*}
where $h_\Lambda(\PLUS)_x=\sum_{y\notin \Lambda}J_{x,y}\PLUS_x$ for $x \in \Lambda$.
We denote the expectations with respect to $\mathsf{P}_\Lambda^\emptyset$ and $\mathsf{P}_{\Lambda_\fg}^\emptyset$ by $\mathsf{E}^\emptyset_{\Lambda}$ and $\mathsf{E}^\emptyset_{\Lambda_\fg}$, respectively. Later, when clear from context, we will drop $h_\Lambda(\PLUS)$ from the notation.
\end{defn}

\begin{prop} \label{prop: inf vol flat} 
The probability measures $\mathsf{P}^\emptyset_{\Lambda}$ (resp. $\mathsf{P}^\emptyset_{\Lambda_\fg}$ ) converge weakly as $\Lambda \uparrow V$ to a probability measure $\mathsf{P}^0$ (resp. $\mathsf{P}^+$) on $\Omega_V$. The convergence is in fact strong for cylinder events, i.e.\ for any $A \in \cA_V$ depending on finitely many edges,
\begin{equation*}
\lim_{\Lambda \uparrow V} \mathsf{P}^\emptyset_{\Lambda}[A] = \mathsf{P}^0[A] 
\quad \text{ and } \quad
\lim_{\Lambda \uparrow V} \mathsf{P}^\emptyset_{\Lambda_\fg}[A] = \mathsf{P}^+[A].
\end{equation*}
Furthermore, $\mathsf{P}^0$ and $\mathsf{P}^+$ are invariant and mixing (hence, ergodic) under automorphisms in $\Gamma$\footnote{Recall we fix a vertex-transitive subgroup $\Gamma \subset \textup{Aut}(G)$.}.
\end{prop}

In order to prove the above proposition, we need the following lemma.
\begin{lem}\label{degree control}
Let $\Lambda\subset V$ finite. Then for every $x \in \Lambda$ we have
\begin{equation*}
    \mathbf{E}^{\emptyset,\emptyset}_{\Lambda_\fg,\Lambda}[\Delta\n_1(x)+\Delta\n_2(x)]=4g\Big(\langle\varphi_x^4\rangle_{\Lambda}^0+\langle\varphi_x^4\rangle_{\Lambda}^\PLUS\Big)+2a\Big(\langle\varphi_x^2\rangle_{\Lambda}^0+\langle\varphi_x^2\rangle_{\Lambda}^\PLUS\Big)-2.
\end{equation*}
\end{lem}
\begin{proof}
For $k\geq 0$, let $u_{2k}:=\langle \varphi^{2k}\rangle_{0}$.
An integration by parts yields that
\begin{equation*}
    (2k+1)u_{2k}=2a\cdot u_{2k+2}+4g\cdot u_{2k+4}.
\end{equation*}
Note that for $i\in \{1,2\}$,
\begin{equation*}
    \mathbf{E}^{\emptyset,\emptyset}_{\Lambda_\fg,\Lambda}[\Delta\n_i(x)]=\dfrac{\sum_{\sn_i=\emptyset}(\Delta\n_i(x)+1)w^{\emptyset}(\n_i)}{\sum_{\sn_i=\emptyset}w^{\emptyset}(\n_i)}-1,
\end{equation*}
and
\begin{eqnarray*}
    (\Delta\n_i(x)+1)w^{\emptyset}(\n_i)&=&w(\n_i)\prod_{y\neq x}u_{\Delta\n_i(y)}\cdot(\Delta\n_i(x)+1)u_{\Delta\n_i(x)}\\ &=& w(\n_i)\prod_{y\neq x}u_{\Delta\n_i(y)}(2a\cdot u_{\Delta\n_i(x)+2}+4g\cdot u_{\Delta\n_i(x)+4})\\ &=& 4g\cdot w^{\lbrace x,x,x,x\rbrace}(\n_i)+2a\cdot w^{\{x,x\}}(\n_i).
\end{eqnarray*}
The desired result follows.
\end{proof}

Without loss of generality, we only prove Proposition \ref{prop: inf vol flat} for $\mathsf{P}^0$. We first construct the limit on special cylinder events (which generate the $\sigma$-algebra) where both the bond parities and degrees of the current are locally specified. Fix $F\subset \mathcal{P}_2(V)$ finite. Let $V(F)$ be the set of $x \in V$ which belong to an edge in $F$. 
For $d_F\in (2\mathbb N)^{V(F)}$, we consider the following event 
\begin{equation*}
\cC_{F,d_F}:=
\left\{ 
\n \in \Omega_V \middle\vert
    \begin{array}{c}
    \n_e \text{ is even for all } e \in F, \text{ and } \\
    \Delta \n_F(x) = d_F(x)	 \text{ for all } x \in V(F)
  \end{array}
\right\},
\end{equation*}
where $\n_F$ is the restriction of $\n$ to pairs in $F$.

\begin{lem} \label{lem: flat measure cyl}
Let $F \subset \mathcal{P}_2(V)$ finite and $d_F \in V(F)^{2\NN}$. The probabilities $\mathsf{P}_{\Lambda}^\emptyset[\mathcal{C}_{F,d_F}]$ converge as $\Lambda \uparrow V$. We define $\mathsf{P}^0[\cC_{F,d_F}]$ to be this limit.
\end{lem}

\begin{proof}

Let $\Lambda \supset V(F)$. Define $\mathring{V}(F) \subset V(F)$ be the set of $x \in V(F)$ such that if $J_{x,y} > 0$ for $y \in V$, then $y \in V(F)$, i.e.\ vertices whose $J$-support are contained in $V(F)$, which might be empty if $J$ has infinite range.

For $\n \in \Omega_\Lambda$, define
$
    w(\n)
    :=
    \prod_{\lbrace x,y\rbrace \subset \Lambda}\frac{(\beta J_{x,y})^{\n_{x,y}}}{(\n_{x,y})!}
    \text{ and }
    Z_\Lambda^A:=\sum_{\sn=\partial A}w^A(\n).
$
We write $\n_{\Lambda \setminus F}$ to denote the restriction of $\n$ to edges in the complement of $F$ and denote $Z_{\Lambda \setminus F}^A$ for the corresponding partition function.

Observe that we can write, 
\begin{equation*}
\begin{aligned}
    w^\emptyset(\n)\mathbbm{1}\{\n \in\mathcal{C}_{F,d_F}\}
    &=
    w(\n_F)\mathbbm{1}\{\n_F\in \mathcal{C}_{F,d_F}\}  \prod_{x\in  \mathring{V}(F)}
\langle \varphi^{d_F(x)}\rangle_0 \times w(\n_{\Lambda\setminus F})
\\
&\quad\quad\quad \times 
\prod_{x\in \Lambda\setminus \mathring{V}(F)}\langle \varphi^{\Delta\n_{\Lambda\setminus F}(x)+d_F(x)\mathbbm{1}[x\in V(F)]}\rangle_0.
\end{aligned}
\end{equation*}
Therefore,
\begin{equation*}
\begin{aligned}
    \mathsf{P}_\Lambda^\emptyset[\mathcal{C}_{F,d_F}]
    &=
    \frac{\sum_{\n \in \Omega_\Lambda}w^\emptyset(\n)\mathbbm{1}[\n \in\mathcal{C}_{F,d_F}]\mathbbm{1}\{\sn=\emptyset\}}{Z^\emptyset_\Lambda}
    =
    \frac{Z^{d_F}_{\Lambda\setminus F}}{Z^\emptyset_\Lambda}f(F,d_F),
\end{aligned}    
\end{equation*}
where 
\begin{equation*}
    f(F,d_F)
    :=
    \prod_{x\in  \mathring{V}(F)}\langle \varphi^{d_F(x)}\rangle_0\sum_{\sn_F=\emptyset}w(\n_F)\mathbbm{1}\{\n_F\in \mathcal{C}_{F,d_F}\}.
\end{equation*}
As a result, we get
\begin{equation}\label{result}
    \mathsf{P}_\Lambda^\emptyset[\mathcal{C}_{F,d_F}]
    =
    \Big\langle e^{-K_F} \prod_{x\in V(F)\setminus \mathring{V}(F)}\varphi_x^{d_F(x)} \Big\rangle_{\Lambda,\beta}^0 \, f(F,d_F),
\end{equation}
where
$
    K_F
    =
    \beta\sum_{\lbrace x,y\rbrace\in F}J_{x,y}\varphi_x\varphi_y.
$
Note that by Young's inequality, for any $(\varphi_x)_{x\in V(F)}\in \mathbb R^{V(F)}$,
\begin{equation}\label{ke bound}
e^{-K_F}\prod_{x \in V(F) \setminus \mathring{V}(F)}\varphi_x^{d_F(x)} 
\leq
e^{\beta\big(\sup_{\lbrace x,y\rbrace \in F}J_{x,y} \big)\sum_{x \in V(F)} \varphi_x^2}  \prod_{x \in V(F) \setminus \mathring{V}(F)}\varphi_x^{d_F(x)} .
\end{equation}
The expectation of the right-hand side of \eqref{ke bound} is finite for $\langle \cdot \rangle_\beta^0$ (i.e.\ the infinite volume measure). Hence, the the left-hand side of \eqref{ke bound} is integrable with respect to $\langle \cdot \rangle_\beta^0$. Since the left-hand side is a local function, the convergence as $\Lambda \uparrow V$ follows by the DLR equations. The observations above yield that $\underset{\Lambda \uparrow V}{\lim}\mathsf{P}_\Lambda^\emptyset(\mathcal{C}_{F,d_F})$ exists and is equal to 
\begin{equation}\label{expression of the limit}
    \Big\langle  e^{-K_F}\prod_{x\in V(F)\setminus \mathring{V}(F)}\varphi_x^{d_F(x)}\Big\rangle_{\beta}^0 \:f(F,d_F).
\end{equation}
\end{proof}

\begin{proof}[Proof of Proposition \textup{\ref{prop: inf vol flat}}]
We first prove the existence of $\sfP^0$. Let $A \in \mathcal{A}_V$ be a cylinder event. Consider a sequence of finite sets of vertices $\Lambda$ converging to $V$. We show that $(\sfP_\Lambda^\emptyset[A])_{\Lambda}$ is Cauchy. The limit does not depend on the sequence since, by Lemma \ref{lem: flat measure cyl}, all limit points have the same values on events $\cC_{F,d_F}$, for any $F \subset \mathcal{P}_2(V)$ finite and $d_F \in V(F)^{2\NN}$.

Let $M \in \NN$, and consider a finite $F \subset \mathcal{P}_2(V)$ such that $A \in \cA_{V(F)}$. Note that the collection $\cC_F:=\{ \cC_{F',d_{F'}} : F' \subset F, d_{F'} \in V(F')^{2\NN} \}$ is stable under intersection. Thus, by inclusion--exclusion, there exists $N=N(M)>0$ such that
$
    \mathsf{P}^\emptyset_{\Lambda}\left[\cA \cap \mathcal{E}\right]
=
\sum_{n=1}^N p_{n,\Lambda},
$
where $\mathcal{E}=\left\lbrace \sup_{x \in V(F)}\Delta \n_F(x) \leq M\right\rbrace$, and $p_{n,\Lambda}$ are proportional to probabilities involving finite intersections and unions of events in $\cC_F$. By Lemma \ref{lem: flat measure cyl}, $p_{n,\Lambda} \rightarrow p_n$ as $\Lambda \uparrow V$. 

On the other hand, note that
\begin{equation*}
\begin{aligned}
\sup_{\Lambda : V(F) \subset \Lambda \subset V}\mathsf{P}^\emptyset_{\Lambda}\left[A \cap \mathcal{E}^c\right]
&\leq
\sup_{\Lambda : V(F) \subset \Lambda \subset V} \sum_{x \in V(F)} \mathsf{P}^\emptyset_{\Lambda}[ \Delta \n_E(x) > M]
\\
&\leq
\frac{1}M \sum_{x \in V(F)} \sup_{\Lambda : V(F) \subset \Lambda \subset V} \mathsf{E}^\emptyset_{\Lambda} [\Delta \n(x)]
\\
&\leq 
\frac{C(F,\beta)}M \underset{M \rightarrow \infty}{\longrightarrow} 0,
\end{aligned}
\end{equation*}
where the first inequality is by trivial inclusion and a union bound; the second inequality is by Markov's inequality; and the third inequality is by the moment bounds of $\Delta\n$ in Lemma \ref{degree control} together with the convergence of $\langle \varphi_x^{2k} \rangle_\Lambda^0$ as $\Lambda \uparrow V$. 

Hence, $(\mathsf{P}_{\Lambda}^\emptyset[A])_{\Lambda}$ is Cauchy. Define $\mathsf{P}^0 [A]$ to be this limit. Since cylinder events generate $\cA_V$, this extends to a unique probability measure on $\Omega_V$.

Now we establish the invariance under $\Gamma$. By \eqref{expression of the limit} and the fact that $\langle \cdot\rangle_{\beta}^0$ is invariant under $\Gamma$, we have that for all $x \in V$, $F\subset \mathcal{P}_2(V)$ finite, and $d_F\in V(F)^{2\mathbb N}$,
$
\mathsf{P}^0\left[\gamma_x\mathcal{C}_{F,d_F}\right]
    =
    \mathsf{P}^0\left[\mathcal{C}_{F,d_F}\right].
$
$\Gamma$-invariance then follows by a truncation argument as above.

In order to establish mixing, notice that (again by a truncation argument) it suffices to show that for two finite sets $F,F'\subset E$, 
\begin{equation}\label{ergod aim}
    \lim_{d_G(o,x)\rightarrow \infty}\mathsf{P}^0\left[\mathcal{C}_{F,d_F}\cap \gamma_x\mathcal{C}_{F',d_F'}\right]
    =
\mathsf{P}^0\left[\mathcal{C}_{F,d_F}\right]\mathsf{P}^0\left[\mathcal{C}_{F',d_{F'}}\right],
\end{equation}
where $\gamma_x \in \Gamma$ is such that $\gamma_x\cdot o = x$. Taking $x$ such that $d_G(o,x)$ is large enough, we notice that 
$
\mathcal{C}_{F,d_F}\cap \gamma_x\mathcal{C}_{F',d_{F'}}
=
\mathcal{C}_{F\cup(\gamma_x F'),d_F\cup d_{\gamma_x F'}}.
$
Thus,
\begin{equation}\label{ergod f}
     \lim_{d_G(o,x) \rightarrow \infty} f(F\cup(\gamma_x F'),d_F\cup d_{\gamma_x F'})
     =
     f(F,d_F)f(F',d_F').
\end{equation}

Moreover, by Proposition \ref{prop: free plus ergod}, the restriction of $\langle \cdot \rangle_{\beta}^0$ to even functions is ergodic under $\Gamma$. Hence.
\begin{equation}\label{ergod expectation}
\begin{aligned}
    \lim_{d_G(o,x)\rightarrow \infty} \Big\langle \prod_{y\in V(F\cup (\gamma_x F'))\setminus \mathring{V}(F\cup (\gamma_x F'))}\varphi_y^{(\Delta_F\cup\Delta_{\gamma_x F'})(x)} e^{-K_{F\cup(\gamma_x F')}}\Big\rangle_{\beta}^0
    \\
    =
    \Big\langle \prod_{y\in V(F)\setminus \mathring{V}(F)}\varphi_y^{\Delta_F(x)} e^{-K_{F}}\Big\rangle_{\beta}^0\Big\langle \prod_{y\in V(F')\setminus \mathring{V}(F')}\varphi_y^{\Delta_{F'}(x)} e^{-K_{F'}}\Big\rangle_{\beta}^0.
\end{aligned}    
\end{equation}
Combining \eqref{ergod f} and \eqref{ergod expectation} yields \eqref{ergod aim}. 
\end{proof}

\begin{rem} \label{rem: local limit}
Note that connections to the ghost vertex $\fg$ disappear in the weak limit. Indeed, for any sourceless current $\n$ with $\n_{o,\fg}>0$, consider the configuration $\n'_{o,\fg}=\n_{o,\fg}-1$ and $\n'_{x,y}=\n_{x,y}$ otherwise and notice that $w_{\beta,h_\Lambda(\fg)}^\emptyset(\n)\leq \beta h_\Lambda(\fg)_o w_{\beta,h_\Lambda(\fg)}^{\{o,\fg\}}(\n')$. Summing over all such $\n$ we obtain
$
\mathsf{P}^{\emptyset}_{\Lambda_\fg}[\n_{o,\fg}>0]
\leq
\beta h_\Lambda(\fg)_o \langle \phi_o \rangle^{\PLUS}_\Lambda,
$
which tends to $0$ by the integrability assumption on $J$. 

In addition, no ``long'' edges remain in the weak limit. More precisely, fix $\Delta \subset V$ and take any $n$ large enough such that $\Lambda:=B_n \supset \Delta$. We claim that
$
\sum_{y \in \Lambda \setminus \Delta} \mathsf{P}^\emptyset_{\Lambda_\fg} [  \n_{o,y} > 0 ]
\leq
C(\Delta),
$
where $C(\Delta)$ is uniform over $\Lambda$ and such that $C(\Delta) \rightarrow 0$ as $\Delta \uparrow V$. Indeed, proceeding as above, we have
\begin{equation*}
\sum_{y \in \Lambda \setminus \Delta} \mathsf{P}^\emptyset_{\Lambda_\fg} [  \n_{o,y} > 0 ]
\leq \sum_{y \in \Lambda \setminus \Delta} \beta J_{o,y}  \langle \phi_o \phi_y \rangle^{\PLUS}_{\Lambda}
\leq
\beta \sqrt{\langle \phi_o^2\rangle_{\Lambda}^{\PLUS}} \sum_{y \in \Lambda \setminus \Delta} J_{o,y}  \sqrt{\langle \phi_y^2 \rangle^{\PLUS}_{\Lambda}}
\end{equation*}
where we used the Cauchy--Schwarz inequality in the last inequality. By Corollary \textup{\ref{app cor: finite vol reg}} and Remark \textup{\ref{rem: boundary saviour}}  there exists $C > 0$ such that $\langle \phi_y^2 \rangle_\Lambda^{\PLUS} \leq \sqrt{C \max(\log( d_G(o,y)), 1)}$ uniformly in $\Lambda$. 
Hence, there exists $c_1 > 0$ such that, uniformly in $\Lambda$ and $\Delta$,
\begin{equation*}
\sum_{y \in \Lambda \setminus \Delta} \mathsf{P}^\emptyset_{\Lambda_\fg} [  \n_{o,y} > 0 ]
\leq
c_1 \sum_{y \in V\setminus \Delta} J_{o,y}  \sqrt[4]{C \max(\log( d_G(o,y)), 1)}=:C(\Delta)
\end{equation*}
After a summation by parts, the fact that $C(\Delta)\rightarrow 0$ is ensured by the integrability assumption on $J$, c.f. the calculation in Lemma \textup{\ref{lem: psi sum}}.

The above considerations imply that certain global properties, such as the degree of a vertex in the random current remaining even, is preserved in the infinite volume limit. This can be seen as a strengthening of weak convergence of these measures to a type of local limit.  
\end{rem}

\subsection{Infinite volume distributions on tangled currents}

For $\Lambda \subset V$ finite, equip $\Omega_{\Lambda_\fg,\Lambda}^\cT$ with its natural Borel $\sigma$-algebra $\cA_{\Lambda_\fg,\Lambda}^\cT$. Equip $\Omega^{\cT}_{V,V}$ with the $\sigma$-algebra generated by local (tangled) events, denoted by $\mathcal{A}_{V,V}^{\cT}$. We canonically identify measures on $\Omega_{\Lambda_\fg,\Lambda}^\cT$ with their push-forward under the natural injection $\cA_{\Lambda_\fg,\Lambda}^{\cT} \hookrightarrow \mathcal{A}_{V,V}^{\cT}$, by viewing the ghost as the contraction of $\Lambda^c$ to a single vertex. Recall also that we view $\Omega^{\mathcal{T}}_{\Lambda_\fg}$ (resp. $\Omega^{\mathcal{T}}_{\Lambda}$) as a subset of $\Omega^{\mathcal{T}}_{\Lambda_\fg,\Lambda}$ equipped with its natural Borel $\sigma$-algebra $\mathcal{A}^\mathcal{T}_{\Lambda_\fg}\subset \cA_{\Lambda_\fg,\Lambda}^\cT$ (resp.  $\mathcal{A}^\mathcal{T}_{\Lambda}\subset \cA_{\Lambda_\fg,\Lambda}^\cT$).

\begin{defn}
For $\Lambda \subset V$ finite, let $\bfP_{\Lambda_\fg, \Lambda}^{\emptyset,\emptyset}$, $\bfP_{\Lambda_\fg}^\emptyset$, and $\bfP_{\Lambda}^\emptyset$ be the probability measures on $\Omega_{\Lambda_\fg,\Lambda}^\cT$ defined, for $A \in \cA^{\cT}_{\Lambda_\fg,\Lambda}$, by
\begin{equation*}
\begin{aligned}
    \bfP_{\Lambda_\fg,\Lambda}^{\emptyset,\emptyset}[A]
    &:=
    \frac{\sum_{\n_1 \in \Omega_{\Lambda_\fg},\n_2 \in \Omega_\Lambda} \mathbbm{1}\{ \sn_1=\emptyset, \sn_2=\emptyset\} w_{\beta,h_\Lambda(\PLUS)}^\emptyset(\n_1)w_{\beta}^\emptyset(\n_2)\rho_{\n_1,\n_2}^{\emptyset,\emptyset}[A]}
    {\sum_{\n_1 \in \Omega_{\Lambda_\fg},\n_2 \in \Omega_\Lambda} \mathbbm{1}\{ \sn_1=\emptyset, \sn_2=\emptyset\}  w_{\beta,h_\Lambda(\PLUS)}^\emptyset(\n_1)w_{\beta}^\emptyset(\n_2)},
    \\
    \bfP_{\Lambda_\fg}^\emptyset [A]
    &:=
     \frac{\sum_{\n_1 \in \Omega_{\Lambda_\fg}} \mathbbm{1}\{ \sn_1=\emptyset\} w_{\beta,h_\Lambda(\PLUS)}^\emptyset(\n_1)\rho_{\n_1}^{\emptyset}[A]}
    {\sum_{\n_1 \in \Omega_{\Lambda_\fg}} \mathbbm{1}\{ \sn_1=\emptyset\}  w_{\beta,h_\Lambda(\PLUS)}^\emptyset(\n_1)},
    \\
    \bfP_{\Lambda}^\emptyset [A]
    &:=
     \frac{\sum_{\n_1 \in \Omega_{\Lambda}} \mathbbm{1}\{ \sn_1=\emptyset\} w_{\beta}^\emptyset(\n_1)\rho_{\n_1}^{\emptyset}[A]}
    {\sum_{\n_1 \in \Omega_{\Lambda}} \mathbbm{1}\{ \sn_1=\emptyset\}  w_{\beta}^\emptyset(\n_1)}.
\end{aligned}    
\end{equation*} 
\end{defn}

\begin{prop}\label{prop: existence mixing infinite volume tanglings}
The probability measures $\bfP^{\emptyset,\emptyset}_{\Lambda_\fg,\Lambda}$, $\bfP^\emptyset_{\Lambda_\fg}$, and $\bfP^\emptyset_\Lambda$ converge weakly as $\Lambda \uparrow V$, respectively, to probability measures $\bfP^{+,0}$, $\bfP^+$, $\bfP^0$ on $\Omega_{V,V}^\cT$. The convergence is in fact strong for cylinder events, i.e.\ for any $A \in \cA_{V,V}^\cT$ depending on finitely many edges and tanglings,
\begin{equation*}
\lim_{\Lambda \uparrow V} \bfP^{\emptyset,\emptyset}_{\Lambda_\fg, \Lambda}[A]
=
\bfP^{+,0}[A],
\quad
\lim_{\Lambda \uparrow V} \bfP^{\emptyset}_{\Lambda_\fg}[A]
=
\bfP^{+}[A],
\quad
\lim_{\Lambda \uparrow V} \bfP^{\emptyset}_{\Lambda}[A]
=
\bfP^{0}[A].
\end{equation*}
Furthermore, $\bfP^{+,0}$, $\bfP^+$, and $\bfP^0$ are invariant and mixing (hence, ergodic) under automorphisms in $\Gamma$. 
\end{prop}

\begin{proof} 
Without loss of generality, we prove the result for $\bfP^{+,0}$. First observe that $\bfP_{\Lambda_\fg, \Lambda}^{\emptyset,\emptyset}$ is obtained via the following coupling: first, sample $(\n_1,\n_2)$ according to the measure $\mathsf{P}_{\Lambda_\fg}^\emptyset\otimes\mathsf{P}_{\Lambda}^\emptyset$. Then, independently on each block $\mathcal{B}_x^{\emptyset,\emptyset}(\n_1,\n_2)$ for $x\in \Lambda$, sample a tangling $\ct_x$ according to the measure $\rho^{\emptyset,\emptyset}_{x,\n_1,\n_2}$ constructed in Proposition \ref{thm: switching lemma}. Thus, to define $\bfP^{+,0}$, we extend the coupling to infinite volume in the obvious way by using Proposition \ref{prop: inf vol flat}. 

The $\Gamma$-invariance and mixing is a direct consequence of Proposition \ref{prop: inf vol flat} and the construction via the coupling. 
\end{proof}


 \subsection{Uniqueness of the infinite cluster}

We prove that samples under the measures $\bfP^{+,0}$, $\bfP^+$, and $\bfP^0$, viewed as their associated multigraphs (as defined in Section \ref{subsec: tangled currents}) $\cH^{\emptyset,\emptyset}(\n_1,\n_2,\ct)$, $\mathcal{H}^{\emptyset}(\n_1,\ct_1)$, and $\mathcal{H}^{\emptyset}(\n_2,\ct_2)$ respectively, almost surely have at most one infinite cluster. We start by proving an insertion tolerance property for local events that are in $\cA_V$, i.e.\ do not depend on the tanglings. This, together with Proposition~\ref{prop: pos prob}, implies a (weak) insertion tolerance property for local events in $\cA_{V,V}^\cT$ (resp. $\cA_V^{\cT}$).

\begin{defn}
 For $\Lambda \subset V$ finite, define $\Phi_\Lambda:\Omega_{V}\times \Omega_V \rightarrow \Omega_{V}\times \Omega_V$ as the map $(\n_1,\n_2) \mapsto (\n_1', \n_2')$, where, for $x,y \in V$,
 \begin{equation*}
(\n_1',\n_2')\{x,y\}
=
\begin{cases}
(2,0), \qquad &\text{ if } (\n_1,\n_2)\{x,y\} = (0,0), \, x,y\in \Lambda \text{ and } J_{x,y} > 0,
\\
(\n_1,\n_2)\{x,y\}, \qquad &\text{ otherwise}.
\end{cases}
\end{equation*}
In particular, $\n_2'=\n_2$. We similarly define a map $\Omega_V\rightarrow \Omega_V$ that we still denote by $\Phi_\Lambda$.
\end{defn}

\begin{prop}\label{prop: insertion tolerance}
For $\Lambda \subset V$ finite, there exists $c=c(\Lambda, \beta, J) > 0$ such that, for all $A \in \cA_V$,
\begin{align}\label{eq: insertion tolerance}
        \mathsf{P}^+ \otimes \sfP^0[\Phi_\Lambda(A)]
    &\geq
    c\mathsf{P}^+ \otimes \sfP^0[A],
    \\\mathsf{P}^+[\Phi_\Lambda(A)]&\geq c\mathsf{P}^+[A],\label{eq: inser tol 2}
    \\\mathsf{P}^0[\Phi_\Lambda(A)]&\geq c\mathsf{P}^0[A].\label{eq: inser tol 3}
\end{align}
\end{prop}

\begin{proof} We only prove \eqref{eq: insertion tolerance} as the proof of \eqref{eq: inser tol 2} and \eqref{eq: inser tol 3} follows by similar considerations.
It is sufficient to consider $A$ depending only on finitely many edges. Recall that $B_n$ is the ball of radius $n$ centred on the fixed origin $o \in V$. Let $N$ be such that $A \in \cA_{B_N}$. We show the estimate \eqref{eq: insertion tolerance} holds for the measures $\mathsf{P}_{B_n \sqcup \{\fg\},}^\emptyset \otimes \sfP^\emptyset_{B_n}$ with constant uniform in $n > N$. The result then follows by Proposition \ref{prop: inf vol flat}. Note that for every $(\n_1',\n_2')$ we have $|\Phi_{\Lambda}^{-1}(\n_1',\n_2')|\leq 2^{|\mathcal{P}_2(\Lambda)|}$, hence 
\begin{equation}\label{piece 0}
\begin{aligned}
    \mathsf{P}_{B_n \sqcup \{\fg\},}^\emptyset \otimes \sfP^\emptyset_{B_n}[\Phi_\Lambda(A)]
     \geq 
     \dfrac{2^{-|\mathcal{P}_2(\Lambda)|}} {Z_{B_n\sqcup \{\fg\},B_n}^{\emptyset,\emptyset }}
     \sum_{\substack{\n_1 \in \Omega_{B_n \sqcup \{\fg\}}, \n_2 \in \Omega_{B_n} \\ \sn_1=\emptyset,\sn_2=\emptyset \\
     \n\in A}} w_{\beta}^\emptyset(\n_1')w_{\beta}^\emptyset(\n_2'),
\end{aligned}     
\end{equation}
where we recall $\Phi_{\Lambda}(\n_1,\n_2)=(\n_1',\n_2')$, and $\n=\n_1+\n_2$. By definition,
\begin{equation}\label{piece 1}
    w_{\beta}^\emptyset(\n_1')w_{\beta}^\emptyset(\n_2')
    =
    w_{\beta}^\emptyset(\n_1)w_{\beta}^\emptyset(\n_2) 
    \prod_{x\in \Lambda}\frac{\langle \varphi^{\Delta\n_1'(x)}\rangle_{0}}{\langle \varphi^{\Delta\n_1(x)}\rangle_{0}} \,
    \prod_{\lbrace x,y\rbrace \subset \Lambda}  \frac{(J_{x,y})^2}{2},
\end{equation}
where the final product is over all $x,y\in \Lambda$ such that $J_{x,y}>0$ and $(\n_1,\n_2)\lbrace x,y\rbrace=(0,0)$.
By Proposition \ref{bound single-site}, there exists $c_1>0$ such that, for all $k\geq 0$, $\dfrac{\langle \varphi_x^{2k+2}\rangle_{0}}{\langle \varphi_x^{2k}\rangle_{0}} \geq c_1$. Since $\Delta\n_1\leq\Delta\n_1'(x)\leq \Delta\n_1(x)+2|\Lambda|$, we conclude that
\begin{equation}\label{piece 2}
    \prod_{x\in \Lambda}\frac{\langle \varphi^{\Delta\n_1'(x)}\rangle_{0}}{\langle \varphi^{\Delta\n_1(x)}\rangle_{0}}\geq 
    \prod_{x\in \Lambda} \min\left(1, c_1^{|\Lambda|}\right)=
    \min\left(1, c_1^{|\Lambda|^2}\right)>0.
\end{equation}
Inserting \eqref{piece 1} and \eqref{piece 2} into \eqref{piece 0}, we get the existence of $c=c(\Lambda,\beta,J)$ such that, for all $n>N$, 
\begin{equation*}
    \mathsf{P}_{B_n \sqcup \{\fg\},}^\emptyset \otimes \sfP^\emptyset_{B_n}[\Phi_\Lambda(A)]
    \geq 
    c \mathsf{P}_{B_n \sqcup \{\fg\}}^\emptyset \otimes \sfP^\emptyset_{B_n}[A].
\end{equation*}
Taking $n\to \infty$ concludes the proof.
\end{proof}

\begin{defn}
 For $\Lambda \subset V$ finite, define $\Phi_\Lambda^\cT:\Omega_{V,V}^\cT \rightarrow \Omega_{V,V}^\cT$ as the map $(\n_1,\n_2,\ct) \mapsto (\n'_1,\n_2',\ct')$, given by $(\n_1',\n_2')=\Phi_\Lambda(\n_1,\n_2)$ and 
 \begin{equation*}
\ct_x'
=
\begin{cases}
\{\cB_x^{\emptyset,\emptyset}(\n'_1,\n_2')\}, \qquad &\text{ if } x\in \Lambda,
\\
\ct_x, \qquad &\text{ otherwise},
\end{cases}
\end{equation*}
i.e.\ in $(\n'_1,\n_2',\ct')$, the event $\textup{ECT}(x)$ happens for all $x\in \Lambda$. We similarly define a map $\Omega^\mathcal{T}_V\rightarrow \Omega^\mathcal{T}_V$ that we still denote by $\Phi^\mathcal{T}_\Lambda$.
\end{defn}

\begin{prop}\label{prop: insertion tolerance with tanglings}
If $A \in \cA_V$ is such that $\bfP^{+,0}[A]>0$, then $\bfP^{+,0}[\Phi_\Lambda^\cT(A)]>0$ for every $\Lambda \subset V$ finite. A similar result holds for $\bfP^+$ and $\bfP^0$.
\end{prop}
\begin{proof}
Recall that, conditionally on $\n$, the tanglings $\ct=(\ct_x)_{x\in V}$ are independent and the distribution of $\ct_x$ only depends on $(\Delta\n_1(x),\Delta\n_2(x))$. Therefore, the result readily follows from Propositions~\ref{prop: pos prob} and \ref{prop: insertion tolerance}. 
\end{proof}

Let $\mathfrak{N}$ denote the random variable that counts the number of infinite clusters in $\mathcal{H}(\n_1,\n_2,\ct)$ (resp. $\mathcal{H}(\n_1,\ct_1)$ or $\mathcal{H}(\n_2,\ct_2)$). Define for $\ell\in \mathbb N\cup \lbrace \infty\rbrace$,
$
    \mathcal{E}_\ell
    :=
    \lbrace \mathfrak{N}=\ell\rbrace.
$
Note that by ergodicity one has that for all $\ell$, $\bfP^{+,0}[\mathcal{E}_\ell]\in \lbrace 0,1\rbrace$. The following result establishes that almost surely there exists at most one infinite cluster.

\begin{prop}[Uniqueness of infinite cluster]\label{prop: uniqueness of inf cluster} 
For $\bfP\in\{\bfP^{+,0},\bfP^+,\bfP^0\}$ we have
\begin{equation*}
    \bfP[\mathcal{E}_0]+\bfP[\mathcal{E}_1]=1.
\end{equation*}
\end{prop}

\begin{proof}
We only prove the result for $\bfP^{+,0}$, as the proof for $\bfP^{+}$ and $\bfP^0$ is analogous. We start by constructing a convenient sequence of finite domains $(\Lambda_N)_{N}$ as follows. For each $x\in V$, consider a sequence of vertices $\pi_{o,x}=(u_0=o,u_1,\cdots,u_k=x)$ with $J_{u_i,u_{i+1}}>0$ for all $0\leq i\leq k-1$, which exists by the irreducibility assumption. We then consider the domains $\Lambda_N:=\bigcup_{x\in B_N} \pi_x$.

We first establish that $\bfP^{+,0}[\mathcal{E}_\ell]=0$ for $\ell \in \mathbb N\setminus\lbrace 0,1\rbrace$. Let $\ell \in \mathbb N$ with $\ell\geq 2$ and assume by contradiction that $\bfP^{+,0}[\mathcal{E}_\ell]=1$. Let $\mathcal{I}_N = \mathcal{I}_N(\ell)\subset \cE_\ell$ be the event that the $\ell$ infinite clusters intersect the blocks of the vertices in $B_N$. Observe that $\bfP^{+,0}[\mathcal{I}_N]\underset{N\rightarrow \infty}\longrightarrow 1$, so in particular there exists $N$ such that $\bfP^{+,0}[\cI_N]>0$. By Proposition~\ref{prop: insertion tolerance with tanglings}, we have
$
\bfP[\Phi_{\Lambda_N}^\cT(\cI_N)]>0.
$
However, one can easily see that $\Phi_{\Lambda_N}^\cT(\cI_N)\subset \cE_1$, and thus $\bfP^{+,0}[\mathcal{E}_1]>0$, which is a contradiction with $\bfP^{+,0}[\cE_\ell]=1$, as desired.

Now we establish $\bfP^{+,0}[\mathcal{E}_\infty]=0$. 
Assume for a contradiction that $\bfP^{+,0}[\mathcal{E}_\infty]=1$. Let $\mathcal{J}_N$ be the event that at least three distinct infinite clusters intersect the blocks of the vertices in $B_N$. Note that
$\bfP^{+,0}[\mathcal{J}_N]\underset{N\rightarrow\infty}\longrightarrow 1$,
hence there is some $N$, fixed throughout, such that
$
\bfP^{+,0}[\mathcal{J}_N]>0.
$
Given $\Lambda\subset V$ and a double tangled current $(\n_1,\n_2,\ct)$, we denote by $\mathcal{L}(\Lambda,\n_1,\n_2,\ct)$ the lift of $\Lambda$ in $\cH(\n_1,\n_2,\ct)$, i.e. the subgraph of $\cH(\n_1,\n_2,\ct)$ spanned by $\cup_{x\in\Lambda} \cB_x(\n_1,\n_2)$. 
We introduce the following notion of coarse trifurcation. Let $\mathsf{CT}_o$ denote the set of $(\n,\ct)$ such that:
\begin{enumerate}
    \item[-] $\n(x,y)>0$ for all edges $xy\in \Lambda_N$.
    \item[-] ECT$(x)$ happens for all vertices  $x\in \Lambda_N$.
    \item[-] The cluster of the origin in $\mathcal{H}(\n_1,\n_2,\ct)$ breaks into at least three distinct infinite clusters by removing $\cL(\Lambda_N,\n_1,\n_2,\ct)$. 
\end{enumerate}
One can easily verify that $\Phi_{\Lambda_N}(\cJ_N)\subset \mathsf{CT}_o$, hence by Proposition~\ref{prop: insertion tolerance with tanglings} we have
\begin{equation}\label{eq:trifurcation positive}
    \bfP^{+,0}[\mathsf{CT}_o]
    >
    0.
\end{equation}
Furthermore, note that by $\Gamma$-invariance, we have $\bfP^{+,0}[\mathsf{CT}_0]=\bfP^{+,0}[\mathsf{CT}_x]$ for all $x \in V$, where $\mathsf{CT}_x=\gamma_x \mathsf{CT}_0$, and $\gamma_x\in \Gamma$ is an automorphism such that $\gamma_x o=x$. A vertex $x\in V$ is called a coarse trifurcation if $\mathsf{CT}_x$ occurs.

For each $x\in V$ denote $\Lambda_N(x)=\gamma_x\Lambda_N$. Let $n \in \NN$ and define $S_n\subset B_n$ to be a set of (sparse) vertices such that $\Lambda_N(x)\cap\Lambda_{N}(y)=\emptyset$ for all distinct $x,y\in S_n$, and $S_n$ is maximal with respect to this property. Then 
$
|S_n|\geq \frac{|B_n|}{|B_{K}|},
$
where $K>N$ is large enough so that $\Lambda_N\subset B_K$.
Let $T_n=T(S_n)$ denote the set of coarse trifurcations in $S_n$, and $L_n$ denote the set of edges in $\cH(\n_1,\n_2,\ct)$ between $\cL(B_n,\n_1,\n_2,\ct)$ and $\cL(V\setminus B_n,\n_1,\n_2,\ct)$. We now consider the graph obtained by identifying all vertices of $\mathcal{L}(\Lambda_N(t),\n_1,\n_2,\ct)$ into a single vertex for each $t\in T_n$. It is standard that one can construct a subgraph of this new graph that is a forest, where every leaf corresponds to an element of $L_n$, while every element of $T_n$ corresponds to a vertex of degree at least three. See e.g.\ \cite[Theorem 7.6]{LyonsPeres}.
Therefore, we have
\begin{equation}\label{eq:T_n<L_n}
|T_n|\leq|L_n|.
\end{equation} 
On the one hand, by \eqref{eq:trifurcation positive} and $\Gamma$-invariance, 
\begin{equation}\label{eq:bound on T_n}
    \mathbf{E}^{+,0}[|T_n|]=|S_n|\bfP^{+,0}[\mathsf{CT}_0]\geq c(N)|B_n|,
\end{equation}
for some constant $c(N)>0$.
On the other hand, for every $m>0$ we have, by $\Gamma$-invariance,
\begin{equation*}
\begin{aligned}
\mathbf{E}^{+,0}[|L_n|]
&=
\sum_{\substack{x \in B_n\\y \in V \setminus B_n}} \mathbf{E}^{+,0}[\n_{x,y}]  
\leq
\sum_{x\in \partial^m_{\textup{in}} B_n} \mathbf{E}^{+,0}[\Delta\n(x)]+\sum_{\substack{x\in B_n\setminus \partial_{\textup{in}}^m B_n}} \mathbf{E}^{+,0}[\Delta^m\n(x)]\\
&\leq |\partial^m_{\textup{in}} B_n|\mathbf{E}^{+,0}[\Delta\n(o)] + |B_n|\mathbf{E}^{+,0}[\Delta^m\n(o)],
\end{aligned}
\end{equation*}
where $\partial^m_{\textup{in}} B_n:=\{x\in B_n : d_G(x, V\setminus B_n)\leq m\}$ and $\Delta^m\n(x):=\sum_{y\in V : d_G(x,y)>m} \n_{x,y}$. Notice that $|\partial^m_{\textup{in}} B_n|\leq |B_{m}| |\partial B_n|$.
Moreover, recall that $c_0:=\mathbf{E}^{+,0}[\Delta\n(o)]<\infty$ by Lemma~\ref{degree control}. In particular, $c_m:=\mathbf{E}^{+,0}[\Delta^m\n(o)]\to 0$ as $m\to \infty$. Altogether, we have
\begin{equation}\label{eq:bound on L_n}
\mathbf{E}^{+,0}[|L_n|]\leq c_0|B_m||\partial B_n|+c_m|B_n|
\end{equation}
Combining \eqref{eq:T_n<L_n}, \eqref{eq:bound on T_n} and \eqref{eq:bound on L_n}, we obtain 
\begin{equation*}
c(N)\leq c_0|B_m|\frac{|\partial B_n|}{|B_n|} + c_m.
\end{equation*}
Letting $n\to\infty$ and then $m\to\infty$, we find that $c(N)\leq 0$, thus obtaining the desired contradiction.
\end{proof}
\begin{rem}\label{rem: at most two ends}
As a by-product of the above proof we also obtain that the infinite cluster (if it exists) in $\mathcal{H}(\n_1,\n_2,\ct)$ has at most two ends almost surely. Indeed, if the infinite cluster has at least three ends, one can proceed as above to prove that coarse trifurcations happen with positive probability and then obtain a contradiction.
\end{rem}

 \section{$\Gamma$-invariant Gibbs measures} \label{sec: gamma inv }

In this section, we prove Theorem~\ref{thm:main}. Fix $\beta>0$, which we may sometimes omit from the notation without risk of confusion. By Theorem~\ref{Lebowitz}, we first need to prove that the condition $4$ therein holds. A straightforward argument then yields that under \eqref{eq: condition for continuity}, $|\cG(\beta)|=1$ to complete the proof. We therefore make the following assumption for the rest of this section, and in the end, we derive a contradiction.

\begin{ass} 
Assume that there exist $x,y\in V$ such that $J_{x,y}>0$ and
\begin{equation}\label{ass:contradiction} \tag{A}
    \langle \varphi_x \varphi_y  \rangle^+_\beta
	\neq
	\langle \varphi_x \varphi_y  \rangle^0_\beta.
\end{equation}
\end{ass} 

Under this assumption, we show that the following event happens with positive probability.

 \begin{defn}[Bridges]
We say that a pair $x,y\in V$ is a bridge for a double tangled current $(\n_1,\n_2,\ct)$ (resp. a tangled current $(\n,\ct)$) if there exist two vertices $\overline{x}$ and $\overline{y}$ in the blocks of $x$ and $y$ in the multigraph $\cH(\n_1,\n_2,\ct)$ (resp. $\cH(\n,\ct)$) such that there is exactly one edge $e$ between $\overline{x}$ and $\overline{y}$, and furthermore the cluster of $\overline{x},\overline{y}$ splits into two disjoint infinite clusters after removing $e$. We denote the event that $\lbrace x,y\rbrace$ is a bridge by $\cA_{x,y}$.
 \end{defn}
 
In what follows, we write $\{A\longleftrightarrow B \text{ in }\Lambda\}$ to denote the event that $A$ is connected to $B$ in $\mathcal{H}_{\Lambda}(\n_1,\n_2,\ct)$ (resp. $\mathcal{H}_{\Lambda}(\n,\ct)$). To simplify the notation, for $x,y\in V$, we write $x\longleftrightarrow y$ instead of $\mathcal{B}_x\longleftrightarrow \mathcal{B}_y$. Moreover, we identify a moment $A\in \mathcal{M}(V)$ with the corresponding multiset it defines, where each $x\in V$ appears as many times as the value of $A_x$, e.g.\ we identify the moment $A$ defined as $A_x=A_y=2$ and $A_z=0$ otherwise, with $xxyy$.
 
\begin{prop}\label{prop:bridges}
Assume \eqref{ass:contradiction}. Then, 
 \begin{equation*}
 \bfP^{+,0}_{\beta}[\cA_{x,y}]>0
 \quad
 \text{ and }
 \quad
 \bfP^{+}_{\beta}[\cA_{x,y}]>0.
 \end{equation*}
\end{prop}

\begin{proof} 
Let $\cA^f_{x,y}$ denote the event that there exist two vertices $\overline{x}$ and $\overline{y}$ in the blocks of $x$ and $y$ in the multigraph $\cH(\n_1,\n_2,\ct)$ (resp. $\mathcal{H}(\n,\ct)$) such that there is exactly one edge $e$ between $\overline{x}$ and $\overline{y}$, and furthermore the cluster of $\overline{x},\overline{y}$ splits into two disjoint clusters after removing $e$ that are both connected to $\fg$.
By Proposition~\ref{prop: existence mixing infinite volume tanglings} and Remark~\ref{rem: local limit}, it suffices to show that there exists $c>0$ such that $\bfP_{B_N \sqcup \{\fg\}, B_N}^{\emptyset,\emptyset}[\cA^f_{x,y}]>c$ and $\bfP_{B_N \sqcup \{\fg\}}^{\emptyset}[\cA^f_{x,y}]>c$ for every $N$ large enough.

Let $c:=\langle \varphi_x \varphi_y  \rangle^+_\beta - \langle \varphi_x \varphi_y  \rangle^0_\beta>0$ and take $N$ large enough so that $\langle \varphi_x \varphi_y  \rangle^\PLUS_{B_N} - \langle \varphi_x \varphi_y  \rangle^0_{B_N}$ is larger than $c/2$. By the switching principle of Theorem \ref{thm: switching lemma}, we have 
\begin{equation}\label{eq:p2p_plus-free1}
c/2<\langle \varphi_x \varphi_y  \rangle^\PLUS_{B_N} - \langle \varphi_x \varphi_y  \rangle^0_{B_N}=\langle \varphi_x \varphi_y  \rangle^\PLUS_{B_N} \bfP^{xy,\emptyset}_{B_N\sqcup \{\fg\},B_N}[\tilde x \centernot\longleftrightarrow \tilde y \text{ in } B_N],
\end{equation}
where $\tilde x=xa(1),\tilde y=ya(1)$ are the extra vertices added in $\cB_x, \cB_y$ due to the sources. Notice that due to the sources constraint, $\tilde x$ is connected to $\tilde y$ in $B_N\sqcup \{\fg\}$, in particular on the event in \eqref{eq:p2p_plus-free1} both $\tilde x$ and $\tilde y$ are connected to $\fg$, and therefore $\lbrace x,y\rbrace$ would become a bridge by adding an edge between $\tilde x$ and $\tilde y$ in $\cH^{xy,\emptyset}(\n_1,\n_2,\ct)$. We now implement this modification rigorously.
First, take $M$ large enough (depending on $c$ only) such that
\begin{equation}\label{eq:p2p_plus-free2}
\langle \varphi_x \varphi_y  \rangle^\PLUS_{B_N}\bfP^{xy,\emptyset}_{B_N\sqcup \{\fg\},B_N}[\tilde x \centernot{\longleftrightarrow} \tilde y \text{ in } B_N,~ \Delta\n(x)\leq M,~ \Delta\n(y)\leq M]>c/4.
\end{equation}
To each $(\n_1,\n_2,\ct)$ in the event above, we consider the modified configuration $(\n_1',\n_2,\ct')$ constructed as follows.
First, set $\n_1'(x,y)=\n_1(x,y)+1$ and $\n_1'(e)=\n_1(e)~\forall e\neq \lbrace x,y\rbrace$ (notice that $\partial \n_1'=\emptyset$). In order to define $\ct'$, first set $\ct'_z=\ct_z~\forall z\notin \{x,y\}$, which makes sense as $\cB_z^{xy,\emptyset}(\n_1,\n_2)=\cB_z^{\emptyset,\emptyset}(\n'_1,\n_2) ~ \forall z\notin \{x,y\}$. 
By seeing the previously extra vertices $\tilde x, \tilde y$ now as the endpoints of the added edge between $x$ and $y$ in $\n_1'$, we construct a natural bijection $\psi_z:\cB_z^{xy,\emptyset}(\n_1,\n_2)\to \cB_z^{\emptyset,\emptyset}(\n'_1,\n_2)$ for $z\in\{x,y\}$.
We can then set $\ct'_x=\psi_x(\ct_x)$  and  $\ct'_y=\psi_y(\ct_y)$. One can see that $\lbrace x,y\rbrace$ is a bridge in the configuration $(\n_1',\n_2,\ct')$, with $\overline{x}=\tilde{x}$ and $\overline{y}=\tilde{y}$--- see Figure~\ref{fig: bridge xy}. On the one hand, notice that the map $(\n_1,\n_2,\ct)\mapsto(\n_1',\n_2,\ct')$ is at most $(M+1)^2$-to-$1$ due to the degree constraints. On the other hand, by \eqref{eq:weights_def} and Proposition~\ref{prop: pos prob}, the ratio between the probabilities of $(\n_1,\n_2,\ct)$ and $(\n_1',\n_2,\ct')$ under $\bfP^{xy,\emptyset}_{B_N\sqcup\{\fg\}, B_N}$ and $\bfP^{\emptyset,\emptyset}_{B_N\sqcup \{\fg\}, B_N}$, respectively, is bounded above by a constant $C=C(M)\in(0,\infty)$. All in all, we conclude that for $N$ large enough,
\begin{equation*}
\bfP^{\emptyset,\emptyset}_{B_N\sqcup \{\fg\}, B_N}[\cA_{x,y}^f]>\frac{c}{4C \langle \varphi_x \varphi_y  \rangle^\PLUS_{B_N}},
\end{equation*}
as we wanted to prove. For the single tangled current, simply notice that by the stochastic domination \eqref{eq: domination} we have $\bfP^{xy}_{B_N\sqcup \{\fg\}}[\tilde x {\centernot\longleftrightarrow} \tilde y \text{ in } B_N]\geq \bfP^{xy,\emptyset}_{B_N\sqcup \{\fg\}, B_N}[\tilde x {\centernot\longleftrightarrow} \tilde y \text{ in } B_N]$, and the above proof adapts readily.
\end{proof}

\begin{figure}[!tbp]
  \centering
    \includegraphics[width=\textwidth]{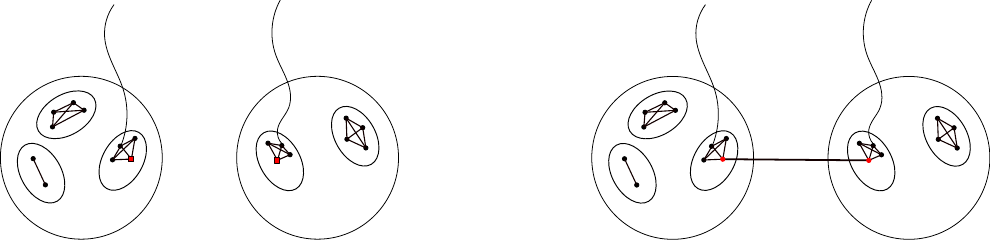}
    \put(-330,-10){$\mathcal{B}_x$}
    \put(-250,-10){$\mathcal{B}_y$}
    \put(-260,88){$\fg$}
    \put(-320,88){$\fg$}
    \put(-120,-10){$\mathcal{B}_x$}
    \put(-30,-10){$\mathcal{B}_y$}
    \put(-105,88){$\fg$}
    \put(-48,88){$\fg$}
    \caption{An illustration of the mapping used in the proof of Proposition \ref{prop:bridges}. On the left, an example of $(\n_1,\n_2,\ct)$ satisfying the event $\{\tilde x {\centernot\longleftrightarrow} \tilde y \text{ in } B_N\}$; on the right, its image $(\n_1',\n_2,\ct')$ by the map. The red vertices represent $\overline{x}, \overline{y}$.}
    \label{fig: bridge xy}
\end{figure}

\begin{defn}[Flows]
Given two vertices $\tilde x,\tilde y$ in a multigraph $\cH$, we define the flow $f(\tilde x,\tilde y,\cH)$ between $\tilde x$ and $\tilde y$ in $\cH$ as the maximal number of edge disjoint paths between $\tilde x$ and $\tilde y$ in $\cH$. We also define $f(\tilde x,\infty,\cH)$ as the maximal number of edge disjoint infinite path in $\cH$ starting from $\tilde x$.

For a double tangled current $(\n_1,\n_2,\ct)$ (resp. a tangled current $(\n,\ct)$) on $V$ and $x,y\in V$, we denote by $f(x,y,\n_1,\n_2,\ct)$ (resp. $f(x,y,\n,\ct)$) the maximum of $f(\tilde x,\tilde y,\cH(\n_1,\n_2,\ct))$ (resp. $f(\tilde{x},\tilde{y},\cH(\n,\ct)$) over all $\tilde x \in \cB_x$ and $\tilde y\in \cB_y$, and by $f(x,\infty,\n_1,\n_2,\ct)$ (resp. $f(x,\infty,\n,\ct)$) the maximum of $f(\tilde x,\infty,\cH(\n_1,\n_2,\ct))$ (resp. $f(\tilde x,\infty,\cH(\n,\ct)$) over all $\tilde x \in \cB_x$.
\end{defn}

\begin{lem}\label{lem:flow_properties}
Assume \eqref{ass:contradiction}. One has
\begin{align}
\label{eq:flow_p2p}
\lim_{d_G(o,y) \rightarrow \infty} \bfP^{+,0}_\beta[f(x,y,\n_1,\n_2,\ct) \geq 2]
&=
0,
\\
\label{eq:flow_inf}
\bfP^{+,0}_\beta[f(x,\infty,\n_1,\n_2,\ct) > 2]
&=
0.
\end{align}
Moreover, a similar statement holds for the measure $\mathbf{P}^+_\beta$.
\end{lem}

\begin{proof}
Note that by ergodicity and Proposition \ref{prop:bridges}, almost surely there are infinitely many bridges. Also, there is a unique infinite cluster, which has at most two ends almost surely by Proposition \ref{prop: uniqueness of inf cluster} and Remark \ref{rem: at most two ends}. For any $(\n_1,\n_2,\ct)$ satisfying these properties, it is proved in \cite[Lemma 3]{RAO20} that  $\lim_{d_G(o,y)\rightarrow \infty} \mathbbm{1}\{f(x,y,\n_1,\n_2,\ct) \geq 2\} = 0$ and $f(x,\infty,\n_1,\n_2,\ct) \leq 2$. Then, \eqref{eq:flow_p2p} follows by dominated convergence and \eqref{eq:flow_inf} follows directly. The proof for $\mathbf{P}^+_\beta$ follows the same argument.
\end{proof}

\begin{lem}\label{lem:inf}
Let $x\in V$ and $\beta>0$. If $\bfP^{+,0}_\beta[x \longleftrightarrow \infty]>0$, then $\inf_{y\in V} \langle \phi_x \phi_y \rangle^0_\beta>0$. 
\end{lem}
\begin{proof}
Since $\beta$ is fixed, we shall omit it from the notation. By the uniqueness of the infinite cluster, for every $y\in V$,
\begin{equation}\label{eq:x con y}
\begin{aligned}
    \bfP^{+,0}[x \longleftrightarrow y]=\bfP^{+,0}[x \longleftrightarrow\infty,\: y \longleftrightarrow \infty] 
    + \bfP^{+,0}[x \longleftrightarrow y,\: x \centernot\longleftrightarrow \infty].
\end{aligned}    
\end{equation}
The first term converges to $(\bfP^{+,0}[x \longleftrightarrow \infty])^2$ by the mixing of Proposition \ref{prop: existence mixing infinite volume tanglings}, while the second term converges to $0$. Thus 
\begin{equation*}
\lim_{d_G(o,y)\to\infty} \bfP^{+,0}[x\longleftrightarrow y]= \big(\bfP^{+,0}[x\longleftrightarrow \infty]\big)^2,
\end{equation*}
hence 
$\bfP^{+,0}[x\longleftrightarrow y]$ remains bounded away from $0$ uniformly in $x,y$. We wish to show that there exists a constant $C>0$ independent of $x,y$ such that 
\begin{equation}\label{eq:xy bound}
\bfP^{+,0}[x\longleftrightarrow y]\leq C \langle \phi_x \phi_y \rangle^0.
\end{equation}
To this end, we shall work at finite volume.
For every $m>0$ large enough so that $x,y\in B_m$,
\begin{equation*}
\begin{aligned}
\bfP^{\emptyset,\emptyset}_{B_N\sqcup \{\fg\}, B_N}[x \longleftrightarrow y \text{ in } B_m]
\leq
\bfP^{\emptyset,\emptyset}_{B_N\sqcup \{\fg\}, B_N}[x \longleftrightarrow y] 
\leq
\bfP^{\emptyset,\emptyset}_{B_N\sqcup \{\fg\}, B_N}[x \longleftrightarrow y \text{ in } B_m]\\
+
\bfP^{\emptyset,\emptyset}_{B_N\sqcup \{\fg\}, B_N}[x,y \longleftrightarrow  B_m^c, x \centernot\longleftrightarrow y \text{ in } B_m],
\end{aligned}
\end{equation*}
hence sending $N$ and $m$ to infinity, and using the uniqueness of the infinite cluster, we deduce that 
\begin{equation*}
\lim_{N\to\infty}\bfP^{\emptyset,\emptyset}_{B_N\sqcup \{\fg\}, B_N}[x \longleftrightarrow y]= \bfP^{+,0}_{}[x \longleftrightarrow y],
\end{equation*}
which implies that $\bfP^{\emptyset,\emptyset}_{B_N\sqcup \{\fg\}, B_N}[x \longleftrightarrow y]$ remains bounded away from $0$ for every large enough $N$.
We can now find constants $M,C_1,C_2>0$ independent of $x,y$ and $N_0=N_0(x,y)>0$ such that for every $N\geq N_0$,
\begin{align*}
\bfP^{\emptyset,\emptyset}_{B_N\sqcup \{\fg\}, B_N}[x\longleftrightarrow y]&\leq 
C_1 \bfP^{\emptyset,\emptyset}_{B_N\sqcup \{\fg\}, B_N}[\mathcal{E}]
\leq C_2 \langle \phi_x^2 \phi_y^2 \rangle^0_{B_N} \, \bfP^{A,\emptyset}_{B_N\sqcup \{\fg\}, B_N}[\mathcal{E}] \\
&\leq C_2 \langle \phi_x \phi_y \rangle^\PLUS_{B_N} \, \langle \phi_x \phi_y \rangle^0_{B_N},
\end{align*}
where $\mathcal{E}:=\{x\longleftrightarrow y,\: \textup{ECT}(x),\: \textup{ECT}(y), \Delta\n(x)\leq M,\: \Delta\n(y)\leq M\big\}$ and $A=xxyy$. Sending $N$ to infinity and using that 
$
\langle \phi_x \phi_y \rangle^+ \leq \langle \phi_x^2  \rangle^+<\infty 
$
by the Cauchy--Schwarz inequality and $\Gamma$-invariance, we obtain that \eqref{eq:xy bound} holds.
It follows that $\inf_{y\in V}\langle \phi_x \phi_y \rangle^0>0$, as desired.
\end{proof}

\subsection{Proof of Theorem~\ref{thm:main}}

\begin{proof}[Proof of \eqref{eq: convex hull of pm} in Theorem \textup{\ref{thm:main}}]
Assuming \eqref{ass:contradiction}, we have $\bfP^{+,0}[\cA_{x,y}]>0$ by Proposition~\ref{prop:bridges}. In particular $\bfP^{+,0}[x\longleftrightarrow \infty]>0$, which implies
\begin{equation}\label{eq:p2p_spin}
c_1:=\inf_{y\in V} \langle \phi_x \phi_y \rangle^0_{\beta}>0
\end{equation}
by Lemma~\ref{lem:inf}.
We now consider two cases.

\textbf{First case.} $\bfP^{0}[x\longleftrightarrow\infty]>0$.
Equivalently, we have $\bfP^0[f(x,\infty,\n,\ct)\geq 1]>0$. We also know by Proposition~\ref{prop:bridges} that $\bfP^+[\cA_{x,y}]>0$, so in particular $\bfP^+[f(x,\infty,\n,\ct)\geq 2]>0$. Let $(\n_1,\ct_1)\sim \bfP^+$ and $(\n_2,\ct_2)\sim \bfP^0$ be independent and $(\n_1,\n_2,\ct)\sim \bfP^{+,0}$, then by \eqref{eq: domination} we can couple $(\n_1,\n_2,\ct,\ct_1,\ct_2)$ in such a way that $\ct\succ \ct_1\sqcup \ct_2$. 
Using this stochastic domination and inserting the event $\textup{ECT}(x)$, we obtain
\begin{equation*}
\bfP^{+,0}\big[f(x,\infty,\n_1,\ct_1)\geq 2,\: f(x,\infty,\n_2,\ct_2)\geq 1, \:\textup{ECT}(x)\big]>0.
\end{equation*}
Since $\{f(x,\infty,\n_1,\ct_1)\geq 2, f(x,\infty,\n_2,\ct_2)\geq 1, \textup{ECT}(x)]\}\subset \{f(x,\infty,\n_1,\n_2,\ct)\geq3\}$, we obtain a contradiction with Lemma~\ref{lem:flow_properties}--- see \eqref{eq:flow_inf}.

\textbf{Second case.} $\bfP^{0}[x\longleftrightarrow\infty]=0$.
We prove the following lemmas.
\begin{lem}\label{lem:p2p_free}
Assume $\bfP^0[x\longleftrightarrow\infty]=0$. Then, for every $x,y\in V$ there exists an integer $m=m(x,y)$ such that 
\begin{equation}\label{eq:p2p_free}
\bfP_{B_n}^{xy}[ x \longleftrightarrow  y \text{ in } B_m]
>
1/2
\end{equation}
for all $n$ sufficiently large.
\end{lem}

\begin{lem}\label{lem:p2p_plus}
Assume \eqref{ass:contradiction}. Then, there exists $c_2 > 0$ such that, for all $x \in V$ there exists infinitely many $y \in V$ such that, there exists an integer $m =m(x,y)$ such that
 \begin{equation}\label{eq:p2p_plus}
\langle \phi_x \phi_y \rangle^\PLUS_{B_n} \bfP_{B_n\sqcup \{\fg\}}^{xy}[ x \longleftrightarrow y \text{ in } B_m]
>
c_2
 \end{equation}
for all $n$ sufficiently large.
\end{lem}

Before proving Lemmas~\ref{lem:p2p_free} and \ref{lem:p2p_plus}, we conclude the proof of the main result. 
Let $x,y,m,n$ be such that both \eqref{eq:p2p_free} and \eqref{eq:p2p_plus} hold. By taking the product between \eqref{eq:p2p_free}, \eqref{eq:p2p_plus} and $\langle \phi_x \phi_y \rangle^0_{B_n}$ (which we can assume to be larger than $c_1/2$ by \eqref{eq:p2p_spin}), using \eqref{eq: domination}, and bounding $\Delta\n(x)$ and $\Delta\n(y)$ to insert $\textup{ECT}(x)$, $\textup{ECT}(y)$, we obtain
\begin{equation*}
\langle \phi_x \phi_y \rangle^\PLUS_{B_n} \langle \phi_x \phi_y \rangle^0_{B_n}
\bfP^{xy,xy}_{B_n\sqcup \{\fg\}, B_n}[f(x,y,\n_1,\n_2,\ct)\geq2 \text{ in } B_m]>c_3
\end{equation*}
for some uniform constant $c_3>0$.
Applying the switching principle of Theorem \ref{thm: switching lemma} we obtain
\begin{equation*}
\langle \phi_x^2 \phi_y^2 \rangle^\PLUS_{B_n}
\bfP^{xxyy,\emptyset}_{B_n\sqcup \{\fg\}, B_n}[f(x,y,\n_1,\n_2,\ct)\geq2 \text{ in } B_m]>c_3.
\end{equation*}
We can now remove the sources $A=xxyy$ at a finite cost by bounding $\Delta\n(x)$ and $\Delta\n(y)$, and inserting $\textup{ECT}(x)$, $\textup{ECT}(y)$, to obtain
$
\bfP^{\emptyset,\emptyset}_{B_n\sqcup \{\fg\}, B_n}[f(x,y,\n_1,\n_2,\ct)\geq2 \text{ in } B_m]>c_4
$
for some uniform constant $c_4>0$. Taking $n\to\infty$,
$
\bfP^{+,0}[f(x,y,\n_1,\n_2,\ct)\geq2 ]\geq c_4.
$
This contradicts \eqref{eq:flow_p2p} from Lemma~\ref{lem:flow_properties}, thus concluding the proof of condition 4 in Proposition \ref{Lebowitz} and, hence, establishing \eqref{eq: convex hull of pm} in Theorem \ref{thm:main}. 
\end{proof}

\begin{proof}[Proof of Lemma~\textup{\ref{lem:p2p_free}}]
Assume for a contradiction that \eqref{eq:p2p_free} does not hold. Then, there exist $x,y\in V$ such that for every $m\geq 0$, 
$
\bfP^{xy}_{B_n}[x \longleftrightarrow \partial B_m]\geq 1/2
$
for all $n$ sufficiently large, since there is a path from $x$ to $y$ in $(\n,\ct)$ due to the source constraints. We now wish to compare $\bfP^{\emptyset}_{B_n}[x \longleftrightarrow \partial B_m]$ to $\bfP^{xy}_{B_n}[x \longleftrightarrow \partial B_m]$ to deduce that the former remains bounded away from $0$, thus obtaining a contradiction. To this end, fix a sequence of vertices $u_0=x, u_1,\ldots, u_k=y$ such that $J_{u_i,u_{i+1}}>0$ for all $i=0,1,\ldots, k-1$, and let $M>0$ be large enough (depending on $x,y$ only) so that 
\begin{equation}\label{eq:p2p_plus_proof}
\bfP^{xy}_{B_n}[x \longleftrightarrow \partial B_m, \, \Delta\n(u_i)\leq M ,\: \forall i=0,1,,\ldots,k] \geq 1/4
\end{equation}
for every $m\geq 0$ and every $n$ sufficiently large.
Now given a current $\n$ with sources $x,y$, let $\n'$ be the sourceless current which is defined as
\begin{equation*}
\n'(s,t)=
\begin{cases}
     \n(s,t)+1, & \text{for } \{s,t\}\in \big\{\{u_0u_1\},\ldots,\{u_{k-1}u_k\}\big\}, \\
     \n(s,t), & \text{otherwise}.
 \end{cases}
\end{equation*}
By \eqref{eq:weights_def}, Proposition~\ref{prop: pos prob} and the fact that $\Delta\n(u_i)\leq M$ for all $i=0,1,\ldots,k$, we obtain
\begin{equation*}
w_{\beta}^{xy}(\n)\leq C_5 w_{\beta}^{\emptyset}(\n') \quad \text{and} \quad \rho_{\n}^{xy}[x \longleftrightarrow \partial B_m]\leq C_5\rho_{\n'}^{\emptyset}[x \overset{(\n',\ct)}{\longleftrightarrow} \partial B_m]
\end{equation*} 
for a constant $C_5>0$ depending on $x,y$ and $M$ only.  
Since the map $\n\mapsto \n'$ is injective, it follows that
\begin{equation*}
\langle \phi_x \phi_y \rangle^0_{B_n} \bfP^{xy}_{B_n}[x \longleftrightarrow \partial B_m, \, \Delta\n(u_i)\leq M \, \forall i=0,1,\ldots,k] \leq C_5^2 \bfP^{\emptyset}_{B_n}[x \longleftrightarrow \partial B_m].
\end{equation*}
By the irreducibility property, one can easily prove that $\langle \phi_x \phi_y \rangle^0_{B_n}\geq c_6$ for $n$ large enough, where $c_6>0$ depends on $x,y$ only. Recalling \eqref{eq:p2p_plus_proof}, we conclude that $\bfP^{\emptyset}_{B_n}[x \longleftrightarrow \partial B_m]\geq (c_6/4)C_5^2$ for every $m$ and every $n$ large enough.
Taking $n\to\infty$ and then $m\to \infty$, we obtain $\bfP^0[x \longleftrightarrow \infty]>0$, which concludes the proof by contradiction.
\end{proof}

\begin{proof}[Proof of Lemma~\textup{\ref{lem:p2p_plus}}]
Let $x,x'$ be such that $c_7:=\bfP^+[\mathcal{A}_{x,x'}]>0$. Then, by mixing we have
$
\lim_{y\to\infty}\bfP^+[\mathcal{A}_{x,x'}, \mathcal{A}_{y,y'}]=c_7^2,
$
where $y'=\gamma \cdot x'$ for an automorphism $\gamma\in \Gamma$ such that $\gamma\cdot x=y$. Thus, for every $y$ sufficiently far away from $x$, $\bfP^+[\mathcal{A}_{x,x'}, \mathcal{A}_{y,y'}]>c_7^2/2$. We now wish to modify our current so that it has sources $x,x',y,y'$. Assume without loss of generality that they are all distinct. Given $(\n,\ct)\in \mathcal{A}_{x,x'}\cap\mathcal{A}_{y,y'}$, let $(\n',\ct')$ be the tangled current with sources $x,x',y,y'$ defined as
\begin{equation*}
\n'(e)=\begin{cases}
     \n(e)-1, & \text{ for } e\in \{\lbrace x,x'\rbrace,\lbrace y,y'\rbrace\}, \\
     \n(e), & \text{ otherwise },
\end{cases}
\end{equation*}
where $\ct'$ coincides with $\ct$ up to redefining $\overline{x}, \overline{x}', \overline{y},\overline{y}'$ to be the extra vertices in the blocks added due to the sources, similarly to the proof of Proposition~\ref{prop:bridges}. Note that up to possibly exchanging $x$ and $x'$, and $y$ and $y'$, $(\n',\ct')\in \mathcal{B}_m(x,x',y,y')$ for some $m>0$, where 
$
\mathcal{B}_m(x,x',y,y'):=\lbrace (\n,\ct): \overline{x}\longleftrightarrow\overline{y} \text{ in } B_m, \: \overline{x}\centernot\longleftrightarrow\overline{x}', \: \overline{y}\centernot\longleftrightarrow\overline{y}' \rbrace,
$
and now $\overline{x}, \overline{x}', \overline{y},\overline{y}'$ denote the extra vertices in the blocks.
Thus a union bound implies that
$
\bfP^+[(\n',\ct')\in \mathcal{B}_m(x,x',y,y')]>c_7^2/8,
$
which in turn implies that
$
\bfP^{\emptyset}_{B_n\sqcup \{\fg\}}[(\n',\ct')\in \mathcal{B}_m(x,x',y,y')]>c_7^2/8
$
for all $n$ sufficiently large. After bounding the degrees, it follows that there exists a constant $c_8>0$ independent of $x,x',y,y'$ such that for every $n$ large enough
\begin{equation}\label{eq: proof final lemma 1}
\langle \varphi_x \varphi_{x'} \varphi_y \varphi_{y'} \rangle^\PLUS_{B_n}\bfP^{xx'yy'}_{B_n\sqcup \{\fg\}}[\mathcal{B}_m(x,x',y,y')]\geq c_8\bfP^{\emptyset}_{B_n\sqcup \{\fg\}}[(\n',\ct')\in \mathcal{B}_m(x,x',y,y')],
\end{equation}
hence for every $n$ sufficiently large,
\begin{equation*}
\begin{aligned}
\sum_{\substack{\partial \n=\{x,x',y,y'\}}} w_{\beta}^{xx'yy'}(\n) \rho_{\n}^{xx'yy'}&[\mathcal{B}_m(x,x',y,y')]
\\
&\geq c_8 \sum_{\substack{\partial \n={\emptyset}}} w_{\beta}^\emptyset(\n) \rho_{\n}^{\emptyset}[(\n',\ct')\in\mathcal{B}_m(x,x',y,y')].
\end{aligned}
\end{equation*} 
Our goal now is to prove that
\begin{multline}\label{eq: proof final lemma 2}
    \sum_{\substack{\partial \n=\{x,x',y,y'\}}} w_{\beta}^{xx'yy'}(\n) \rho_{\n}^{xx'yy'}[\mathcal{B}_m(x,x',y,y')]\leq\\ \langle \varphi_{x'}\varphi_{y'} \rangle^\PLUS_{B_n} \sum_{\substack{\partial \n=\{x,y\}}} w_{\beta}^{xy}(\n) \rho_{\n}^{xx'yy'}[\mathcal{B}_m(x,x',y,y')].
\end{multline}
Indeed, with this inequality and \eqref{eq: proof final lemma 1},
\begin{eqnarray*}
    \frac{c_7^2}{8}c_8&\leq& \langle \varphi_x \varphi_{x'}\varphi_y\varphi_{y'}\rangle_{B_n}^\PLUS\bfP^{xx'yy'}_{B_N\sqcup \{\fg\}}[\mathcal{B}_m(x,x',y,y')]\\ &\leq& \langle \varphi_{x'}\varphi_{y'}\rangle_{B_n}^\PLUS\langle \varphi_x\varphi_y\rangle_{B_n}^\PLUS\bfP^{xy}_{B_N\sqcup \{\fg\}}[ x \longleftrightarrow y \text{ in } B_m]\\ &\leq& \sqrt{\langle \varphi_{x'}^2\rangle_{B_n}^\PLUS\langle \varphi_{y'}^2\rangle_{B_n}^\PLUS}\langle \varphi_x\varphi_y\rangle_{B_n}^\PLUS\bfP^{xy}_{B_N\sqcup \{\fg\}}[ x \longleftrightarrow y \text{ in } B_m].
\end{eqnarray*}
This is sufficient to conclude since for $n\geq 1$ sufficiently large one has $\langle \varphi_{x'}^2\rangle_{B_n}^\PLUS,\langle \varphi_{y'}^2\rangle_{B_n}^\PLUS\leq 2\langle \varphi_x^2\rangle^+<\infty$.

We now prove \eqref{eq: proof final lemma 2}. To do this, we analyse the prelimit. Defining $h(\PLUS)_z:=\sum_{y\notin B_n}J_{z,y}\PLUS_y$ for all $z\in B_n$, we have that 
$
\mu_{B_n,N,\beta}^{h(\PLUS)}\left[\Phi_{x',N}\Phi_{y',N}\right]
$ tends to
$
\langle \varphi_{x'}\varphi_{y'}\rangle_{B_n}^\PLUS$ as $N\to\infty$,
where $\mu_{B_n,N,\beta}^{h(\PLUS)}$ has been defined in Definition \ref{def: measure gs graph}, and where we have used Proposition \ref{prop: CV griffiths}. Note that by the symmetries of the block-spins for any $i,j\in \lbrace 1,\ldots,N\rbrace$,
$
     \mu_{B_n,N,\beta}^{h(\PLUS)}[\sigma_{\tilde{x}'}\sigma_{\tilde{y}'}]= \mu_{B_n,N,\beta}^{h(\PLUS)}[\sigma_{(x',i)}\sigma_{(y',j)}],
$
where $\tilde{x}'=(x',1)$ and $\tilde{y}'=(y',1)$, so that the above convergence can be rephrased as
\begin{equation}\label{eq: convergence 2pIsing to 2pPhi4}
    (c_N N)^2\mu_{B_n,N,\beta}^{h(\PLUS)}\left[\sigma_{\tilde{x}'}\sigma_{\tilde{y}'}\right]\underset{N\rightarrow\infty}{\longrightarrow}\langle \varphi_{x'}\varphi_{y'}\rangle_{B_n}^\PLUS.
\end{equation}
Recall the definition of the Hamiltonian of the Ising measure $\mu_{B_n,N,\beta}^{h(\PLUS)}$ on $B_n\times K_N$ from Section~\ref{section: gs approx} and note that 
\begin{equation*}
\begin{aligned}
\mathbf{H}^{h(\PLUS)}_{B_n,\beta,N}(\sigma)
= 
-\sum_{\substack{u,v\in B_n\\ j,k\in \lbrace 1,\cdots, N\rbrace}}J_N((u,i),(v,j))\sigma_{(u,i)}\sigma_{(v,j)}-\beta c_N \sum_{(u,j) \in B_n\times K_N} h(\PLUS)_u \sigma_{(u,j)},
\end{aligned}
\end{equation*}
where $J_N$ is defined implicitly. Hence, we may as well write $\mu_{B_n,N,\beta,J_N}^{h(\PLUS)}=\mu_{B_n,N,\beta}^{h(\PLUS)}$. This notation allows us later to consider interactions other than $J_N$.

By Proposition~\ref{prop: stronger weight convergence}, the left-hand side of \eqref{eq: proof final lemma 2} can be obtained as the limit of
\begin{equation*}
    \frac{(c_N N)^4}{\mathbf{Z}_N^{|B_n|}}\sum_{\substack{\tilde{\n}\in \Omega^N_{B_n\sqcup\{\fg\}}\\\partial \tilde{\n}=\lbrace \tilde{x},\tilde{x}',\tilde{y},\tilde{y}'\rbrace}}w_{N,\beta}(\tilde{\n})\mathbbm{1}\{\mathcal{B}_m^N(\tilde{x},\tilde{x}',\tilde{y},\tilde{y}')\}
\end{equation*}
as $N$ goes to infinity, where $\tilde{z}=(z,1)$ for $z\in \lbrace x,x',y,y'\rbrace$ and
\begin{equation*}
    \mathcal{B}_m^N(\tilde{x},\tilde{x}',\tilde{y},\tilde{y}'):=\lbrace \tilde{\n}\in \Omega^N_{B_n\cup \lbrace\fg\rbrace}: \:\tilde{x}\overset{\tilde{\n}}{\longleftrightarrow}\tilde{y} \text{ in } B_m\times K_N, \: \tilde{x}\overset{\tilde{\n}}{\centernot\longleftrightarrow}\tilde{x}', \: \tilde{y}\overset{\tilde{\n}}{\centernot\longleftrightarrow}\tilde{y}' \rbrace.
\end{equation*} 
Now, denote by $\mathcal{C}_{\tilde{x}}$ the cluster of $\tilde{x}$. We say that a subset of $B_n\times K_N\cup\lbrace \fg\rbrace$ is admissible if it satisfies $\lbrace \cC_{\tilde{x}}=C\rbrace \subset \mathcal{B}_m^N(\tilde{x},\tilde{x}',\tilde{y},\tilde{y}')$. If $C$ is admissible, then
\begin{equation*}
    \sum_{\substack{\tilde{\n}\in \Omega^N_{B_n\sqcup\{\fg\}}\\\partial \tilde{\n}=\lbrace \tilde{x},\tilde{x}',\tilde{y},\tilde{y}'\rbrace}}w_{N,\beta,h(\PLUS)}(\tilde{\n})\mathbbm{1}\{\cC_{\tilde{x}}=C\}=\mu^{h^C(\PLUS)}_{B_n,N,\beta,J_N^C} [\sigma_{\tilde{x}'}\sigma_{\tilde{y}'}]\sum_{\substack{\tilde{\n}\in \Omega^N_{B_n\sqcup \{\fg\}}\\\partial \tilde{\n}=\{\tilde{x},\tilde{y}\}}}w_{N,\beta,h(\PLUS)}(\tilde{\n})\mathbbm{1}\{\cC_{\tilde{x}}=C\},
\end{equation*}
where $J_N^C$ is defined by 
\begin{equation*}
    J_N^C((u,i),(v,j))= \left\{
    \begin{array}{ll}
        J_N((u,i),(v,j)) & \mbox{if } (u,i),(v,j)\notin C \\
        0 & \mbox{otherwise.}
    \end{array}
    \right.
\end{equation*}
and $h^C(\PLUS)=0$ if $\fg\in C$ and $h^C(\PLUS)=h(\PLUS)$ otherwise.
Since $J_N^C\leq J_N$ and $h^C(\PLUS)\leq h(\PLUS)$, by Griffiths' inequality,
\begin{equation*}
    \mu^{h^C(\PLUS)}_{B_n,N,\beta,J_N^C} [\sigma_{\tilde{x}'}\sigma_{\tilde{y}'}]\leq \mu^{h(\PLUS)}_{B_n,N,\beta,J_N} [\sigma_{\tilde{x}'}\sigma_{\tilde{y}'}],
\end{equation*}
so that 
\begin{equation*}
    \sum_{\substack{\tilde{\n}\in \Omega^N_{B_n\sqcup\{\fg\}}\\\partial \tilde{\n}=\lbrace \tilde{x},\tilde{x}',\tilde{y},\tilde{y}'\rbrace}}w_{N,\beta}(\tilde{\n})\mathbbm{1}\{\cC_{\tilde{x}}=C\}\leq\mu^{h(\PLUS)}_{B_n,N,\beta,J_N} [\sigma_{\tilde{x}'}\sigma_{\tilde{y}'}]\sum_{\substack{\tilde{\n}\in \Omega^N_{B_n\sqcup \{\fg\}}\\\partial \tilde{\n}=\{\tilde{x},\tilde{y}\}}}w_{N,\beta}(\tilde{\n})\mathbbm{1}\{\cC_{\tilde{x}}=C\}.
\end{equation*}
Summing over $C$ admissible yields
\begin{equation}\label{eq: proof final lemma 3}
\begin{aligned}
    \sum_{\substack{\tilde{\n}\in \Omega^N_{B_n\sqcup\{\fg\}}\\\partial \tilde{\n}=\lbrace \tilde{x},\tilde{x}',\tilde{y},\tilde{y}'\rbrace}}w_{N,\beta}(\tilde{\n})\mathbbm{1}\{\mathcal{B}_m^N(\tilde{x},\tilde{x}',\tilde{y},\tilde{y}')\}\leq \mu^{h(\PLUS)}_{B_n,N,\beta,J_N} [\sigma_{\tilde{x}'}\sigma_{\tilde{y}'}] \\\times\sum_{\substack{\tilde{\n}\in \Omega^N_{B_n\sqcup\{\fg\}}\\\partial \tilde{\n}=\lbrace \tilde{x},\tilde{y}\rbrace}}w_{N,\beta}(\tilde{\n})\mathbbm{1}\{\mathcal{B}_m^N(\tilde{x},\tilde{x}',\tilde{y},\tilde{y}')\}.
\end{aligned}    
\end{equation}
Multiplying by $(c_N N)^4/\mathbf{Z}_N^{|B_n|}$ on each side of \eqref{eq: proof final lemma 3} and letting $N$ go to infinity  yields \eqref{eq: proof final lemma 2} by Proposition~\ref{prop: stronger weight convergence} and \eqref{eq: convergence 2pIsing to 2pPhi4}. 
\end{proof}

We now finish the proof of Theorem \ref{thm:main}. It remains to prove that \eqref{eq: condition for continuity} is sufficient to establish that $|\cG(\beta)|=1$. By Proposition \ref{prop: spont mag and gibbs is 0} it suffices to show that $m^*(\beta)=0$. From the preceding arguments, we know that $\langle \cdot \rangle^0$ and $\langle \cdot \rangle^+$ coincide on even functions, so condition \eqref{eq: condition for continuity} reads
\begin{equation*}
    \liminf_{d_G(o,x)\rightarrow \infty}\langle \varphi_o\varphi_{x}\rangle^0_\beta=\liminf_{d_G(o,x)\rightarrow \infty}\langle \varphi_o\varphi_{x}\rangle^+_\beta=0.
\end{equation*}
However, by Griffiths' inequality and $\Gamma$-invariance, for all $x\in V$,
$
    \langle \varphi_o\varphi_{x}\rangle^+_\beta\geq m^*(\beta)^2
$
which implies that $m^*(\beta)=0$. $\hfill \square$

\section{Continuity of the phase transition}
\label{sec: continuity}

In this section, we prove Theorem \ref{thm: continuity}. Fix a reflection positive interaction $J$ on $\ZZ^d$ in the sense of \cite[Definition 5.1 and Lemma 5.5]{B} (see also \cite[Section~3]{panis2023triviality} for a complete introduction to the subject in the case of the $\varphi^4$ model). We suppress $J$ from the notation when clear from context. For each such $J$, recall the definition of the associated random walk $(X_k)_{k\geq 0}$ (started at $0$) with step distribution given by
\begin{equation*}
    \mathbb P[X_{n+1}=y\: |\: X_n=x]=\frac{J_{x,y}}{|J|}
\end{equation*}
where $|J|:=\sum_{x\in \mathbb Z^d}J_{0,x}$. It is known that $(X_k)_{k\geq 0}$  is transient if and only if
\begin{equation}\label{eq: condition transience}
\int_{[-\pi,\pi)^d}\frac{1}{1-\mathbb E[e^{ip\cdot X_1}]}\frac{\textup{d}p}{(2\pi)^d}<\infty.
\end{equation}
A proof of this statement can be found in \cite[Proposition 4.2.3]{lawler2010}; it crucially uses the reflection symmetry and aperiodicity of the random walk.

We now state a version of the infrared bound \cite{FSS} that is suited to our setting.
\begin{prop} \label{prop: irb}
Let $J$ be reflection positive and such that the associated random walk is transient, i.e. \eqref{eq: condition transience} holds. For any $(v_x)_{x \in \ZZ^d} \in \CC^{\ZZ^d}$ of finite support, and any $\beta < \beta_c(J)$,
\begin{equation}\label{eq: improved IRB}
    \sum_{x,y\in \ZZ^d}v_x\overline{v_y}\langle \varphi_x\varphi_y\rangle_{\beta}^0\leq \frac{1}{2\beta|J|}\sum_{x,y\in \mathbb Z^d}v_x\overline{v_y}G(x,y)
\end{equation}
where $G$ is the Green's function of $(X_k)_{k \in \NN}$ given by
\begin{equation*}
    G(x,y):=\lim_{L\rightarrow \infty}\frac{1}{L^d}\sum_{p\in \mathbb T_L^*}\frac{e^{ip\cdot(x-y)}}{1-\mathbb E[e^{ip\cdot X_1}]}=\int_{[-\pi,\pi)^d}\frac{e^{ip\cdot(x-y)}}{1-\mathbb E[e^{ip\cdot X_1}]}\frac{\textup{d}p}{(2\pi)^d}<\infty.
\end{equation*}
\end{prop}
\begin{proof}
The convergence and finiteness of $G(x,y)$ is a consequence of \eqref{eq: condition transience}. The infrared bound \eqref{eq: improved IRB} is a direct consequence of \cite[Lemma 4.5]{B} together with the fact that the spontaneous magnetisation vanishes for $\beta < \beta_c$. We stress a subtle point in the reasoning above: one can show, by using similar regularity estimates as in Corollary \ref{app cor: finite vol reg} and arguing as in the proof of \cite[Theorem 4.5]{LP76}, that the $\phi^4$ measures on $\TT_L^d$ converge weakly as $L \rightarrow \infty$ to the unique (since $\beta<\beta_c$, see Proposition~\ref{prop: spont mag and gibbs is 0}) Gibbs measure on $\ZZ^d$.
\end{proof}

We now turn to the proof of Theorem \ref{thm: continuity}.
\begin{proof}[Proof of Theorem \ref{thm: continuity}]
By left-continuity of $\beta\mapsto \langle \varphi_x\varphi_y\rangle_\beta^0$ for $x,y\in \ZZ^d$, we can also extend \eqref{eq: improved IRB} to $\beta=\beta_c$. Using equation \eqref{eq: improved IRB} with $v$ defined by $v_x=|B_n|^{-2}\mathbbm{1}_{B_n}(x)$ yields
\begin{equation}\label{eq: cesaro means}
    \frac{1}{|B_n|^2}\sum_{x,y\in B_n}\langle \varphi_x\varphi_y\rangle_{\beta_c}^0\leq \frac{1}{2\beta_c|J|}\frac{1}{|B_n|^2}\sum_{x,y\in B_n}G(x,y).
\end{equation}
The Riemann--Lebesgue Lemma gives that $G(x,y)\underset{|x-y|_1\rightarrow \infty}{\longrightarrow}0$. In particular,
\begin{equation*}
    \frac{1}{|B_n|^2}\sum_{x,y\in B_n}G(x,y)\underset{n\rightarrow \infty}{\longrightarrow}0.
\end{equation*}
By positivity of correlations we have that the lefthand side of \eqref{eq: cesaro means} converges to $0$, which in turn implies that \eqref{eq: condition for continuity} holds at $\beta=\beta_c$. By Theorem~\ref{thm:main}, one concludes that $|\cG_\Gamma(\beta_c)|=1$, and therefore $m^*(\beta_c)=0$ by Proposition~\ref{prop: spont mag and gibbs is 0}.
\end{proof}

\appendix

\section{Right-continuity of the spontaneous magnetisation}

\begin{prop}\label{prop:rightcontinuity}
The spontaneous magnetisation $m^*(\beta)$ is right-continuous at any $\beta>0$.    
\end{prop}
\begin{proof}
Let $\beta>0$. Then,for any $\Lambda\subset V$ with $\Lambda\ni o$ finite we have
\begin{equation*}
\begin{aligned}
\limsup_{\delta\to 0^+}m^*(\beta+\delta)&\leq \limsup_{\delta\to 0^+} \langle \varphi_o \rangle^{\PLUS}_{\Lambda,\beta+\delta}+\varepsilon(\Lambda,\beta)
=\langle \varphi_o \rangle^{\PLUS}_{\Lambda,\beta}+\varepsilon(\Lambda,\beta), 
\end{aligned}
\end{equation*}
where $\varepsilon(\Lambda,\beta)=\sup_{\delta\in [0,1]}\langle \varphi_o \mathbbm{1}_{\exists x \notin \Lambda: \: |\varphi_x|> \PLUS_x(M)} \rangle^{+}_{\beta+\delta}$. Here we used the monotonicity in boundary conditions (see Proposition \ref{prop: increasing in increasing bc}) and that $\langle \varphi_o \rangle^{\PLUS}_{\Lambda,\beta+\delta}$ is a continuous function of $\delta$. Since $\varepsilon(\beta,\Lambda)$ tends to $0$ as $\Lambda$ increases to $V$ by the continuity of the regularity constants--- see Proposition~\ref{prop: gibbs is regular}--- we have
$
\limsup_{\delta\to 0^+}m^*(\beta+\delta)\leq m^*(\beta).
$
On the other hand, we have that $m^*(\beta+\delta)\geq m^*(\beta)$ for any $\delta>0$, hence 
$
\lim_{\delta\to 0^+}m^*(\beta+\delta)=m^*(\beta),
$
as desired.
\end{proof}

\section{Correlation inequalities}
In this section, all the boundary conditions considered are implicitly assumed to satisfy the summability condition defined in the introduction, i.e.~$\sum_{y\notin \Lambda}J_{x,y}\eta_y<\infty ~~ \forall x\in \Lambda$.

\subsection{The FKG and absolute value FKG inequalities}

\begin{prop}[FKG inequality, {\cite[Theorem 4.4.1]{glimm2012quantum}}] \label{prop: FKG phi4}
Let $\Lambda\subset V$ finite, $\eta\in \RR^V$ and $h\in \mathbb R^\Lambda$. Then, for any increasing and bounded functions $F,G:\RR^\Lambda\to \RR$,
\begin{equation*}
    \langle F(\phi)G(\phi)\rangle_{\Lambda,\beta,h}^\eta\geq \langle F(\phi)\rangle_{\Lambda,\beta,h}^\eta\langle G(\phi)\rangle_{\Lambda,\beta,h}^\eta.
\end{equation*}
\end{prop}

\begin{prop}[Absolute value FKG]\label{absolute FKG}
Let $\Lambda\subset V$ finite, $\eta\in \RR_+^V$ and $h\in \mathbb R_+^\Lambda$. Then, for any bounded increasing functions $F,G:(\RR_+)^\Lambda\to \RR$,
\begin{equation*}
    \langle F(|\varphi|)G(|\varphi|) \rangle^{\eta}_{\Lambda,\beta,h}\geq \langle F(|\varphi|) \rangle^{\eta}_{\Lambda,\beta,h} \langle G(|\varphi|) \rangle^{\eta}_{\Lambda,\beta,h}.
\end{equation*}
\end{prop}
\begin{proof}
We rewrite the Hamiltonian as
\begin{equation*}
    H^\eta_{\Lambda}(\varphi)
    :=
    \sum_{x,y\in \Lambda}V_{x,y}(\varphi_x-\varphi_y)-\sum_{x,y\in \Lambda}J_{x,y}\frac{\varphi_x^2+\varphi_y^2}{2}-\sum_{\substack{x\in \Lambda\\y\in V\setminus \Lambda}}J_{x,y}\varphi_x\eta_y - \sum_{x\in \Lambda} h_x\phi_x,
\end{equation*}
where $V_{x,y}(z)=\frac{1}{2}J_{x,y}z^2$. Note that $V_{x,y}$ is convex, symmetric, i.e.\ $V_{x,y}(-z)=V_{x,y}(z)$, and the second derivative of $V_{x,y}$ is a non-increasing function.
The desired assertion follows from \cite[Corollary 6.4]{LO21}. 
\end{proof}

\subsection{Proof of the Ginibre inequality}\label{appendix: ginibre}
 Recall that if $A:V\rightarrow \NN$ is finitely supported, we write
$
    \varphi_A=\prod_{x\in V}\varphi_x^{A_x}.
$
Let $A,B: V\rightarrow \NN$ be finitely supported. Let $\Lambda$ be a finite subset of $V$ that contains the support of $A$ and $B$, and let $|\eta|\leq\eta'$ outside $\Lambda$. We want to prove that
\begin{equation*}
    \langle \varphi_A\varphi_B\rangle^{\eta'}_{\Lambda,\beta}-\langle \varphi_A\varphi_B\rangle^\eta_{\Lambda,\beta}\geq \left|\langle \varphi_A\rangle^{\eta'}_{\Lambda,\beta}\langle \varphi_B\rangle^\eta_{\Lambda,\beta}-\langle \varphi_A\rangle_{\Lambda,\beta}^\eta\langle\varphi_B\rangle^{\eta'}_{\Lambda,\beta}\right|.
\end{equation*}
Let $h_x:=\sum_{y\notin \Lambda}J_{x,y}\eta_y$ and $h_x':=\sum_{y\notin \Lambda}J_{x,y}\eta_y'$. Then, with the notation of Section \ref{section: gs approx}, using the Ginibre inequality of the Ising model on $\Lambda_N=\Lambda\times K_N$, and linearity,
\begin{multline*}
   \Big|\langle \Phi_{A,N} \rangle^0_{\Lambda_N,\beta,h'}\langle  \Phi_{B,N}\rangle^0_{\Lambda_N,\beta,h}-\langle \Phi_{A,N} \rangle^0_{\Lambda_N,\beta,h}\langle  \Phi_{B,N}\rangle^0_{\Lambda_N,\beta,h'}\Big|\\\leq\langle \Phi_{A,N} \Phi_{B,N}\rangle^0_{\Lambda_N,\beta,h'}-\langle \Phi_{A,N} \Phi_{B,N}\rangle^0_{\Lambda_N,\beta,h}
\end{multline*}
where
$
    \Phi_{A,N}:=\prod_{x\in \Lambda}\Phi_{x,N}^{A_x},  \Phi_{B,N}:=\prod_{x\in \Lambda}\Phi_{x,N}^{B_x}.
$
Now, taking the limit as $N$ goes to infinity and using Proposition \ref{prop: CV griffiths}, we get the desired result. Note that we used the interpretation of boundary conditions as non-homogeneous external fields. Moreover, if $\langle \cdot\rangle_\beta\in \cG(\beta)$,
\begin{equation}\label{eq: appendix ginibre 2}
    \langle \varphi_A\varphi_B\rangle^+_{\beta}-\langle \varphi_A\varphi_B\rangle_\beta\geq \left|\langle \varphi_A\rangle^+_{\beta}\langle \varphi_B\rangle_\beta-\langle \varphi_A\rangle_\beta\langle\varphi_B\rangle^+_{\beta}\right|.
\end{equation}
Indeed, if we fix $\eta'=\PLUS$ above and integrate against $\langle \cdot\rangle_\beta$ among $\eta$ satisfying $|\eta|\leq \PLUS$ outside $\Lambda$,
\begin{multline*}
    \langle \varphi_A\varphi_B\rangle^{\PLUS}_{\Lambda,\beta}\langle\mathbbm{1}_{\forall x \notin \Lambda, \: |\eta_x|\leq \PLUS_x}\rangle_\beta-\Big\langle\langle \varphi_A\varphi_B\rangle^\eta_{\Lambda,\beta}\mathbbm{1}_{\forall x \notin \Lambda, \: |\eta_x|\leq \PLUS_x}\Big\rangle_\beta\\\geq \Big|\langle \varphi_A\rangle^{\PLUS}_{\Lambda,\beta}\Big\langle\langle \varphi_B\rangle^\eta_{\Lambda,\beta}\mathbbm{1}_{\forall x \notin \Lambda, \: |\eta_x|\leq \PLUS_x}\Big\rangle_\beta-\Big\langle\langle \varphi_A\rangle^\eta_{\Lambda,\beta}\mathbbm{1}_{\forall x \notin \Lambda, \: |\eta_x|\leq \PLUS_x}\Big\rangle_\beta\langle\varphi_B\rangle^{\PLUS}_{\Lambda,\beta}\Big|.
\end{multline*}
Using the results of Section \ref{sec: gibbs states} and taking $\Lambda\uparrow V$ yields \eqref{eq: appendix ginibre 2}.
\hfill $\square$

\subsection{Monotonicity in boundary conditions}
\label{subsec: monotone}

We have the following finite volume monotonicity properties for correlations and increasing functions. They follow from the Ginibre inequality and the FKG inequality respectively.
\begin{prop}
Let $\Lambda\subset V$ finite. Let $A: \Lambda\rightarrow \mathbb N$ be a degree function. For $|\eta| \leq \eta'$,
\begin{equation*}
\langle \varphi_A \rangle^\eta_{\Lambda,\beta}
\leq
\langle \varphi_A \rangle^{\eta'}_{\Lambda,\beta}.
\end{equation*}
\end{prop}

\begin{prop}\label{prop: increasing in increasing bc}
Let $\eta \leq \eta'$. Then, for any $f$ increasing and integrable,
\begin{equation*}
\langle f \rangle^\eta_{\Lambda,\beta}
\leq 
\langle f \rangle^{\eta'}_{\Lambda,\beta}.
\end{equation*}
\end{prop}
\subsection{Single-site moment bounds}

\begin{prop}\label{bound single-site}
There exists $C>0$ such that for all $k\geq 0$, 
$    
\langle\varphi^{2k}\rangle_{0}\leq C\langle \varphi^{2k+2}\rangle_{0}.
$    
\end{prop}
\begin{proof}
Note that there exists $C>0$ such that for every $k\geq 0$ we have
\begin{eqnarray*}
    \dfrac{\langle \varphi^{2k}\rangle_0}{\langle \varphi^{2k+2}\rangle_0}&\leq&
    \dfrac{\int_{-1}^1 s^{2k}\rho(\text{d}s)}{\int_{-1}^1s^{2k+2}\rho(\text{ds})}+1 \leq \dfrac{M}{m}\dfrac{2k+3}{2k+1}+1\leq C,
\end{eqnarray*}
where
$
    M:=\max_{s\in[-1,1]}e^{-gs^4-as^2} \text{ and }m:=\min_{s\in [-1,1]}e^{-gs^4-as^2}.
$
\end{proof}

 \section{Uniqueness of the phase transition}\label{app: different pt}
Define,
\begin{align}
    \beta_c^{\textup{MAG}}&:=\inf\Big\lbrace \beta>0: \text{ } \langle \varphi_o\rangle^+_\beta>0\Big\rbrace,
    \\
    \beta_c^{\textup{LRO}}&:=\inf\Big\lbrace \beta>0:\text{ } \inf_{x\in V}\langle \varphi_o\varphi_x\rangle^0_\beta>0\Big\rbrace,
    \\
    \beta_c^{\textup{SUSC}}&:=\inf\Big\lbrace \beta>0: \: \sum_{x\in V}\langle \varphi_o\varphi_x\rangle_\beta^0=\infty \Big\rbrace.
\end{align}

\begin{prop}
Let $G$ be a countably infinite and vertex-transitive graph of polynomial growth. Then,
\begin{equation*}
    \beta_c^{\textup{MAG}}=\beta_c^{\textup{LRO}}=\beta_c^{\textup{SUSC}}.
\end{equation*}
\end{prop}
\begin{proof}
The assertion that $\beta_c^{\textup{MAG}} = \beta_c^{\textup{SUSC}}$ was established in \cite{aizenman1987sharpness}. We now prove that $\beta_c^{\textup{MAG}}=\beta_c^{\textup{LRO}}$. Let $\beta>\beta_c^{\textup{LRO}}$ and write $c:=\inf_{x\in V}\langle \varphi_o\varphi_x\rangle^0_\beta>0$. Then, by monotonicity in the boundary conditions and the mixing of $\langle \cdot \rangle^+_{\beta}$ we have
\begin{equation*}
c\leq \langle \varphi_o \varphi_x\rangle^0_{\beta}\leq \langle \varphi_o \varphi_x\rangle^+_{\beta}
\underset{d_G(o,x)\rightarrow \infty}{\longrightarrow}\big[\langle \varphi_o\rangle^+_{\beta}\big]^2.
\end{equation*}
This implies that $\langle \varphi_o\rangle^+_{\beta}>0$. Let now $\beta<\beta_c^{\textup{LRO}}$. Consider some $\beta<\beta'< \beta_c^{\textup{LRO}}$. It follows from \eqref{eq: ineq a b} that 
$
\langle \varphi_o \varphi_x\rangle^+_{\beta}\leq \langle \varphi_o \varphi_x\rangle^0_{\beta'}.
$
Letting $d_G(o,x)\rightarrow\infty$ and using the mixing of $\langle \cdot \rangle^+_{\beta}$ as above we obtain that 
$\langle \varphi_o\rangle^+_{\beta}=0$. The desired assertion follows.
\end{proof}

\begin{prop}\label{prop: spont mag and gibbs is 0}
Let $G$ be as in the proposition above. Let $\beta>0$. The following are equivalent,
\begin{enumerate}
    \item[$(i)$] $|\cG(\beta)|=1$,
    \item[$(ii)$] $\langle \cdot \rangle^+_\beta=\langle \cdot \rangle^-_\beta$,
    \item[$(iii)$] $\langle \varphi_o\rangle^+_\beta=\langle \varphi_o \rangle^-_\beta=0$.
\end{enumerate}
\end{prop}
\begin{proof}
The implication $(i)\implies (ii)$ is trivial. For the implication $(ii)\implies (iii)$, note that if $\phi\sim \langle \cdot \rangle^+_{\beta}$, then $-\phi\sim \langle \cdot \rangle^{-}_{\beta}$. Hence, $\langle\varphi_o\rangle^+_\beta=-\langle \varphi_o \rangle^-_\beta$, which implies that $\langle \varphi_o\rangle^+_\beta=\langle \varphi_o \rangle^-_\beta=0$.

Finally, to show that $(iii)\implies (i)$, let $\langle \cdot \rangle_{\beta}\in \cG(\beta)$. Using Proposition~\ref{prop: construction of plus and minus} and Strassen's theorem \cite{Strassen} we obtain a coupling $(\PP, \varphi^+, \varphi)$, $\varphi^+\sim \langle\cdot \rangle^+_{\beta}$, $\varphi\sim \langle\cdot \rangle_{\beta}$, such that $\PP$-almost surely, $\varphi^+_x\geq \varphi_x$ for every $x\in V$. Again by Proposition \ref{prop: construction of plus and minus} we have $\langle\varphi_x\rangle_\beta=\langle\varphi_x\rangle^+_\beta=0$ for every $x\in V$. It follows that $\PP$-almost surely, $\varphi^+_x=\varphi_x$ for every $x\in V$, hence $\langle \cdot \rangle_{\beta}=\langle \cdot \rangle^+_{\beta}$.
\end{proof}

\section{Probabilistic estimates for Gibbs measures} \label{app: LP}

In this section, we develop the probabilistic estimates for Gibbs measures associated to the $\phi^4$ model on graphs of polynomial growth (see also \cite[Chapter~11]{PanThesis} for a slightly more general setting). The results in this section follow from adaptations of the foundational papers on the subject \cite{LP76} and \cite{R70}, which establish the theory for more general measures on $\ZZ^d$.
We optimise our estimates to the case of $\phi^4$, which has stronger concentration properties than the class considered in \cite{LP76}.

\subsection{Setup and main results}

We start by recalling some notations and definitions. We work on a graph $G=(V,E)$ which is connected, vertex-transitive graph and has polynomial growth, defined as follows. For a distinguished origin $o\in V$, we write $B_k$ to denote the ball of radius $k \in \NN$ centred at $o$ with respect to the graph distance $d_G(\cdot,\cdot)$. We assume that there exists $d\geq 1$ (called the dimension of $G$) and $c_G,C_G>0$ such that
\begin{equation}\label{eq: polygrowth constants}
    c_G k^d
    \leq 
    |B_k|
    \leq 
    C_G k^d, \qquad \forall k \geq 1.
\end{equation}

\begin{rem}
Usually a graph is said to have (at most) polynomial growth if there are constants $d\geq 1$ and $C_G>0$ such that the upper bound $|B_k|\leq C_G k^d$ is satisfied for every $k\geq 1$. That one can always find $c_G, C_G$ and $d$ such that \eqref{eq: polygrowth constants} holds follows from Gromov's theorem \textup{\cite{gromov1981groups}}.
\end{rem}
\begin{rem}\label{rem: delta}
It was proved in \textup{\cite{RT}}  (see also \textup{\cite{BD}} for a relevant result) that for every vertex-transitive graph of polynomial growth there are constants $C>0$ and $\delta>0$ such that
$|B_n\setminus B_{n-1}|\leq C n^{-\delta} |B_n|
\leq 
C n^{-\delta}$
for every $n\geq 1$.
In fact, this was proved under the weaker assumption that the graph satisfies a (uniform) doubling property, i.e.\ without the assumption of vertex-transitivity, which is easy to check in our case.
To the best of our knowledge, it is not known whether one can take $\delta=1$ for all vertex-transitive graphs of polynomial growth.
\end{rem}
\begin{rem}
The complications in generalising the results of \textup{\cite{LP76, R70}} to the context of vertex-transitive graphs of polynomial growth arise from the fact that it is not known whether one can take $\delta=1$ in the remark above. The results of the aforementioned articles apply to a much more general class of unbounded spin systems, which includes in particular the case $g=0$ for some $a$ chosen appropriately, and generalising our results to the whole class of unbounded spin systems considered there without a good understanding of the underlying graph structure seems rather complicated. The fact that we have a quartic term in our single site measure, i.e.\ that $g>0$, plays an important role in our analysis.
\end{rem}

We fix a set of interactions $J=(J_{x,y})_{x,y \in V}$ such that $J_{x,y} \geq 0$ for all $x,y \in V$ and 
$
J_{x,y} \leq \Psi(d_G(x,y)), \text{ for all $x\neq y \in V$},
$
where $\Psi : \NN \rightarrow \RR_+$ is defined as $\Psi(k)=C_\Psi k^{-d-\varepsilon}$ for some constants $C_\Psi>0$ and $\varepsilon>0$. In addition, we fix $(\beta,h,g,a)$. Recall that for $\eta \in \RR^V$, the $\phi^4$ model with these parameters on $\Lambda \subset V$ is the measure $\langle \cdot \rangle^\eta_{\Lambda,\beta,h}$ defined by \eqref{eq: def}. In this section, we only keep the dependence on $\Lambda$ explicit and it is sometimes convenient to write $\nu_\Lambda^\eta$ to denote the corresponding measure. 

For a Borel probability measure $\lambda$ on $\RR^V$, we write $\nu_\Lambda^\lambda$ to denote the corresponding measure with random boundary conditions sampled from $\lambda$. Let 
\begin{equation*}
    \Xi(M) = \{ \phi\in \RR^V : \, \exists  \Lambda_\phi \subset V \text{ finite},~ \phi^4_x \leq M \log d_G(o,x)  ~~\forall x \in V\setminus \Lambda_\phi \}.
\end{equation*}
We call $\lambda$ a law on weakly growing boundary conditions if $\lambda$ is a Borel probability measure on $\RR^V$ supported on $\Xi(M)$ for some $M > 0$. Finally, for $\Delta \subset \Lambda$, we write $\nu_{\Delta}^{(\Lambda), \lambda}$ to denote the restriction of $\nu_\Lambda^\lambda$ to $\RR^\Delta$, i.e.\ the measure defined by
\begin{equation*}
\nu_{\Delta}^{(\Lambda),\lambda}(\phi_\Delta)
=
\frac{1}{Z_\Lambda^\lambda}\int_{\RR^{\Lambda\setminus \Delta}} e^{-U^\lambda_\Lambda(\phi_\Delta \sqcup \phi_{\Lambda \setminus \Delta})} \textup{d}\phi_{\Lambda \setminus \Delta},
\end{equation*}
where $\phi_\Delta \sqcup \phi_{\Lambda \setminus \Delta}$ denotes the concatenation of two configurations and
\begin{equation*}
U^\lambda_\Lambda(\phi)
=
\beta H^\lambda_\Lambda(\phi) + \sum_{x \in \Lambda} \Big( g\phi_x^4 + a\phi_x^2 \Big)
\end{equation*}
is the potential. We now recall the definition of regular measures. Let $\nu$ be a measure on $\RR^V$. For $\gamma, \delta > 0$, we say $\nu$ is $(\gamma,\delta)$-regular if, for every $\Lambda \subset V$ finite, the density of the projection measure satisfies
\begin{equation*}
\frac{\textup{d}\nu}{\textup{d}\phi_\Lambda}(\phi_{\Lambda})
\leq
e^{-\sum_{x \in \Lambda} (\gamma \phi_x^4 - \delta)},
\end{equation*}
where $\textup{d}\phi_\Lambda$ is the Lebesgue measure on $\RR^\Lambda$. Note that, for any $\Lambda \subset V$ finite and $\lambda$ a law on weakly growing boundary conditions, $\nu_\Lambda^\lambda$ is $(\gamma,\delta)$-regular with constants depending only on $(\beta,h,g,a)$. Moreover, the constants are continuous in these dependencies and uniform over all subsets of $\Lambda$. This is a direct consequence of repeatedly applying H\"older's inequalities on $|H_\Lambda(\phi)|$.

Finally, we restate the definition of a Gibbs measure. Recall that a Borel probability measure $\nu$ on $\RR^V$ is \textit{tempered} if:
\begin{itemize}
\item for every $\Lambda \subset V$ finite, $\nu|_{\Lambda}$ is absolutely continuous with respect to Lebesgue measure on $\RR^\Lambda$,
\item $\nu(R)=1$ where
\begin{equation}\label{def: temperedness const}
R
=
\Big\{ (\varphi_x)_{x\in V} : \exists C\in(0,\infty), \sum_{x \in B_k} \varphi_x^2 \leq C|B_k| \text{ for all } k \in \NN \Big\}.
\end{equation}
\end{itemize}
Furthermore, recall that a Borel probability measure $\nu$ on $\mathbb R^{V}$ is a Gibbs measure at $(\beta,h)$ if it is tempered and for every finite $\Lambda\subset V$ and $f\in \RR^{\Lambda}$ bounded and measurable, the DLR equation
\begin{equation*}
    \nu(f)
    =
    \int_{\eta\in \RR^V}\langle f\rangle^\eta_{\Lambda,\beta,h}\textup{d}\nu(\eta)
\end{equation*}
holds. We call $\cG(\beta,h)$ the set of Gibbs measures. The main results of this section are the following. 
\begin{prop} \label{appendix prop: tightness}
Let $(\nu_{\Lambda_n})_{n\geq 0}$ be a sequence of finite volume measures with (random) weakly growing boundary conditions. Then, the sequence is tight and any limit point is a tempered probability measure.
\end{prop}
\begin{prop} \label{appendix prop: dlr}
Let $(\nu_{\Lambda_n})_{n\geq 0}$ be a sequence of finite volume measure with (random) weakly growing boundary conditions such that $\nu_{\Lambda_n}$ converges weakly to a tempered probability measure $\nu$. Then, $\nu$ is a Gibbs measure, i.e.\ satisfies the DLR equations.
\end{prop}
\begin{prop} \label{appendix prop: regular}
For any $(\beta,h)$, there exist $\gamma=\gamma(\beta,h), \delta=\delta(\beta,h)>0$ such that every Gibbs measure $\nu\in \cG(\beta,h)$ is $(\gamma,\delta)$-regular. Moreover, the constants $(\gamma,\delta)$ depend continuous on $(\beta,h)$.
\end{prop}

The proofs of Propositions \ref{appendix prop: tightness}--\ref{appendix prop: regular} are based on of the proofs of Theorems 4.3, 4.5, and 4.4 in \cite{LP76}, respectively\footnote{Note that the proof of \cite[Theorem 4.5]{LP76} is given in the appendix of \cite{LP76}, where it is mislabelled as proof of Theorem 4.4.}, adapted from the case of $\ZZ^d$ to our setting. The extensions of these proofs require us to develop some technical results, which broadly speaking allow us to distinguish as well as quantify bulk and boundary contributions from the potential. The key geometric estimate is contained in Proposition \ref{largest-q}. The key probabilistic estimate is contained in Proposition \ref{appendix prop: prob}, which in particular implies uniform regularity estimates for finite volume measures, see Corollary \ref{app cor: finite vol reg}. Both key estimates require multiscale control over the decay and oscillations of the kernel $\Psi$--- which in turn controls the decay of the interactions $J$--- on scales determined by the geometry of the graph; this naturally imposes conditions jointly on the allowed graphs and interactions. For graphs of polynomial growth, these conditions are explicit. Since, given Proposition \ref{appendix prop: prob} and Corollary \ref{app cor: finite vol reg}, the proofs of Propositions \ref{appendix prop: tightness}--\ref{appendix prop: regular} are identical to \cite{LP76}, we omit them and refer to the exact references given above.

\subsection{Technical results on bulk-boundary contributions} 

We first verify that $\Psi$ is indeed summable on $G$.
\begin{lem}\label{lem: psi sum}
We have that
\begin{equation*}
\| \Psi \|
:=
\sup_{x \in V} \sum_{y \in V} \Psi( d_G(x,y))
=
\sum_{y \in V}\Psi( d_G(o,y))
<
\infty.
\end{equation*}
\end{lem}
\begin{proof}
By summation by parts and $\Gamma$-invariance, we have that 
\begin{equation*}
\sum_{y\in V} \Psi(d_G(x,y))
=
\sum_{k=0}^\infty (\Psi(k)-\Psi(k+1))|B_k|.
\end{equation*}
Note that $\Psi(k)-\Psi(k+1)=O(k^{-(d+1+\varepsilon)})$, from which the result follows.
\end{proof}

We are interested in quantifying the oscillatory behaviour of $\Psi$ at large scales. In our analysis we require the scales to have a nonlinear dependence depending on the dimension. Let $\alpha=\frac{d}{\varepsilon}+1$ and define $F_n:=B_{t_n}$, where 
$t_n=\lfloor n^{\alpha} \rfloor$. We write $V_n=|F_n|$. We often refer to $n$ as the scale even though, strictly speaking, the scale is $t_n$. 

For $q,k\in \NN$, define
\begin{equation*}
\Psi_{k,q}
=
\sup_{x\in F_{q},\, y\in V\setminus F_{q+k}} \Psi(d_G(x,y)).
\end{equation*}
The coefficient $\Psi_{k,q}$ tracks the oscillation of $\Psi$ from scale $q$ to scale $q+k$. The next lemma states that the decay condition on $\Psi$ is enough to ensure that asymptotically the oscillations are small with respect to the volume.

\begin{lem}\label{lem: summable}
We have that 
\begin{equation}\label{fast-decay}
\lim_{q\to\infty}\sum_{k=1}^\infty \left(\Psi_{k,q}-\Psi_{k+1,q}\right) V_{q+k+1}
=
0.
\end{equation}
\end{lem}

\begin{proof}
By summation by parts
\begin{equation*}
\sum_{k=1}^\infty \left(\Psi_{k,q}-\Psi_{k+1,q}\right) V_{q+k+1}
\leq 
\Psi_{1,q}V_{q+2}+C_G\sum_{k=2}^\infty \Psi_{k,q} \left(t_{q+k+1}^d-t_{q+k}^d\right)
\end{equation*}
where $C_G$ is as in \eqref{eq: polygrowth constants}.
Note that $\Psi_{k,q}=C_\Psi(t_{q+k}+1-t_q)^{-d-\varepsilon}$. On one hand, by telescoping and a simple calculation
\begin{equation*}
t_{q+k}+1-t_q\geq\sum_{i=1}^k t_{q+i}-t_{q+i-1}\geq c_1\sum_{i=\lfloor k/2 \rfloor}^k (q+i)^{\alpha-1} \geq c_2 k(q+k)^{\alpha-1}
\end{equation*}
for some constants $c_1,c_2>0$, hence $\Psi_{k,q}\leq C_\Psi(c_2k)^{-d-\varepsilon}(q+k)^{-(\alpha-1)(d+\varepsilon)}$. On the other hand, 
$t_{q+k+1}^d-t_{q+k}^d=O((q+k)^{\alpha d-1})$. Since $(\alpha-1)(d+\varepsilon)>\alpha d -1$, the desired assertion follows.
\end{proof}

We now introduce a sequence $(\psi_j)_{j\in \NN}$ that plays the role of $C$ in \eqref{def: temperedness const} from the definition of tempered measures. It is chosen to increase from scale to scale sufficiently slowly so that the following holds.

\begin{lem}\label{lem: psi def}
There exists an increasing sequence $(\psi_j)_{j\in \mathbb{N}}$ converging to infinity such that $\psi_1\geq 1$, $
\lim_{j\to\infty} \frac{\psi_{j+1}}{\psi_j}
=
1$,
and
\begin{equation}\label{conv-0}
\lim_{q\to\infty}\sum_{k=1}^\infty \left(\Psi_{k,q}-\Psi_{k+1,q}\right)\psi_{q+k+2} V_{q+k+2}
=
0.
\end{equation}
\end{lem}

\begin{proof}
Let $(k_n)_{n\in \NN}$ be a strictly increasing sequence of positive integers such that 
\begin{equation*}
\sum_{k=m_{n,q}}^\infty \left(\Psi_{k,q}-\Psi_{k+1,q}\right)V_{q+k+2}
\leq
2^{-n} \quad \text{for every } q\geq 1,
\end{equation*}
where $m_{n,q}:=\max\{1,k_n-q-2\}$. Such a sequence exists because of \eqref{fast-decay} together with the fact that $V_{q+k+2}=O(V_{q+k+1})$. Define
\begin{equation*}
    \psi_j:= 
     \begin{cases}
    1, & \text{for } j< k_1 \\
    n, & \text{for } k_n\leq j< k_{n+1}, n\geq 1.
 \end{cases}
 \end{equation*}
Then, $\psi_1 \geq 1$, $\psi_{j+1}/\psi_j \rightarrow 1$, and for every $q\geq k_1-1$,
\begin{equation*}
\begin{aligned}
\sum_{k=1}^\infty \left(\Psi_{k,q}-\Psi_{k+1,q}\right)\psi_{q+k+2}V_{q+k+2}
&=
\sum_{n=n(q)}^\infty \sum_{k=m_{n,q}}^{m_{n+1,q}-1} \left(\Psi_{k,q}-\Psi_{k+1,q}\right)\psi_{q+k+2}V_{q+k+2}
\\
 &\leq \sum_{n=n(q)}^\infty n 2^{-n} 
 \rightarrow 
 0
\end{aligned} 
\end{equation*}
as $q \rightarrow \infty$, where $n(q):=\max\{n\geq 1 \mid m_{n,q}=1\}$. 
\end{proof}

We now state a technical result which plays an important role in the proof of the key probabilistic estimate in the next section. Loosely speaking, it states that, at the last scale $q$ at which the bulk contribution of $\sum_{x \in F_q} \phi_x^2$ dominates the volume, the boundary behaviour on the subsequent scale can be controlled by the bulk behaviour.   

\begin{prop}\label{largest-q}
There exists a constant $C_1>0$ such that the following holds. Let $\phi \in \RR^V$. Suppose there exists $q=q(\phi)\geq 1$ such that $q$ is the largest integer for which 
\begin{equation*}
\sum_{x \in F_q} \phi_x^2 
\geq 
\psi_q V_q.
\end{equation*}
Then,
\begin{equation*}
\frac{1}{2}\sum_{x\in F_{q+1}} \sum_{y\in V\setminus F_{q+1}} \Psi(d_G(x,y))\left(\phi_x^2+\phi_y^2\right)
\leq
C_1 \sum_{x\in F_{q+1}} \phi_x^2.
\end{equation*}
\end{prop}

\begin{proof}
First observe that, for $k \geq 0$, we have
\begin{equation}\label{alt-sum}
\begin{aligned}
\sum_{k=1}^\infty \Psi_{k,q+1}\sum_{x\in F_{q+k+2}\setminus F_{q+k+1}} \phi_x^2 
&= 
\sum_{k=1}^\infty \left(\Psi_{k,q+1}-\Psi_{k+1,q+1}\right)\sum_{x\in F_{q+k+2}\setminus F_{q+2}} \phi_x^2
\\&\leq
\sum_{k=1}^\infty \left(\Psi_{k,q+1}-\Psi_{k+1,q+1}\right)\psi_{q+k+2} V_{q+k+2}.
\end{aligned}
\end{equation}

We now handle the double sum by taking cases according to the position of $x$ and $y$. Let us write
\begin{equation*}
\begin{aligned}
D
&:=
\sum_{x\in F_{q+1}} \sum_{y\in V\setminus F_{q+1}} \Psi(d_G(x,y))\left(\phi_x^2+\phi_y^2\right) 
\\ 
&=
\sum_{y\in F_{q+2}\setminus F_{q+1}} \phi_y^2 \sum_{x\in F_{q+1}} \Psi(d_G(x,y))
+ \sum_{y\in V\setminus F_{q+2}} \phi_y^2 \sum_{x\in F_{q+1}} \Psi(d_G(x,y))
\\
&+
\sum_{x\in F_{q+1}} \phi_x^2 \sum_{y\in V\setminus F_{q+1}} \Psi(d_G(x,y)) 
\\
&=:
D_1 + D_2 + D_3 .
\end{aligned}
\end{equation*}
Observe that, by Lemma~\ref{lem: psi sum},
\begin{equation*}
D_1 + D_3
\leq
\| \Psi \| \sum_{x \in F_{q+2}} \phi_x^2\leq \Vert \Psi\Vert \psi_{q+2}V_{q+2}.
\end{equation*}
Since $V_{q+2}=O(V_q)$ and $\psi_{q+2}/\psi_q\rightarrow 1$ as $q\rightarrow \infty$, it then follows that 
\begin{equation*}
D_1+D_3\leq C_3 \psi_q V_q\leq C_3\sum_{x\in F_{q+1}}\phi_x^2
\end{equation*}
for some $C_3>0$ large enough. To handle $D_2$,
we bound $\Psi(d_G(x,y))$ by $\Psi_{k,q+1}$ for $x\in F_{q+1}$ and $y\in F_{q+k+2}\setminus F_{q+k+1}$ to obtain that 
\begin{equation*}
D_2\leq V_{q+1} \sum_{k=1}^\infty \Psi_{k,q} \sum_{y\in F_{q+k+2}\setminus F_{q+k+1}} \phi_y^2.
\end{equation*}
Using \eqref{alt-sum} we obtain that
\begin{equation*}
D_2\leq V_{q+1} \sum_{k=1}^\infty \left(\Psi_{k,q+1}-\Psi_{k+1,q+1}\right) \psi_{q+k+2}V_{q+k+2}.
\end{equation*}
By \eqref{conv-0} and the fact that $V_{q+1}=O(V_q)$ we have
$
D_2\leq C_4\psi_qV_q\leq C_4 \sum_{x\in F_{q+1}}\phi_x^2
$
for some constant $C_4>0$.
The desired assertion follows readily. 
\end{proof}

\subsection{Key probabilistic regularity estimates}

We now prove the key probabilistic estimates which justify the extension of the proofs of Propositions~\ref{appendix prop: tightness}, \ref{appendix prop: dlr} and \ref{appendix prop: regular} for the case of $\ZZ^d$ to our setting. We state it for deterministic boundary conditions; the proof is exactly the same in the case of random tempered boundary conditions, although with heavier notation.  

\begin{prop} \label{appendix prop: prob}
Let $q_0>0$ and let $\Lambda = B_{n}$ for $n > t_{q_0}$ so that $F_{q_0} \subset \Lambda$. There exists $c_1=c_1(q_0) > 0$ such that, for all $\Delta \subset \Lambda$ with $F_{q_0} \subset \Delta \subset \{ x \in \Lambda : d_G(x,\partial \Lambda) > n^{\delta/2} \}$ where $\delta$ is as in Remark \ref{rem: delta}, and for any $x \in \Delta$, the following decomposition holds. For all $\eta \in \RR^V$ such that $|\eta_z| \leq \PLUS_z$, for all $z \in \Lambda^c$, for all $\phi_\Delta \in \RR^\Delta$, $\nu_{\Delta}^{(\Lambda),\eta}[\phi_\Delta]=\nu^{1,\eta}_\Delta[\phi_\Delta] +\nu^{2,\eta}_\Delta[\phi_\Delta]$ with
\begin{align}
\label{eq: nu' bound}
\nu^{1,\eta}_\Delta[\phi_\Delta]
&\leq
c_1^{|\Delta|} e^{-\frac g2 \phi_x^4} \nu_{\Delta \setminus \{x\}}^{(\Lambda)}[\phi_{\Delta \setminus \{x\}}],
\\
\label{eq: nu'' bound}
\nu^{2,\eta}_\Delta[\phi_\Delta]
&\leq
\sum_{q \geq q_0} e^{ -\sum_{z \in F_{q+1}(x)\cap \Delta} \frac g2\phi_z^4 + c_1 V_{q+1}} 
\, \nu_{\Delta \setminus F_{q+1}(x)}^{(\Lambda)}[\phi_{\Delta \setminus F_{q+1}(x)}].
\end{align}
We stress that $c_1$ uniform over all such  $\Delta$ and $\eta$, and moreover that it is continuous in the parameters of the model.
\end{prop}

\begin{proof}
Let $q_0 \in \NN$ be a sufficiently large constant to be a posteriori. Fix $x \in \Delta$. Write
\begin{equation*}
A_{\Delta, q_0}(x)
=
\{ \phi_\Delta \in \RR^\Delta : \sum_{y \in F_q(x)} \phi_y^2 \leq \psi_q V_q \textup{ for all } q \geq q_0 \},
\end{equation*}
where $F_q(x)=\gamma_x F_q$ where $\gamma_x$ is an automorphism such that $\gamma_x o=x$. Define
\begin{equation*}
\begin{aligned}
\nu^{1,\eta}_\Delta[\cdot]
&:=
\nu_\Delta^{(\Lambda),\eta}[\cdot \mathbbm{1}\{A_{\Delta,q_0}(x)\}],
\qquad
\nu^{2,\eta}_\Delta
:=
\nu^{(\Lambda)}_\Delta - \nu^{1,\eta}_\Delta.
\end{aligned}
\end{equation*}
On the event $A_{\Delta,q_0}(x)$, we are able to control the bulk behaviour (centred at $x$) of $\phi_y^2$ in spatial average; on the complement we have control via the oscillation and decay bounds.

We first establish \eqref{eq: nu' bound}. For $\Lambda = \Lambda_1 \sqcup \Lambda_2$, for all $\varphi'_\Lambda,\phi_\Lambda \in \RR^\Lambda$, define
\begin{equation*}
W_\Lambda(\phi'_{\Lambda_1},\phi_{\Lambda_2})
:=
U^\eta_{\Lambda_1}(\phi'_{\Lambda_1}) + U^\eta_{\Lambda_2}(\phi_{\Lambda_2}) - U^\eta_\Lambda(\varphi_{\Lambda_1}'\sqcup\phi_{\Lambda_2}).
\end{equation*}
In particular,
\begin{equation*}
W_\Lambda(\varphi_{\Lambda_1},\varphi_{\Lambda_2})
=
\Big\{\sum_{x,y \in \Lambda_1} + \sum_{x,y \in \Lambda_2}- \sum_{x,y \in \Lambda } \Big\}  J_{x,y}\phi_x\phi_y
=
-\sum_{x \in \Lambda_1, y \in \Lambda_2} J_{x,y} \phi_x \phi_y. 
\end{equation*}
In particular, observe that for our fixed $x\in \Delta$,
\begin{equation*}
W_\Lambda(\phi_x,\phi_{\Lambda \setminus \{x\}})
=
W_{\Lambda}(\phi_x', \phi_{\Lambda\setminus  \{x\}}) - \sum_{y \in \Lambda \setminus \{x\}} J_{x,y} (\phi_x - \phi_x') \phi_y.
\end{equation*}
Thus, by Young's inequality,
\begin{equation*}
\left| \sum_{y \in \Lambda \setminus \{x\}} J_{x,y} (\phi_x - \phi_x') \phi_y \right|
\leq
\frac{\sum_{y\in \Lambda\setminus \{x\}} J_{x,y}}{2} \Big( \phi_x^2 + (\phi'_x)^2 \Big) + T_1
\end{equation*}
where
\begin{equation*}
T_1
=
\sum_{y \in \Lambda \setminus \lbrace x\rbrace} J_{x,y} \phi_y^2
\leq
\sum_{y \in F_{q_0}(x)} J_{x,y} \phi_y^2 + \sum_{y \in \Lambda \setminus F_{q_0}(x)} J_{x,y} \phi_y^2
:=
T_2 + T_3.
\end{equation*}
By the definition of $A_{\Delta,q_0}(x)$,
$
T_2
\leq
|J|_\infty \psi_{q_0} V_{q_0}.
$

In order to handle $T_3$, we observe that
$
\sum_{y \in F_q(x)} \phi_y^2
\leq 
\psi_{q} V_{q}.
$
for every $q\geq q_0$. Moreover,
\begin{equation*}
\begin{aligned}
\sum_{y \in V\setminus F_{q_0}(x)}  \Psi(d_G(x,y)) \phi_y^2
&\leq 
\sum_{k\geq 1} \Psi_{k,q_0-1}\sum_{y\in F_{q_0+k}(x)\setminus F_{q_0-1+k}(x)} \phi_y^2\\
&=
\sum_{k\geq 1} (\Psi_{k,q_0-1}-\Psi_{k+1,q_0-1})\sum_{y\in F_{q_0+k}(x)\setminus F_{q_0}(x)} \phi_y^2\\
&\leq
\sum_{k\geq 1} (\Psi_{k,q_0-1}-\Psi_{k+1,q_0-1})\psi_{q_0+k}V_{q_0+k},
\end{aligned}
\end{equation*}
where in the first inequality we have bounded $\Psi(d_G(x,y))$ by $\Psi_{k,q_0-1}$ for $y\in F_{q_0+k}(x)\setminus F_{q_0-1+k}(x)$ (we use that $x\in F_{q_0-1}(x)$), then in the second inequality we use a summation by parts. Combining the above estimates together and using Lemma \ref{lem: psi def}, we have established that
\begin{equation*}
T_3
\leq
\sum_{k\geq 1} (\Psi_{k,q_0-1}-\Psi_{k+1,q_0-1})\psi_{q_0+k}V_{q_0+k}
\leq C_1.
\end{equation*}
All together we have established
\begin{equation*}
|W_\Lambda(\phi_x,\phi_{\Lambda \setminus \{x\}})
-
W_{\Lambda}(\phi_x', \phi_{\Lambda\setminus  \{x\}})|
\leq
\frac{\|\Psi\|}2 \big(\phi_x^2 + (\phi_x')^2\big) + |J|_\infty \psi_{q_0} V_{q_0} + C_1.
\end{equation*}

We now bound $\nu^{1,\eta}_\Delta[\varphi_\Delta]$. First observe that, for any $\phi_x' \in [-A,A]$, where $A>0$,
\begin{equation*} 
\begin{aligned}
\nu^{1,\eta}_\Delta[\phi_\Delta]
&=
Z_\Lambda^{-1} \int e^{-U^\eta_{\{x\}}(\phi_x) - U^\eta_{\Lambda \setminus \{x\}}(\phi_{\Lambda \setminus \{x\}}) + W(\phi_x, \phi_{\Lambda \setminus \{x\}})} \textup{d}\phi_{\Lambda \setminus \Delta}
\\
&=
e^{-U^\eta_{\{x\}}(\phi_x)}Z_\Lambda^{-1} \int e^{- U^\eta_{\Lambda \setminus\{x\}}(\phi_{\Lambda \setminus \{x\}}) + W(\phi_x', \phi_{\Lambda \setminus \{x\}}) - W(\phi_x', \phi_{\Lambda \setminus \{x\}}) + W(\phi_x, \phi_{\Lambda \setminus \{x\}}) } \textup{d}\phi_{\Lambda \setminus \Delta}
\\
&\lesssim
e^{-U^\eta_{\{x\}}(\phi_x)} e^{\frac{\sum_{y\in V} J_{0,y}}{2} \Big( \phi_x^2 + (\phi'_x)^2 \Big)} \int e^{-U^\eta_{\Lambda \setminus \{x\}}(\phi_{\Lambda \setminus \{x\}}) + W(\phi_x', \phi_{\Lambda \setminus \{x\}})} \textup{d}\phi_{\Lambda \setminus \Delta}
\\
&\lesssim
e^{-\frac g2 \phi_x^4}  \int e^{-U^\eta_{\Lambda}(\phi_{\Lambda\setminus \{x\}} \sqcup \phi'_x) + U^\eta_{\{x\}}(\phi_x') }\textup{d}\phi_{\Lambda \setminus \Delta}
\\
&\lesssim
e^{-\frac g2 \phi_x^4}  \int e^{-U^\eta_{\Lambda}(\phi_{\Lambda\setminus \{x\}} \sqcup \phi'_x)}\textup{d}\phi_{\Lambda \setminus \Delta}.
\end{aligned}
\end{equation*}
Above, the first two lines are straightforward manipulations; the third line follows from the preceding bounds; the fourth line follows from Young's inequality inside the exponentials, the fact that $|\phi_x'| \leq A$, and estimates similar to \eqref{eq: o(1)} to handle the boundary terms in $U^\eta_{\{x\}}(\phi_x)$; and the final line is by using that $|\phi_x'| \leq A$. Thus, taking the integral mean, 
\begin{equation*}
\nu^{1,\eta}_\Delta[\phi_\Delta]
\leq
c_1^{|\Delta|} e^{-\frac{g}{2}\phi_x^4} \int_{-A}^A \int e^{-U^\eta_{\Lambda}(\phi_{\Lambda\setminus \{x\}}\sqcup \varphi_x') }\textup{d}\phi_{\Lambda \setminus \Delta}\textup{d}\varphi_x'
\leq
c_1^{|\Delta|} e^{-\frac{g}{2}\phi_x^4} \int e^{-U^\eta_{\Lambda}(\phi_\Lambda) }\textup{d}\phi_{\Lambda \setminus \Delta \cup \{x\}}.
\end{equation*}
This last bound establishes \eqref{eq: nu' bound}.

We now turn to \eqref{eq: nu'' bound}. Write
\begin{equation*}
B_{\Delta,q}(x)
=
\{ \phi_\Delta \in \RR^\Delta : \sum_{y \in F_q(x)} \phi_y^2 > \psi_q V_q, \:  \sum_{y \in F_k(x)} \phi_y^2 \leq\psi_k V_k, \: \forall k>q\}.
\end{equation*}
Then, by a union bound and re-ordering terms
\begin{equation*}
\begin{aligned}
\nu^{2,\eta}_\Delta[\phi_\Delta]
&\leq
\sum_{q \geq q_0} \frac{\mathbbm{1}_{\{\phi_\Delta \in B_{\Delta,q}(x)\}}}{Z_\Lambda} \int e^{-U^\eta_\Lambda(\phi_\Lambda)} \textup{d}\phi_{\Lambda \setminus \Delta} 
\\
&=
\sum_{q \geq q_0} \frac{\mathbbm{1}_{\{\phi_\Delta \in B_{\Delta,q}(x)\}}}{Z_\Lambda} \int e^{-U^\eta_{F_{q+1}(x)\cap \Delta}(\phi_{F_{q+1}(x)\cap \Delta}) - U^\eta_{\Lambda \setminus( F_{q+1}(x)\cap\Delta)}(\phi_{\Lambda \setminus (F_{q+1}(x) \cap \Delta)})}
\\
&\qquad\qquad\qquad\qquad\qquad
\times e^{ W_\Lambda(\phi_{F_{q+1}(x)\cap \Delta}, \phi_{\Lambda \setminus (F_{q+1}(x)\cap\Delta)} )} \textup{d}\phi_{\Lambda \setminus \Delta} 
\\
&\leq
\sum_{q \geq q_0} \frac{\mathbbm{1}_{\{\phi_\Delta \in B_{\Delta,q}\}}}{Z_\Lambda} e^{-\sum_{y \in F_{q+1}(x) \cap \Delta} \frac {3g}4 \phi_y^4 + CV_{q+1}} 
\\
&\qquad \qquad \times
\int e^{-U^\eta_{\Lambda \setminus (F_{q+1}(x)\cap\Delta)}(\phi_{\Lambda \setminus (F_{q+1}(x) \cap \Delta)})+  W_\Lambda(\phi_{F_{q+1}(x)\cap \Delta}, \phi_{\Lambda \setminus (F_{q+1}(x) \cap \Delta)} )}  \textup{d}\phi_{\Lambda \setminus \Delta} 
\end{aligned}
\end{equation*}
where the third line is uses similar reasoning as above to handle $U^\eta_{F_{q+1}(x)\cap \Delta}(\phi_{F_{q+1}(x)\cap \Delta})$. 

Let $\phi_{F_{q+1}(x)\cap \Delta}' \in \RR^{F_{q+1}(x)\cap \Delta}$ be such that $|\phi'_z| \leq A$ for all $z\in F_{q+1}(z)\cap \Delta$. Observe that we can write 
\begin{equation*}
 W_\Lambda(\phi_{F_{q+1}(x)\cap \Delta}, \phi_{\Lambda\setminus (F_{q+1}(x)\cap \Delta)} )
 =
  W(\phi_{F_{q+1}(x)\cap \Delta}', \phi_{\Lambda \setminus (F_{q+1}(x) \cap \Delta)} ) + T_4,
\end{equation*}
where
\begin{equation}\label{eq: t4}
|T_4|
\leq
\sum_{\substack{z \in F_{q+1}(x) \cap \Delta \\ y \in \Lambda \setminus (F_{q+1}(x)\cap \Delta)}} \Psi(d_G(z,y)) \Big(\frac 12 \phi_z^2 + \frac 12 (\phi_z')^2 + \phi_y^2 \Big)
\end{equation}
On the event $B_{\Delta,q}(x)$, we can apply Proposition \ref{largest-q} and $\Gamma$-invariance to obtain
\begin{equation*}
\eqref{eq: t4}
\leq
C_2 \sum_{z \in F_{q+1} \cap \Delta} \phi_z^2  + C_3 \sum_{z \in F_{q+1}(x) \cap \Delta} (\phi_z')^2. 
\end{equation*}

Thus, by Young's inequality, using the boundedness of $\phi_x'$, we obtain 
\begin{equation*}
\begin{aligned}
\nu^{2,\eta}_\Delta[\phi_\Delta]
&\leq
\sum_{q \geq q_0} \frac{1}{Z_\Lambda} e^{-\sum_{y \in F_{q+1}(x) \cap \Delta}\frac{g}{2} \phi_y^4 + C_4 V_{q+1}} 
\\
&\qquad \qquad \times
\int e^{-U^\eta_{\Lambda \setminus (F_{q+1}(x)\cap \Delta)}(\phi_{\Lambda \setminus (F_{q+1}(x)\cap \Delta)})+  W_\Lambda(\phi'_{F_{q+1}(x)\cap \Delta}, \phi_{\Lambda \setminus (F_{q+1}(x) \cap \Delta)} )}  \textup{d}\phi_{\Lambda \setminus \Delta} 
\\
&=
\sum_{q \geq q_0} \frac{1}{Z_\Lambda} e^{-\sum_{y \in F_{q+1}(x) \cap \Delta}\frac{g}{2} \phi_y^4 + C_4 V_{q+1}} 
\\
&\qquad \qquad \times
\int e^{-U^\eta_{\Lambda }(\phi_{\Lambda \setminus (F_{q+1}(x) \cap \Delta)} \sqcup \phi'_{F_{q+1}(x)\cap\Delta})  + U^\eta_{F_{q+1}(x)\cap \Delta}(\phi'_{F_{q+1}(x)\cap \Delta})} \textup{d}\phi_{\Lambda \setminus \Delta}. 
\end{aligned}
\end{equation*}
The desired estimate follows by using the boundedness of $|\phi'_{F_{q+1}(x)\cap\Delta}|$ and taking the integral mean with respect to $\phi_{F_{q+1}\cap \Lambda}'$, similarly as for the estimate on $\nu^{1,\eta}_\Delta$.
\end{proof}

\begin{cor} \label{app cor: finite vol reg}
There exists $c_1>0$ such that the following holds. Let $n \geq 1$ and let $\Lambda \subset V$ be finite such that $\Lambda \supset B_n$. Let $\Delta \subset \Lambda$ such that $\Delta \subset B_{m}$, where $m = {n-n^{\frac \delta 2}}$ and $\delta$ is as in Remark \textup{\ref{rem: delta}}. Then, for all $\eta\in \RR^V$ such that $|\eta_z| \leq \sqrt{\log d_G (o,\partial \Lambda)}$ for all $z \in \Lambda^c$,
\begin{equation*}
\nu_\Delta^{(\Lambda), \eta}[\phi_\Delta]
\leq
e^{-\sum_{x \in \Delta} \big(\frac g2 \phi_x^4 - c_1\big)}.
\end{equation*}
\end{cor}

\begin{proof}
The result is based on an induction argument that, given Proposition \ref{appendix prop: prob}, is identical to the proof of \cite[Theorem 2.2]{R70}. We omit the details.
\end{proof}

\begin{rem} \label{rem: boundary saviour}
The result of Corollary \textup{\ref{app cor: finite vol reg}} can be extended up to the boundary with the cost of setting $c_1 = O(\log d_G(o,\partial\Lambda))$. In particular, by using an exponential Chebyschev argument, one can bound moments; for instance for some $C>0$ we have
\begin{equation*}
\sup_{\Delta \subset \Lambda}\sup_{x \in \Delta}\nu^{(\Lambda),\eta}_{\Delta}[\phi_x^2]
\leq
\sqrt{C\log d_G(o,\partial\Lambda)}.
\end{equation*}
\end{rem}

\subsection{Generalisation to tempered amenable graphs} \label{app: tempered amenable}

The results of this section can be generalised to (supposedly) a larger class of amenable graphs that we now introduce. Recall that a connected, locally finite, and vertex-transitive graph $G=(V,E)$ is called amenable if 
\begin{equation}
\inf_{A \subset V, |A|<\infty} \frac{|\partial A|}{|A|}
=
0.
\end{equation}
Equivalently, $G$ is amenable if and only if it admits a F{\o}lner sequence $(F_n)_{n \in \NN} \subset V$, i.e.\ a sequence $(F_n)_{n \in \NN}$ such that $\frac{|\partial F_n|}{|F_n|} \rightarrow 0$ as $n\rightarrow \infty$. We call $(F_n)_{n \in \NN}$ a F{\o}lner exhaustion if it is a F{\o}lner sequence and, in addition, $F_n \uparrow V$ as $n \rightarrow \infty$.

We say that a connected, locally finite, vertex-transitive graph $G$ is a \textit{tempered amenable graph} if it admits a function $f:\NN\rightarrow (0,\infty)$ such that $f(x+y)\leq f(x)f(y)$ and a F{\o}lner exhaustion $(F_n)_{n \in \NN}$ that satisfy
\begin{align}
\begin{split}
\sum_{x\in V\setminus\{o\}}\frac{1}{f(d_G(o,x))}<\infty     
\end{split}
\\
\begin{split}\label{conv_1}
\limsup_{n\to\infty}\frac{|F_{n+1}|}{|F_n|}<\infty
\end{split}
\\
\begin{split}
F_n\cup \partial F_n \subset F_{n+1}    
\end{split}
\\
\begin{split}
\lim_{n\to\infty}\frac{(|F_{n+1}|-|F_n|)\max_{x\in F_{n+1}\setminus F_n}\{\sqrt[4]{\log f(d_G(o,x))}\} }{|F_{n+1}|}=0.   
\end{split}
\end{align}
We restrict to kernels $\Psi:\NN\rightarrow (0,\infty)$ which are decreasing and satisfy
\begin{equation}
\lim_{q\to\infty}\sum_{k=1}^\infty \left(\Psi_{k,q}-\Psi_{k+1,q}\right) |F_{q+k+1}|
=
0.
\end{equation}
Moreover, we define
\begin{equation}
\Xi(M)
=
\{ \phi : \exists V_\phi \subset V, |\phi_x| \leq \sqrt[4]{M\log f(d_G(o,x))} \, \forall x \in V_\phi^c\}.
\end{equation}
Then, one can check that the proofs of Propositions \ref{appendix prop: tightness}--\ref{appendix prop: regular} extend to this setting.

\begin{rem}
Any connected, vertex-transitive graph of polynomial growth is a tempered amenable graph. It would be interesting to find amenable graphs of intermediate or exponential growth that satisfy our assumptions; we do not know of any such examples as of yet.
\end{rem}

\paragraph{Acknowledgements.}
Above all we thank Hugo Duminil-Copin for proposing the problem and encouraging all of us to collaborate with each other. We thank Ajay Chandra and Aran Raoufi for useful discussion. We thank Lucas D'Alimonte, Dmitrii Krachun, R\'emy Mahfouf, and two referees for useful comments. TSG thanks Sainee Sharma for inspiration.

\paragraph{Funding.}
TSG acknowledges the support of the Simons Foundation, Grant 898948, HDC. CP and RP were supported by the Swiss National Science Foundation and the NCCR SwissMAP. 
FS has received funding from the European Research Council (ERC) under the European Union’s Horizon 2020 research and innovation program (grant agreement No 851565).


\bibliographystyle{alpha}
\bibliography{surveybis.bib}







\end{document}